\documentclass[11pt]{amsart}
\usepackage{amssymb}
\usepackage{mathtools}
\usepackage{color}

\newtheorem{theo}{Theorem} 

\newtheorem{lemma}{Lemma}[section]
\newtheorem{prop}[lemma]{Proposition}
\newtheorem{corol}[lemma]{Corollary}

\newtheorem{claim}[lemma]{Claim}

\theoremstyle{remark}
\newtheorem{remark}[lemma]{Remark}

\newtheorem{notation}[lemma]{Notation}

\theoremstyle{definition}
\newtheorem{defi}[lemma]{Definition}

\newcommand{\profiles}{\left( U_L^j,\left\{\lambda_{j,n},x_{j,n},t_{j,n}\right\}_n \right)_{j\geq 1}}
\newcommand{\tprofiles}{\left( \tU_L^j,\left\{\tlambda_{j,n},\tx_{j,n},\ttt_{j,n}\right\}_n \right)_{j\geq 1}}
\newcommand{\tprofilesp}{\left( \tU_L^j,\left\{\tlambda_{j,p},\tx_{j,p},\ttt_{j,p}\right\}_p \right)_{j\geq 1}}

\newcommand{\NN}{\mathbb{N}}
\newcommand{\RR}{\mathbb{R}}

\newcommand{\eps}{\varepsilon}

\newcommand{\EEE}{\mathcal{E}}

\newcommand{\JJJ}{\mathcal{J}}
\newcommand{\KKK}{\mathcal{K}}

\newcommand{\PPP}{\mathcal{P}}
\newcommand{\SSS}{\mathcal{S}}

\newcommand{\tV}{\widetilde{V}}

\newcommand{\ve}{\mathbf{e}}
\newcommand{\vu}{\vec{u}}
\newcommand{\vw}{\vec{w}}

\newcommand{\vU}{\vec{U}}
\newcommand{\vV}{\vec{V}}
\newcommand{\vW}{\vec{W}}
\newcommand{\vS}{\vec{S}}
\newcommand{\vQ}{\vec{Q}}
\newcommand{\vv}{\vec{v}}
\newcommand{\vell}{\boldsymbol{\ell}}

\newcommand{\ttt}{\tilde{t}}

\newcommand{\tlambda}{\widetilde{\lambda}}

\newcommand{\tZ}{\widetilde{Z}}
\newcommand{\tB}{\widetilde{B}}
\newcommand{\tG}{\widetilde{G}}

\newcommand{\tJ}{\widetilde{J}}

\newcommand{\tw}{\widetilde{w}}

\newcommand{\ttau}{\widetilde{\tau}}
\newcommand{\tnu}{\widetilde{\nu}}

\newcommand{\tU}{\widetilde{U}}

\newcommand{\tu}{\widetilde{u}}

\newcommand{\SN}{\frac{2(N+1)}{N-2}}

\newcommand{\tz}{\widetilde{z}}
\newcommand{\tx}{\widetilde{x}}

\newcommand{\loc}{\rm loc}
\newcommand{\hdot}{\dot{H}^1}

\newcommand{\EMPH}[1]{\medskip\noindent\textit{#1}.}

\DeclareMathOperator{\supp}{supp}

\newcommand{\ds}{\displaystyle}

\numberwithin{equation}{section} 

 \title[Profiles for bounded solutions of dispersive equations]{Profiles for bounded solutions of
dispersive equations, with applications to energy-critical wave and
Schr\"odinger equations}

\author[T.~Duyckaerts]{Thomas Duyckaerts$^1$}
\author[C.~Kenig]{Carlos Kenig$^2$}
\author[F.~Merle]{Frank Merle$^3$}
\thanks{$^1$LAGA, Universit\'e Paris 13 (UMR 7539). Partially supported by ERC Grant Dispeq, ERC Avanced Grant  no. 291214, BLOWDISOL and ANR Grant SchEq}
 \thanks{$^2$University of Chicago. Partially supported by NSF Grant DMS-0968472 and by NSF Grant DMS-1265249}
 \thanks{$^3$Cergy-Pontoise (UMR 8088), IHES. Partially supported by ERC Advanced Grant  no. 291214, BLOWDISOL}
\begin{document}

\maketitle

\textbf{Dedicado a nuestro amigo Gustavo Ponce}

\tableofcontents

\section{Introduction}
This article presents a new compactness argument to  describe the
asymptotics of bounded solutions of focusing nonlinear dispersive
equations. We will mainly consider the energy-critical wave equation in space dimension $N\in\{3,4,5\}$,
for which our results are more complete:
\begin{equation}
\label{CP}
\left\{ 
\begin{gathered}
\partial_t^2 u -\Delta u-|u|^{\frac{4}{N-2}}u=0,\quad (t,x)\in I\times \RR^N\\
u_{\restriction t=0}=u_0\in \hdot,\quad \partial_t u_{\restriction t=0}=u_1\in L^2,
\end{gathered}\right.
\end{equation}
where $I$ is an interval ($0\in I$), $u$ is real-valued, $\hdot:=\hdot(\RR^N)$, and $L^2:=L^2(\RR^N)$. 

We will also give a consequence of our method for solutions of the
energy-critical nonlinear Schr\"odinger equation (NLS):
\begin{equation}
\label{NLS}
\left\{
\begin{gathered}
i\partial_tu+\Delta u=-|u|^{\frac{4}{N-2}}u,\\
u_{\restriction t=0}=u_0\in \hdot.
\end{gathered}
\right.
\end{equation}

The equation \eqref{CP} is locally well-posed in $\hdot\times L^2$. If $u$ is a solution, we will denote by $(T_-(u),T_+(u))$ its maximal interval of existence. On $(T_-(u),T_+(u))$, the following two quantities are conserved:
$$ E(u(t),\partial_tu(t))=\frac 12\int |\nabla u|^2dx+\frac{1}{2}\int (\partial_tu)^2dx-\frac{N-2}{2N}\int |u|^{\frac{2N}{N-2}}dx$$
(the energy) and 
$$ P(u,\partial_t u(t))=\int \partial_t u\,\nabla u\,dx.$$
(the momentum). We will denote this quantities by $E[u]$ and $P[u]$ respectively.

In this paper we are interested in solutions of \eqref{CP} that are bounded in $\hdot\times L^2$ for positive time. Examples of these are given by stationary solutions of the equation, i.e. $Q\in \hdot(\RR^N)$ such that
\begin{equation}
 \label{stationary} 
 -\Delta Q=|Q|^{\frac{4}{N-2}}Q.
\end{equation}
We denote by $\Sigma$ the set of nonzero, $\hdot$ solutions of \eqref{stationary}. $\Sigma$ contains 
$$ W=\frac{1}{\left(1+\frac{|x|^2}{N(N-2)}\right)^{\frac{N}{2}-1}}$$
which is the unique (up to scaling and sign-change) radial stationary solution of \eqref{CP}. There also exist elements of $\Sigma$ without spherical symmetry and with arbitrarily large energy: see the article W.~Y.~Ding \cite{Ding86}, for constructions of solutions through variational arguments, and the recent works of 
M.~del Pino, M.~Musso, F.~Pacard and A.~Pistoia (\cite{dPMPP11}, \cite{dPMPP13}) for more explicit constructions.  

Other nonradial bounded solutions of \eqref{CP} are given by travelling waves, which are Lorentz transforms of stationary solution. 
If $Q\in \Sigma$, $\vell\in \RR^N$ with $\ell=|\vell|<1$ (here and in the sequel $|\vell|$ denotes the Euclidean norm of $\vell$), then
\begin{equation}
 \label{defQl}
 Q_{\vell}(t,x)=Q\left(\left(-\frac{t}{\sqrt{1-\ell^2}}+\frac{1}{\ell^2} \left(\frac{1}{\sqrt{1-\ell^2}}-1\right)\vell\cdot x\right)\vell+x\right)
\end{equation} 
is a bounded solution of \eqref{CP} such that
\begin{equation}
\label{travelling_wave}
Q_{\vell}(t,x)=Q_{\vell}(0,x-t\vell).
\end{equation} 
Note that, denoting a vector of $\RR^N$ by $x=(x_1,x')$ where $x_1\in \RR$, $x'\in \RR^{N-1}$ and assuming (after a rotation) that $\vell=\ell(1,0,\ldots,0)$, we can write \eqref{defQl} as 
\begin{equation*}
 Q_{\vell}(t,x)=Q\left(\frac{x_1-t\ell}{\sqrt{1-\ell^2}},x'\right).
\end{equation*} 
In this work, we prove the following result:
\begin{theo}
\label{T:profiles}
Let $u$ be a solution of \eqref{CP} and assume that $u$ does not scatter forward in time and 
 \begin{equation}
  \label{bound_u}
  \sup_{t\in [0,T_+(u))}\|(u(t),\partial_tu(t))\|_{\hdot\times L^2}<\infty.
 \end{equation} 
 Then there exist sequences $\{t_n\}_n$ in $[0,T_+(u))$, $\{\lambda_n\}_n$ in $(0,+\infty)$, an integer $J\geq 1$, and, for all $j\in \{1,\ldots,J\}$, $Q^j\in \Sigma$, $\vell_j\in \RR^N$ with $\left|\vell_j\right|<1$, and a sequence $\{x_{j,n}\}_n$ in $\RR^N$, such that 
 \begin{equation*}
\lim_{n\to\infty} t_n=T_+(u),\quad 1\leq j<k\leq J\Longrightarrow \lim_{n\to\infty}\frac{x_{j,n}-x_{k,n}}{\lambda_{n}}=+\infty,
\end{equation*}
and
\begin{itemize}
 \item for all $T>0$, 
\begin{multline}
\label{CV_THM1}
\lambda_nT+t_n<T_+(u)\text{ for large }n \text{ and }\\
\lim_{n\to\infty}\int_0^T\int_{\RR^N} \bigg|\lambda_{n}^{\frac{N}{2}-1}u\left(t_n+\lambda_{n}t,\lambda_{n}x\right)-\sum_{j=1}^{J}Q_{\vell_j}^{j}\left(t,x-\frac{x_{j,n}}{\lambda_{n}}\right)\bigg|^{\frac{2(N+1)}{N-2}}\,dx\,dt=0.
\end{multline} 
\item for all $R>0$, for all $j\in \{1,\ldots,J\}$, 
\begin{equation}
\label{CV_THM2}
 \lim_{n\to\infty} \int_{|x|\leq R} \bigg|\lambda_{n}^{\frac{N}{2}}\nabla_{t,x}u\left(t_n,x_{j,n}+\lambda_{n}x\right)-\nabla_{t,x} Q^{j}_{\vell_j}(0,x)\bigg|^2\,dx=0,
\end{equation}
where $\nabla_{t,x}u=\left(\partial_tu,\partial_{x_1}u,\ldots,\partial_{x_N}u\right)$.
\end{itemize}
\end{theo}
\begin{remark}
 In the theorem, $u$ can be global or blow up in finite time (type II blow-up). Examples of such solutions are known in both cases (see e.g. \cite{KrScTa09}, \cite{DoKr13}). Note however that in all known examples, there is only one stationary profile and (apart from the trivial cases $u=Q_{\vell}$, $Q\in \Sigma$) this profile is always equal to the stationary solution $W$ defined above.
\end{remark}
\begin{remark}
 A large part of the proof of Theorem \ref{T:profiles} is not specific to equation \eqref{CP} and works for general nonlinear dispersive equations. See Section \ref{S:NLS} for an application to the Schr\"odinger equation \eqref{NLS}.
\end{remark}
\begin{remark}
The conclusion of the theorem implies that for $j=1,\ldots, J$,
  \begin{equation*}
\left(\lambda_n^{\frac{N}{2}-1}u\left(t_n,\lambda_n\cdot+x_{j,n}\right),\lambda_n^{\frac{N}{2}}\partial_t u\left(t_n,\lambda_n\cdot+x_{j,n}\right)\right)\xrightharpoonup[n\to \infty]{} (Q_{\vell_j}^j(0),\partial_t Q_{\vell_j}^j(0))
\end{equation*} 
weakly in $\hdot \times L^2$. 
\end{remark}
\begin{remark}
 A result of the type of Theorem \ref{T:profiles} was first proved for wave-maps in a radial setting by D.~Christodoulou, and A.S.~Tahvildar-Zadeh \cite{ChTZ93}, and M.~Struwe \cite{Struwe03b}. The analoguous result for more general wave maps was proved by J.~Sterbenz and D.~Tataru \cite{StTa10a,StTa10b}. A crucial ingredient of the proofs of these results is the monotonicity of the wave-maps energy flux, which is not available for focusing wave equations as \eqref{CP}. 

Note that \eqref{CV_THM2} implies by finite speed of propagation:
\begin{equation*}
\forall T>0,\; \forall R>0,\quad 
  \lim_{n\to\infty} \int_0^{T}\int_{|x|\leq R} \bigg|\lambda_{n}^{\frac{N}{2}}\nabla_{t,x}u\left(t_n+\lambda_nt,x_{j,n}+\lambda_{n}x\right)-\nabla_{t,x} Q^{j}_{\vell_j}(t,x)\bigg|^2\,dx\,dt=0,
\end{equation*} 
which is the exact analog, for equation \eqref{CP} of the $H^1_{\loc}$ convergence result in \cite[Theorem 1.3 (A)]{StTa10b}. 
\end{remark}
\begin{remark}
The norm $S(I)=L^{\frac{2(N+1)}{N-2}}(I\times \RR^N)$ appearing in \eqref{CV_THM1} is a scale-invariant Strichartz norm adapted to equation \eqref{CP}. It follows from the local well-posedness theory that for all compact intervals $I\subset (T_-(u),T_+(u))$, $\|u\|_{S(I)}$ is finite. Furthermore if $\|u\|_{S([0,T_+(u)))}<\infty$, then $T_+(u)=+\infty$ and $u$ scatters forward in time, in $\hdot\times L^2$ to a solution of the linear wave equation. We refer to \cite{KeMe08} for details.
\end{remark}

\begin{remark}
 We assumed $N\leq 5$ to simplify the exposition. The proof can be adapted to the case $N\geq 6$ using the results in \cite{BuCzLiPaZh13}.
\end{remark}

 Theorem \ref{T:profiles} is the first step of the proof of a full decomposition:
\begin{equation}
\label{full_decompo}
u(t_n,x)=\sum_{j=1}^J \frac{1}{\lambda_{j,n}^{\frac{N}{2}-1}}Q^j_{\vell_j}\left(0,\frac{x-x_{j,n}}{\lambda_{j,n}}\right)+v(t_n,x)+o_n(1), 
\end{equation} 
where $v$ is a radiation term (a solution of the linear wave equation), and $o_n(1)$ goes to $0$ in the energy space, for a sequence of times $t_n\to T_+(u)$. We refer to \cite{DuKeMe12b} for the proof of this decomposition in the radial case. See also \cite{Cote13P} for wave maps. 

Our final goal is to prove that any bounded solution $u$ of \eqref{CP} can be written, as $t\to T_+(u)$, as a sum of decoupled solitary waves and a dispersive term (\emph{soliton resolution conjecture for equation \eqref{CP}}). This conjecture was settled in \cite{DuKeMe13} for radial initial data: in this case, there is no space translation and Lorentz invariance, and the only possible profile is (up to sign change and scaling) the explicit solution $W$ defined above. In the nonradial case, the conjecture was proved with an additional smallness assumption (and only for finite-time blow-up solutions) in \cite{DuKeMe12}.

The proof of Theorem \ref{T:profiles} uses the profile decomposition of Bahouri and G\'erard \cite{BaGe99}. The solutions $Q_{\ell_j}^j$ with scaling parameters $\{\lambda_n\}_n$ are, in a certain sense, the most concentrated profiles in a profile decomposition of $\{(u(t_n),\partial_tu(t_n))\}_n$. More precisely, they are the first nonlinear profiles with respect to a total preorder relation on the profiles adapted to equation \eqref{CP} (see Notation \ref{N:preorder} for the precise definition).  
If the full decomposition \eqref{full_decompo} also holds along the time sequence $\{t_n\}_n$, this means exactly that the sequence of scaling parameters $\{\lambda_n\}_n$ appearing in Theorem \ref{T:profiles} is of the same order than the sequence of smallest scaling parameters $\{\lambda_{j,n}\}_n$ of \eqref{full_decompo}, corresponding to the profiles with the fastest rate of concentration. Unfortunately, Theorem \ref{T:profiles} does not give any information on profiles with a slower rate of concentration.

\bigskip

The proof of Theorem \ref{T:profiles} is based on the notion of solutions with the compactness property, introduced in \cite{GlMe95P} for nonlinear Schr\"odinger equations (NLS) and \cite{MaMe00} for Korteweg-de Vries equation. See also \cite{KeMe06} for energy-critical NLS, \cite{Keraani06},  \cite{TaViZh08} and \cite{Dodson12} for mass-critical NLS (and \cite{MeRa04} for a stronger notion of nondispersive solutions)), \cite{KeMe08} for the energy critical wave equation \eqref{CP}.
\begin{defi}
\label{D:compact}
We say that a solution $u$ of \eqref{CP} has the \emph{compactness property} when there exists $\lambda(t)>0$, $x(t)\in \RR^N$, defined for $t\in (T_-(u),T_+(u))$ such that:
$$K=\left\{ \left( \lambda^{\frac{N}{2}-1}(t)u\left(t,\lambda(t)\cdot+x(t)\right), \lambda^{\frac N2}(t)\partial_t u\left(t,\lambda(t)\cdot+x(t)\right)\right),\; t\in (T_-(u),T_+(u))\right\}$$
has compact closure in $\hdot\times L^2$.
 \end{defi}
 Note that the stationary solutions and travelling waves defined above are solutions of \eqref{CP} with the compactness property. 
Theorem \ref{T:profiles} uses the following result, that states that any solution with the compactness property converges strongly to one of these solutions along a sequence of times.
 \begin{prop}
\label{P:compact}
 Let $u$ be a nonzero solution with the compactness property, with maximal time of existence $(T_-,T_+)$. Then
\begin{enumerate}
 \item \label{I:half_global} $T_-=-\infty$ or $T_+=+\infty$.
\item \label{I:theo_virial} there exist two sequences $\{t_n^{\pm}\}_n$ in $(T_-,T_+)$, two elements $Q^{\pm}$ of $\Sigma$ and one vector $\vell$ in the open unit ball of $\RR^N$ such that $\lim_{n\to +\infty}t_n^{\pm}=T_{\pm}$ and
\begin{multline}
 \label{theo_CV}
\lim_{n\to\infty}\left\|\lambda^{\frac{N}{2}-1}\left(t_n^{\pm}\right)u\left(t_n^{\pm},\lambda\left(t_n^{\pm}\right)\cdot +x\left(t_n^{\pm}\right)\right)-Q_{\vell}^{\pm}\left(0\right)\right\|_{\hdot}\\
+\left\|\lambda^{\frac{N}{2}}\left(t_n^{\pm}\right)\partial_t u\left(t_n^{\pm},\lambda\left(t_n^{\pm}\right)\cdot +x\left(t_n^{\pm}\right)\right)-\partial_tQ_{\vell}^{\pm}\left(0\right)\right\|_{L^2}=0.
\end{multline}
\end{enumerate}
\end{prop}
\begin{remark}
The vector $\vell$ is given in terms of the momentum and energy of $u$. Indeed, $E[u]>0$ and $\vell=-P[u]/E[u]$ (see Propositions \ref{P:D1} and \ref{P:D1bis} below).
\end{remark}

In general, solutions of nonlinear dispersive equations with the compactness property are expected to be very special solutions like solitons or self-similar solutions. In the case of equation \eqref{CP}, the self-similar behaviour was excluded in \cite{KeMe08}, and we conjecture in view of Proposition \ref{P:compact}, that $0$, and $Q_{\vell}$ with $Q\in \Sigma$ and $|\vell|<1$ are the only solutions of \eqref{CP} with the compactness property. This conjecture was settled in \cite{DuKeMe11a} in the radial case. We shall prove it in a subsequent paper \cite{DuKeMe13Pb} in the nonradial setting with an additional nondegeneracy assumption on the profile $Q^+$ given by Proposition \ref{P:compact}.

\bigskip

Let us say a few words about the proof of our results.

The proof of Proposition \ref{P:compact} in Section \ref{S:virial} is based on a monotonicity formula. The use of monotonicity formulas to prove rigidity results on solutions with the compactness property is by now standard, but usually requires, for focusing equations, a size restriction on the solution. The fact that such a formula works without size restriction seems to be specific to the energy-critical wave equation \eqref{CP}. It is also possible to prove rigidity theorems by the ``channels of energy'' method. This strategy was so far mostly implemented in a radial setting (see \cite{DuKeMe12c,Shen12P,KeLaSc13P}). Let us also mention that the proof of \cite{DuKeMe13Pb} combines the two techniques, using both Proposition \ref{P:compact} and a channel of energy property in a nonradial setting.

Theorem \ref{T:profiles} is a consequence of Proposition \ref{P:compact} and a very general minimality argument based on profile decomposition (see H.~Bahouri and P.~G\'erard \cite{BaGe99} for equation \eqref{CP}). It consists in choosing, among all sequences of times $\{t_n\}_n$ such that $t_n\to T_+(u)$, a sequence minimizing some quantities involving the total energy and the number of nonscattering nonlinear profiles in the profile decomposition of $\{(u(t_n),\partial_tu(t_n))\}_n$. A clever adjustment of the times sequence is then needed in order to get exactly the local strong convergences \eqref{CV_THM1} and \eqref{CV_THM2}.

This construction generalizes some of the proofs in Section 3 of \cite{DuKeMe12}, and may also be seen as an extension of the ̀``compactness'' step of the compactness/rigidity method initiated in \cite{KeMe06}. In fact the main global well-posedness and scattering results in \cite{KeMe08} can be seen to follow directly from \eqref{CV_THM2} in Theorem \ref{T:profiles}, thus bypassing the ``critical element'' construction in \cite{KeMe08}. Similarly, the global well-posedness and scattering results for NLS equation \eqref{NLS} proved in \cite{KeMe06} follows from Theorem \ref{T:NLS} in Section \ref{S:NLS} and the rigidity theorem in \cite{KeMe06}.

Theorem \ref{T:profiles} is proved in Section \ref{S:minimality}. Section \ref{S:profiles} contains preliminaries on profile decompositions, including the introduction of a convenient order for the profiles, and the study of sequences of profile decompositions (see Lemma \ref{L:DPD}). 

Let us emphasize again that the minimality argument mentioned in the previous paragraphs is not specific to equation \eqref{CP}, but works for quite general dispersive equations, implying that for any bounded non-scattering solution of the equation, there exists a sequence of times converging to the final time of existence such that the modulated solution converges (in some weak sense) to solutions with the compactness property. Indeed, ignoring Proposition \ref{P:compact}, the first part of the proof of Theorem \ref{T:profiles} would yield the following weaker result:
\begin{prop}
\label{P:profiles}
Let $u$ be a solution of \eqref{CP} and assume that $u$ does not scatter forward in time and 
 \begin{equation*}
 \sup_{t\in [0,T_+(u))}\|(u(t),\partial_tu(t))\|_{\hdot\times L^2}<\infty.
 \end{equation*} 
 Then there exist sequences $\{t_n\}_n$ in $[0,T_+(u))$, $\{\lambda_n\}_n$ in $(0,+\infty)$, $\{x_{n}\}_n$ in $\RR^N$ and a solution $U$ of \eqref{CP} with the compactness property such that 
   \begin{equation*}
\left(\lambda_n^{\frac{N}{2}-1}u\left(t_n,\lambda_n\cdot+x_{n}\right),\lambda_n^{\frac{N}{2}}\partial_t u\left(t_n,\lambda_n\cdot+x_n\right)\right)\xrightharpoonup[n\to \infty]{} (U(0),\partial_tU(0)).
\end{equation*} 
\end{prop}
Analogs of this weak version of Theorem \ref{T:profiles} can be proved by the same method for many critical and subcritical dispersive equations. This shows the crucial role played by solutions with the compactness property in the study of bounded solutions for these equations, even when no restriction on the size of the solution is assumed. We refer to Theorem \ref{T:NLS} in Section \ref{S:NLS} for a version of Theorem \ref{T:profiles} for the Schr\"odinger equation \eqref{NLS}. Note that for most focusing equations, including \eqref{NLS}, partial rigidity results such as Proposition \ref{P:compact} above (let alone general classification of solutions with the compactness property) are not known. 



\subsection*{Notations}
If $I$ is an interval, we denote
$$S(I)=L^{\frac{2(N+1)}{N-2}}(I\times \RR^N).$$
If $u$ is a function of $t\in \RR$, $x\in \RR^N$, we let $\vu=(u,\partial_tu)$ and $\nabla_{t,x}u=\left(\partial_tu,\partial_{x_1}u,\ldots,\partial_{x_N}u\right)$. If $(v_0,v_1)\in \hdot\times L^2$, we write $S_L(t)(v_0,v_1)=v(t)$, $\vS_L(t)(v_0,v_1)=\vv(t)$, where $v(t)$ is the solution of the linear wave equation
\begin{equation}
\label{LCP}
\partial_t^2v-\Delta v=0,\quad (v,\partial_tv)_{\restriction t=0}=(v_0,v_1).
\end{equation}
We denote by $o_n(1)$ any sequence $\{t_n\}_n$ of real numbers such that $\lim_n t_n=0$.

The open unit ball of $\RR^N$ for the Euclidean norm $|\cdot|$ is denoted by $B^N$.
\subsection*{Acknowledgment}
The authors would like to thank the anonymous referees for helpful comments and suggestions.
\section{Properties of solutions with the compactness properties}
\label{S:virial}
If $Q\in \Sigma$, we denote:
\begin{equation*}
 Q_{\ell}(t,x)=Q_{\ell \ve_1}(t,x)=Q\left(\frac{x_1-t\ell}{\sqrt{1-\ell^2}},x'\right).
\end{equation*} 
In this section, we prove:
\begin{prop}
\label{P:D1}
Let $u$ be a nonzero solution of \eqref{CP} with the compactness property. Let $T_{\pm}=T_{\pm}(u)$. Then
\begin{gather}
\label{D2}
E[u]>0\\
\label{D3}
T_+=+\infty\text{ or }T_-=-\infty.
\end{gather}
\end{prop}
\begin{prop}
 \label{P:D1bis}
 Let $u$ be as in Proposition \ref{P:D1}.
 Assume (to fix ideas) that $P[u]$ is parallel to $\ve_1=(1,0,\ldots,0)$. Let $\ell$ be such that
\begin{equation}
\label{defl}
 \ell \ve_1=-\frac{P[u]}{E[u]}.
 \end{equation} 
 Then $\ell\in (-1,+1)$ and there exists $Q\in \Sigma$, a sequence $\{t_n\}_n$ in $(T_-,T_+)$ such that $t_n\to T_+$ and
 \begin{multline}
 \label{D4}
  \lim_{n\to\infty} \left(\lambda^{\frac{N}{2}-1}(t_n) u\left(t_n,\lambda(t_n)\cdot+x(t_n)\right), \lambda^{\frac{N}{2}}(t_n) \partial_t u\left(t_n,\lambda(t_n)\cdot+x(t_n)\right)\right)\\
  =\left(Q_{\ell}(0),\partial_tQ_{\ell}(0)\right).
 \end{multline}
 in $\hdot\times L^2$.
\end{prop}
Note that Proposition \ref{P:compact} follows immediately from Propositions \ref{P:D1} and \ref{P:D1bis}.
We will only sketch most of the proofs, which are very similar to some of the proofs in \cite{DuKeMe12}.
\subsection{Positive energy and infinite interval of existence}
In this subsection we sketch the proof of Proposition \ref{P:D1}.

\EMPH{Sketch of proof of \eqref{D2}}
It is exactly \cite[Lemma 4.4]{DuKeMe12}. The proof is by contradiction. If $E[u]\leq 0$, then by \cite{KeMe08}, $u$ blows up in finite time in both time directions. Using that $u$ has the compactness property, one gets (see e.g. Lemma 4.8 of \cite{KeMe08}) that there exists a unique negative (respectively positive) time blow-up point $x_-$ (respectively $x_+$) in $\RR^N$ such that 
\begin{multline}
 \label{D5}
 \quad\supp (u,\partial_tu)\subset\\
 \Big\{ (t,x) \in (T_-,T_+)\times \RR^N\text{ s.t. } |x-x_+|\leq |T_+-t|\text{ and }|x-x_-|\leq |T_--t|\Big\}.\quad
\end{multline} 
Let
$$ y(t)=\int u^2(t,x)\,dx.$$
Then by explicit computation, using that the energy is nonpositive, we get that $y$ is convex. Furthermore, by Hardy's inequality and the property \eqref{D5} of the support of $u$, $y$ satisfies
$$ \lim_{t \to T_+} y(t)=\lim_{t \to T_-} y(t)=0.$$
Thus $y$ is equal to $0$, a contradiction since by our assumption $u$ is not identically $0$.

\EMPH{Sketch of proof of \eqref{D3}}
We argue by contradiction, assuming that both $T_-$ and $T_+$ are finite. As before, we deduce that there exist two blow-up points, $x_-$ and $x_+$, such that  \eqref{D5} holds.
By \eqref{D2}, the energy of $u$ is positive. Assuming (after a space rotation) that $P[u]$ is parallel to $ \ve_1$, we can define $\ell$ by \eqref{defl}.
Let 
\begin{equation}
\label{def_e}
e(t,x)=\frac{1}{2}|\nabla u(t,x)|^2+\frac{1}{2}(\partial_tu(t,x))^2-\frac{N-2}{2N}|u(t,x)|^{\frac{2N}{N-2}}
\end{equation} 
 be the density of energy, and  
$\Psi(t)=\int xe(t,x)\,dx$. Then (see \cite[Claim 2.12]{DuKeMe12}) $\Psi'(t)=\ell E[u] \ve_1$, and thus, integrating between $T_-$ and $T_+$, $\ell(T_+-T_-)E[u]\ve_1=(x_+-x_-)E[u]$. Since $E[u]\neq 0$, we obtain:
\begin{equation}
 \label{D6}
 \ell(T_+-T_-) \ve_1=x_+-x_-.
\end{equation} 
Now, let
\begin{equation}
 \label{D7}
 \tZ(t)=(\ell^2-1)\int (x-t\ell  \ve_1)\cdot\nabla u\partial_tu\,dx+\frac{N-2}{2} (\ell^2-1)\int u\partial_t u\,dx-\ell^2\int (x_1-t\ell)\partial_{x_1}u\partial_t u\,dx.
\end{equation} 
By Claim 2.12 in \cite{DuKeMe12} and the definition of $\ell$, 
\begin{equation}
 \label{D8}
 \tZ'(t)=\int (\partial_t u+\ell\partial_{x_1}u)^2\,dx.
\end{equation} 
Furthermore,
$$ \lim_{t\to T_{\pm}} \tZ(t)=(\ell^2-1)(x_{\pm}-T_{\pm}\ell \ve_1)\cdot P[u]-\ell^2(x_{1\pm}-T_{\pm}\ell)P_1[u],$$
where $P_1[u]=\int \partial_{x_1}u_0u_1$ (respectively $x_{1\pm}$) is the first coordinate of the momentum (respectively of $x_{\pm}$) in the canonical basis of $\RR^N$. Combining with \eqref{D6}, we get $\lim_{t\to T_+}\tZ(t)=\lim_{t\to T_-}\tZ(t)$. By \eqref{D8},
\begin{equation}
 \label{D10} \int_{T_-}^{T_+} \int (\partial_tu+\ell\partial_{x_1}u)^2\,dx\,dt=0.
\end{equation} 
This implies:
\begin{equation}
 \label{D11}
 \partial_tu+\ell\partial_{x_1}u=0 \text{ in } (T_-,T_+)\times \RR^N.
\end{equation} 
Differentiating with respect to $t$, we get, using also \eqref{CP},
\begin{equation}
 \label{D12}
 \Delta u+|u|^{\frac{4}{N-2}}u-\ell^2\partial_{x_1}^2u=0 \text{ in } (T_-,T_+)\times \RR^N.
\end{equation} 
This shows, using Lemma \ref{L:D2} below, that $\ell\in (-1,+1)$ and $u(t,x)=Q_{\ell}(t,x)$ for some stationary solution $Q$, contradicting the fact that $u$ is not globally defined. It remains to show:
\begin{lemma}
 \label{L:D2}
 Let $f\in \hdot(\RR^N)\setminus\{0\}$ and $\ell\in \RR$. Assume
 $$(1-\ell^2)\partial_{x_1}^2f+ \sum_{j=2}^N\partial_{x_j}^2 f+|f|^{\frac{4}{N-2}}f=0$$
 Then $\ell^2<1$ and there exists a stationary solution $Q$ of \eqref{CP} such that $f(x)=Q_{\ell}(0,x)$.
\end{lemma}
\begin{proof}[Sketch of proof]
 Using finite time of propagation (to exclude the case $\ell^2>1$) or a Pohozaev identity (to exclude the case $\ell^2=1$), it is shown in Steps 1 and 2 of the proof of \cite[Lemma 2.6]{DuKeMe12} that $\ell^2<1$. \footnote{Note that a correction to the part of \cite[Lemma 2.6]{DuKeMe12} that is not used here is contained in \cite{DuKeMe13Pb}. See also the corrected version \cite{DuKeMe13arXiv} on arXiv.}
 
 Let $g(x)=f(\sqrt{1-\ell^2}x_1,x_2,\ldots,x_N)$. Then the assumptions on $f$ imply
 $$\Delta g+|g|^{\frac{4}{N-2}}g=0,\quad g\in \hdot(\RR^N)$$
 and the result follows.
 \end{proof}
 \subsection{Congergence to a solitary wave}
 In this subsection we give a sketch of the proof of Proposition \ref{P:D1bis}. We divide it into three lemmas.
 \begin{lemma}
  \label{L:D3}
  Let $u$ be as in Proposition \ref{P:D1}. Assume furthermore $T_+<\infty$. Then the conclusion of Proposition \ref{P:D1bis} holds. 
 \end{lemma}
\begin{proof}[Sketch of proof]
We will assume without loss of generality that $x(t)$ and $\lambda(t)$ are continuous functions of $t$ (see \cite[Remark 5.4]{KeMe06}).

 \EMPH{Step 1}
 Assume to fix ideas $T_+=1$. By \cite[Lemma 4.8]{KeMe08}, there exists $x_+\in \RR^N$ such that
 \begin{equation}
 \label{supp_u+}
  \supp u\subset \{(t,x),\; |x-x_+|\leq |t-1|\}.
 \end{equation} 
As a consequence,  $\lim_{t\to 1}x(t)=x_+$. We will assume $x_+=0$, so that
 \begin{equation}
 \label{CVx+}
 \lim_{t\to 1}x(t)=0.
 \end{equation} 
By \cite[Lemma 4.7]{KeMe08}, there exists $C>0$ such that $\lambda(t)\leq C(1-t)$ for $t$ close to $1$. By \cite[Section 6]{KeMe08}, self-similar blow-up is excluded: there exists a sequence $\{t_n\}$ in $(0,1)$, with $t_n\to 1$ such that
 \begin{equation}
  \label{D13}
  \lim_{n\to\infty} \frac{\lambda(t_n)}{1-t_n}=0.
 \end{equation} 
 \EMPH{Step 2: control of space translation} We prove that for any sequence $t_n\to 1$ such that \eqref{D13} holds, we have
 \begin{equation}
 \label{D14}
 \lim_{n\to\infty}\frac{x(t_n)}{1-t_n}=-\ell \ve_1.
 \end{equation} 
 This is exactly Lemma 4.6 of \cite{DuKeMe12}. Let us give a quick idea of the proof. Let $\Psi(t)$ be defined by
 \begin{equation}
  \label{D15}
  \Psi(t)=\int xe(t,x)\,dx
 \end{equation} 
 Then by explicit computations, equation \eqref{CP}, and the definition of $\ell$, 
 \begin{equation}
  \label{Psi'}
 \Psi'(t)=\ell E[u] \ve_1. 
 \end{equation} 
 On the other hand, by \eqref{D13} and the compactness of $\overline{K}$, 
 \begin{equation}
  \label{D16}
  \lim_{n\to \infty} \frac{1}{1-t_n}\int (x-x(t_n))e(t_n,x)\,dx=0
 \end{equation} 
 i.e.
 \begin{equation}
  \label{D17}
  \lim_{n\to\infty} \left|\frac{1}{1-t_n} \left( \Psi(t_n)-x(t_n)E[u] \right)\right|=0.
 \end{equation} 
 By \eqref{CVx+}, $\lim_{t\to 1}\Psi(t)=0$. Integrating \eqref{Psi'} between $t_n$ and $1$ we get
 $$\Psi(t_n)=-\ell E[u](1-t_n) \ve_1$$
 which concludes, in view of \eqref{D17} (and since $E[u]\neq 0$), the proof of \eqref{D14}.
 
 \EMPH{Step 3: virial argument} We show that for any sequence $\{t_n\}$ in $(0,1)$ with $t_n\to 1$ such that \eqref{D13} holds, we have
 \begin{equation}
  \label{D18}
  \lim_{n\to\infty} \frac{1}{1-t_n}\int_{t_n}^1 \int (\partial_tu(t,x)+\ell\partial_{x_1}u(t,x))^2\,dx\,dt=0.
 \end{equation} 
 This is Lemma 4.7 of \cite{DuKeMe12}. Again, we only give a quick idea of the proof.
 Let
 \begin{multline}
  \label{defZ}
  Z(t)=(\ell^2-1) \int \left(x+\ell (1-t) \ve_1\right)\cdot\nabla u\partial_tu\\
  +\frac{N-2}{2}(\ell^2-1)\int u\partial_t u-\ell^2\int \left(x_1+\ell(1-t)\right)\partial_{x_1}u\partial_t u.
 \end{multline} 
 Then $Z'(t)=\int (\partial_tu+\ell\partial_{x_1}u)^2\,dx$ and \eqref{D18} will follow from
 \begin{equation}
  \label{D19} \lim_{n\to\infty} \frac{Z(t_n)}{1-t_n} =0.
 \end{equation} 
 The property \eqref{D19} follows from \eqref{D13}, \eqref{D14} and the compactness of $\overline{K}$. 
 To prove it, fix a small $\eps>0$, and divide the integrals defining $Z(t)$ into the regions $|x-x(t_n)|\geq A_{\eps}\lambda(t_n)$ and $|x-x(t_n)|\leq A_{\eps}\lambda(t_n)$, where $A_{\eps}$ is a large positive parameter, given by the compactness of $\overline{K}$. We omit the details, and refer to the proof of \cite[Lemma 4.7]{DuKeMe12}
 
\EMPH{Step 4: end of the proof} This step is the same as the proof of \cite[Lemma 4.9]{DuKeMe12}. By the arguments of \cite[Corollary 5.3]{DuKeMe11a}, we deduce from \eqref{D18} that there exists a sequence $\{t'_n\}$ in $(0,1)$, with $t_n'\to 1$ and 
$$\lim_{n\to\infty} \left(  \lambda^{\frac{N}{2}-1}(t_n') u\left(t_n',\lambda(t_n')\cdot+x(t_n')\right), \lambda^{\frac{N}{2}}(t_n) \partial_t u\left(t_n',\lambda(t_n')\cdot+x(t_n')\right)\right)=\left(U_0,U_1\right)$$
in $\hdot\times L^2$, where the solution $U$ of \eqref{CP} with initial data $(U_0,U_1)$ satisfies, for some $T\in (0,T_+(U))$,
\begin{equation}
 \label{D20}
 \partial_t U+\ell\partial_{x_1}U=0 \text{ in }(0,T)\times \RR^N.
\end{equation} 
 Using Lemma \ref{L:D2} as in the end of the proof of Proposition \ref{P:D1} above, we obtain that $\ell^2<1$ and $U(t,x)=Q_{\ell}(t,x)$ for some stationary solution $Q$. The proof is complete.
 \end{proof}
\begin{lemma}
 \label{L:D4} 
 Let $u$ be as in Proposition \ref{P:D1}. Assume furthermore $T_+=+\infty$ and
 \begin{equation}
  \label{D22}
  \lim_{t\to +\infty}\frac{\lambda(t)}{t}=0.
 \end{equation} 
 Then the conclusion of Proposition \ref{P:D1bis} holds. 
\end{lemma}
\begin{proof}
 Again, we will assume without loss of generality that $x(t)$ and $\lambda(t)$ are continuous functions of $t$. Lemma \ref{L:D4} is exactly Steps 2,3 and 4 of the proof of \cite[Lemma 4.10]{DuKeMe12}. We give a short summary of the arguments.

 \EMPH{Step 1: control of the space translation} We prove
 \begin{equation}
  \label{D23}
  \lim_{t\to +\infty} \frac{\left|x(t)-t\ell  \ve_1\right|}{t}=0.
 \end{equation} 
 Let, for a large parameter $\tau$, 
 $$\Psi_{\tau}(t)=\int x\varphi\left(\frac{x}{\tau}\right) e(t,x)\,dx,$$
 where $\varphi(x)=1$ if $|x|\leq 3$ and $\varphi(x)=0$ if $|x|\geq 4$. Then, by explicit computations and the equation \eqref{CP},
 \begin{equation}
  \label{D24}
  \left|\Psi'_{\tau}(t)-\ell E[u] \ve_1\right|\leq C\int_{|x|\geq 3\tau}\left( |\nabla u|^2+(\partial_t u)^2+ |u|^{\frac{2N}{N-2}}+\frac{1}{|x|^2}u^2\right)\,dx.
 \end{equation} 
 Note also that
 \begin{equation}
  \label{D25}
  \Psi_{\tau}(\tau)-x(\tau)E[u]=\int \left( x\varphi\left(\frac{x}{\tau}  \right) -x(\tau)\right)e(\tau,x)\,dx.
 \end{equation} 
 Using the compactness of $\overline{K}$, the bound $|x(t)|\leq C+|t|$ (which follows easily from finite speed of propagation, see e.g. the proof of (4.9) in \cite{KeMe08}) and \eqref{D22}, we get 
 \begin{align*}
 \lim_{\tau\to+\infty} \frac{1}{\tau}\int_0^{\tau} \int_{|x|\geq 3\tau}\left( |\nabla u|^2+(\partial_t u)^2+ |u|^{\frac{2N}{N-2}}+\frac{1}{|x|^2}u^2\right)\,dx\,dt&=0\\
  \lim_{\tau\to +\infty} \frac{1}{\tau}\int \left( x\varphi\left(\frac{x}{\tau}  \right) -x(\tau)\right)e(\tau,x)\,dx&=0
 \end{align*}
Integrating \eqref{D24} between $0$ and $\tau$ and combining with \eqref{D25} we get \eqref{D23}.

\EMPH{Step 2. Virial argument}
We next prove 
\begin{equation}
 \label{D26}
 \lim_{T\to+\infty} \frac{1}{T}\int_0^T \int (\partial_t u+\ell\partial_{x_1}u)^2\,dx\,dt=0.
\end{equation} 
This is Step 3 of the proof of \cite[Lemma 4.10]{DuKeMe12}. Let $R>0$ be a large parameter. Let 
\begin{multline}
 \label{D27}
 Z_R(t)=(\ell^2-1)\int (x-t\ell  \ve_1)\cdot\nabla u\partial_tu \varphi\left( \frac{x-t\ell \ve_1}{R} \right)+\\
 \frac{N-2}{2}(\ell^2-1)\int u\partial_tu \varphi\left( \frac{x-t\ell \ve_1}{R} \right)-\ell^2 \int (x_1-t\ell)\partial_{x_1}u\partial_t u \varphi\left( \frac{x-t\ell  \ve_1}{R} \right),
\end{multline} 
where $\varphi$ is as in Step 1. Then
$$\left|Z_R'(t)-\int (\partial_t u+\ell\partial_{x_1}u)^2\right|\leq C\int_{|x-t\ell  \ve_1|\geq R} |\nabla u|^2+(\partial_tu)^2+|u|^{\frac{2N}{N-2}}+\frac{1}{|x|^2}u^2,$$
and \eqref{D26} follows from estimates using the compactness of $\overline{K}$, assumption \eqref{D22} and Step 1.

\EMPH{Step 3: end of the proof}
This is Step 4 in the proof of \cite[Lemma 4.10]{DuKeMe12}, and is very similar to the ends of the proofs of Proposition \ref{P:D1} and Lemma \ref{L:D3}. We omit the details.
 \end{proof}
The proof of Proposition \ref{P:D1bis} will be complete once we have proved:
\begin{lemma}
 \label{L:D5}
  Let $u$ be as in Proposition \ref{P:D1}. Assume furthermore $T_+=\infty$ and that \eqref{D22} does not hold. Then the conclusion of Proposition \ref{P:D1bis} is valid. 
 \end{lemma}
\begin{proof}
Let 
 \begin{equation}
\label{defwn}
 w_n(s,y)=\lambda(t_n)^{\frac{N}{2}-1} u\left(t_n+\lambda(t_n)s,\lambda(t_n)y+x(t_n)\right).  
 \end{equation} 
 Then (extracting subsequences if necessary), there exists $(w_0,w_1)\in \hdot\times L^2$ such that
 \begin{equation}
  \label{DD10}
  \lim_{n\to\infty} (w_n(0),\partial_tw_n(0))=(w_0,w_1)\text{ in }\hdot\times L^2.
 \end{equation} 
Since $u$ is not identically $0$, $(w_0,w_1)\neq (0,0)$. 

Note that by finite speed of propagation, $\lambda(t)/t$ is bounded as $t\to+\infty$. Indeed, assume that for a sequence $t_n\to +\infty$, one has $\lambda(t_n)/t_n\to +\infty$. 
%
If $R>0$,
\begin{multline*}
 \int_{|x|\geq |t_n|+R} |\nabla u(t_n,x)|^2+(\partial_tu(t_n,x))^2\,dx=\int_{\left|y+\frac{x(t_n)}{\lambda(t_n)}\right|\geq \frac{|t_n|+R}{\lambda(t_n)}} |\nabla w_{n}(0,y)|^2+(\partial_t w_{n}(0,y))^2\,dy\\
\underset{n\to\infty}{\longrightarrow} \int_{\RR^N} |\nabla w_0(y)|^2+(w_1(y))^2\,dy\neq 0,
\end{multline*}
where we have used that $\frac{|t_n|+R}{\lambda(t_n)}\to 0$ as $n\to \infty$. This contradicts finite speed of propagation, proving as announced that $\lambda(t)/t$ is bounded for large $t$.

 Since \eqref{D22} does not hold, there exists a sequence $t_n\to+\infty$, and $\tau_0\in (0,+\infty)$ such that
 \begin{equation}
 \label{DD9}
 \lim_{n\to\infty} \frac{\lambda(t_n)}{t_n}=\frac{1}{\tau_0}.
 \end{equation}

 Let $w$ be the solution of \eqref{CP} with initial data $(w_0,w_1)$. We distinguish two cases.

 \EMPH{Case 1: $T_-(w)<-\tau_0$}
 Let $s_n=-\frac{t_n}{\lambda(t_n)}$. Then 
 \begin{multline}
 \label{expwn}
 (w_n(s_n,y),\partial_t w_n(s_n,y))\\
 =\left(\lambda(t_n)^{\frac{N}{2}-1} u(0,\lambda(t_n)y+x(t_n)),\lambda(t_n)^{\frac{N}{2}}\partial_tu(0,\lambda(t_n)y+x(t_n))\right).
 \end{multline} 
 By \eqref{DD9}, \eqref{DD10} and the assumption $T_-(w)<-\tau_0$, we obtain
 $$\lim_{n\to\infty} (w_n(s_n), \partial_t w_n(s_n))=(w(-\tau_0),\partial_t w(-\tau_0))\text{ in }\hdot\times L^2.$$
 This shows by \eqref{expwn} that $\lambda(t_n)$ is bounded, contradicting \eqref{DD9}.
 
 \EMPH{Case 2. $T_-(w)\geq -\tau_0$}
 We prove that $w$ has the compactness property, using a by now standard argument (see e.g. \cite[Section 7]{KeMe08}). Indeed, fix $s\in (T_-(w),T_+(w))$. For large $n$, define 
 \begin{align*}
  u_{0n}(y)&=\lambda\left( t_n+\lambda(t_n)s \right)^{\frac{N}{2}-1}u\Big( t_n+\lambda(t_n)s,\lambda\left(  t_n+\lambda(t_n)s\right)y+x\left( t_n+\lambda(t_n)s \right) \Big)\\
  u_{1n}(y)&=\lambda\left( t_n+\lambda(t_n)s \right)^{\frac{N}{2}}\partial_tu\Big( t_n+\lambda(t_n)s,\lambda\left(  t_n+\lambda(t_n)s\right)y+x\left( t_n+\lambda(t_n)s \right) \Big). 
 \end{align*}
(note that $(u_{0n},u_{1n})$ depends also on $s$).
Since $s>T_-(w)\geq -\tau_0$, we get, in view of \eqref{DD9}, that $t_n+\lambda(t_n)s$ is positive for large $n$, which shows that for large $n$, $u_{0n}$ and $u_{1n}$ are well defined, and $(u_{0n},u_{1n})\in K$. Next note that
\begin{align*}
 u_{0n}(y)&=\left(\frac{\lambda\left( t_n+\lambda(t_n)s \right)}{\lambda(t_n)}\right)^{\frac{N}{2}-1}w_{n}\left( s,\frac{\lambda\left(  t_n+\lambda(t_n)s\right)}{\lambda(t_n)}y+\frac{x\left( t_n+\lambda(t_n)s \right) -x(t_n)}{\lambda(t_n)}\right)\\
  u_{1n}(y)&=\left(\frac{\lambda\left( t_n+\lambda(t_n)s \right)}{\lambda(t_n)}\right)^{\frac{N}{2}}\partial_sw_{n}\left(s,\frac{\lambda\left(  t_n+\lambda(t_n)s\right)}{\lambda(t_n)}y+\frac{x\left( t_n+\lambda(t_n)s \right)-x(t_n)}{\lambda(t_n)} \right). 
 \end{align*}
By the continuity of the flow of \eqref{CP}
$$ \lim_{n\to\infty}(w_n(s),\partial_tw_n(s))=(w(s),\partial_tw(s))\text{ in }\hdot\times L^2.$$
Since $(u_{0n},u_{1n})\in K$ for all $n$ and $0\notin K$, we deduce that there exists a constant $C(s)>0$ such that for all $n$
$$\frac{1}{C(s)} \leq \frac{\lambda\left( t_n+\lambda(t_n)s \right)}{\lambda(t_n)}\leq C(s) \text{ and }\left|\frac{x\left( t_n+\lambda(t_n)s \right) -x(t_n)}{\lambda(t_n)}\right|\leq C(s).$$
Extracting subsequences, we can assume 
$$\lim_{n\to\infty}\frac{\lambda\left( t_n+\lambda(t_n)s \right)}{\lambda(t_n)}=\tlambda(s)>0\text{ and }\lim_{n\to\infty}\frac{x\left( t_n+\lambda(t_n)s \right) -x(t_n)}{\lambda(t_n)}=\tx(s)\in \RR^N.$$
Using again that $(u_{0n},u_{1n})\in K$, we deduce
$$ \left(\tlambda^{\frac{N}{2}-1}(s)w(s,\tlambda(s)\cdot+\tx(s)),\tlambda^{\frac{N}{2}}(s)\partial_sw(s,\tlambda(s)\cdot+\tx(s))\right)\in \overline{K}.$$
Since the above construction of $\tlambda(s)$ and $\tx(s)$ works for all $s\in (T_-(w),T_+(w))$, we get  that $w$ has the compactness property. 

We next deduce the desired convergence \eqref{D4}. Since $T_-(w)$ is finite, we deduce from Lemma \ref{L:D3} that $\ell\in (-1,+1)$, and that there exists a stationary solution $Q$ and a sequence $\{\tau_n\}$ in $(T_-(w),T_+(w))$, such that $\tau_n\to T_-(w)$ as $n\to\infty$ and
\begin{multline*}
\lim_{n\to\infty}\left(\tlambda^{\frac{N}{2}-1}(\tau_n)w(\tau_n,\tlambda(\tau_n)\cdot+\tx(\tau_n)),\tlambda^{\frac{N}{2}}(\tau_n)\partial_tw(\tau_n,\tlambda(\tau_n)\cdot+\tx(\tau_n))\right)\\ =(Q_{\ell}(0),\partial_tQ_{\ell}(0))\text{ in }\hdot\times L^2. 
\end{multline*}
(we have used that the momentum and energy of $w$ and of $u$ are the same to get $\ell\ve_1=-P[w]/E[w]=-P [u]/E[u]\in (-1,+1)$).
Let $p\geq 1$ be an integer, and choose an index $n_p$ such that $T_-(w)<\tau_{n_p}<0$ and
\begin{multline}
 \label{DD10'}
\bigg\|
\left(\tlambda^{\frac{N}{2}-1}(\tau_{n_p})w(\tau_{n_p},\tlambda(\tau_{n_p})\cdot+\tx(\tau_{n_p})),\tlambda^{\frac{N}{2}}(\tau_{n_p})\partial_tw(\tau_{n_p},\tlambda(\tau_{n_p})\cdot+\tx(\tau_{n_p}))\right)\\
-(Q_{\ell}(0),\partial_tQ_{\ell}(0))\bigg\|_{\hdot\times L^2}<\frac 1p.
\end{multline} 
Since
\begin{equation}
 \label{DD11}
\lim_{k\to +\infty} \left\| \left(w_k(\tau_{n_p})-w(\tau_{n_p}),\partial_tw_k(\tau_{n_p})-\partial_tw(\tau_{n_p})\right)\right\|_{\hdot\times L^2}=0,
\end{equation} 
we get that for large $k$ (in view of the definition \eqref{defwn} of $w_k$)
\begin{multline}
\label{DD12}
\bigg\| \left(\mu_{k,p}^{\frac{N}{2}-1} u\left(t_k+\lambda(t_k)\tau_{n_p},\mu_{k,p}\cdot +x_{k,p}\right),\mu_{k,p}^{\frac{N}{2}} \partial_t u\left(t_k+\lambda(t_k)\tau_{n_p},\mu_{k,p}\cdot +x_{k,p}\right)\right)\\
-(Q_{\ell}(0),\partial_tQ_{\ell}(0))\bigg\|_{\hdot\times L^2}<\frac{2}{p},
\end{multline} 
where $\mu_{k,p}=\lambda(t_k)\tlambda(\tau_{n_p})$, $x_{k,p}=\lambda(t_{k})\tx(\tau_{n_p})+x(t_k)$. Since $\tau_{n_p}>-\tau_0$, we obtain  by \eqref{DD9} 
$$\lim_{k\to+\infty} t_k+\lambda(t_k)\tau_{n_p}=+\infty.$$
Choose $k_p$ such that
$$ t_p':= t_{k_p}+\lambda(t_{k_p})\tau_{n_p}\geq p,$$
and \eqref{DD12} holds for $k=k_p$. Let $\mu_p'=\mu_{k_p,p}$, $x_p'=x_{k_p,p}$. Then
\begin{equation}
 \label{DD13}
\lim_{p\to+\infty} \left({\mu_p'}^{\frac{N}{2}-1} u\left(t_p',\mu_p'\cdot +x_p'\right),{\mu_p'}^{\frac{N}{2}} \partial_tu(t_p',\mu_p'\cdot+x_p')\right)=\left(Q_{\ell}(0),\partial_tQ_{\ell}(0)\right)\text{ in }\hdot\times L^2.
 \end{equation}
and $\lim_{p\to\infty} t_p'=+\infty$. Using the compactness of $K$, it is easy to deduce that  \eqref{DD13} still holds when $\mu'_p$ and $x_p'$ are replaced with $\lambda(t_p')$ and $x(t_p')$ (extracting subsequences, rescaling and space translating $Q$ if necessary).  
\end{proof}

\section{Profile decomposition}
\label{S:profiles}
In this section we recall a few facts about the profile decomposition of H.~Bahouri and P.~G\'erard. We also put the profiles of this decomposition into an order that is convenient when writing the approximation of a solution of \eqref{CP} as a sum of profiles (Subsections \ref{SS:preorder} and \ref{SS:approx}) and prove a new result concerning sequences of profile decompositions (Subsection \ref{SS:DPD}).
\subsection{Definition}
Let $\big\{(u_{0,n},u_{1,n})\big\}_n$ be a bounded sequence in $\hdot\times L^2$. For $j\geq 1$, consider a solution $U_L^j$ of the linear wave equation \eqref{LCP}, and a sequence $\left\{\lambda_{j,n},x_{j,n},t_{j,n}\right\}_n$  in $(0,+\infty)\times \RR^N\times \RR$. The sequences of parameters $\left\{\lambda_{j,n},x_{j,n},t_{j,n}\right\}_n$, $j\geq 1$ are said to be \emph{orthogonal} if for all $j,k\geq 1$
\begin{equation}
\label{orthogonal}
j\neq k\Longrightarrow
\lim_{n\to\infty}\left|\log\frac{\lambda_{j,n}}{\lambda_{k,n}}\right|+\frac{\left|t_{j,n}-t_{k,n}\right|}{\lambda_{j,n}}+\frac{\left|x_{j,n}-x_{k,n}\right|}{\lambda_{j,n}}=\infty.
\end{equation} 
We say that  $\left( U_L^j,\left\{\lambda_{j,n},x_{j,n},t_{j,n}\right\}_n \right)_{j\geq 1}$ is a \emph{profile decomposition} of the sequence $\big\{(u_{0,n},u_{1,n})\big\}_n$ if \eqref{orthogonal} is satisfied and, denoting by
\begin{equation}
 \label{linear_profile}
 U^j_{L,n}(t,x)=\frac{1}{\lambda_{j,n}^{\frac{N}{2}-1}}U_L^j\left( \frac{t-t_{j,n}}{\lambda_{j,n}},\frac{x-x_{j,n}}{\lambda_{j,n}} \right),
\end{equation} 
and 
\begin{equation}
 \label{def_wnJ}
 w_n^J(t,x)=S_L(t)\left(u_{0,n},u_{1,n}\right)-\sum_{j=1}^JU^j_{L,n}(t,x),
\end{equation} 
the following property holds:
\begin{equation}
 \label{prop_wnJ}
 \lim_{J\to\infty}\limsup_{n\to\infty} \left\|w_n^J\right\|_{S(\RR)}=0.
\end{equation} 
By the paper \cite{BaGe99} of Bahouri and G\'erard\footnote{This article is written in space dimension $3$ but an adaptation of the proof gives the case of general space dimension (see \cite{Bulut10}).}, if $\{(u_{0,n},u_{1,n})\}_n$ is a bounded sequence of $\hdot\times L^2$, there exists a subsequence (that we will also denote by $\{(u_{0,n},u_{1,n})\}_n$) that admits a profile decomposition $\left( U_L^j,\left\{\lambda_{j,n},x_{j,n},t_{j,n}\right\}_n \right)_{j\geq 1}$. We recall the following Pythagorean expansions, valid for all $J\geq 1$:
\begin{gather}
\label{Pyt1}
\lim_{n\to\infty}\left\|u_{0,n}\right\|^2_{\hdot}-\left(\sum_{j=1}^J\left\|U_{L,n}^j(0)\right\|^2_{\hdot}+\left\|w_n^J(0)\right\|^2_{\hdot}\right)=0\\
\label{Pyt2}
\lim_{n\to\infty}\left\|u_{1,n}\right\|^2_{L^2}-\left(\sum_{j=1}^J\left\|\partial_tU_{L,n}^j(0)\right\|^2_{L^2}+\left\|\partial_tw_n^J(0)\right\|^2_{L^2}\right)=0\\
 \label{Pyt3}
 \lim_{n\to\infty}\left\|u_{0,n}\right\|^{\frac{2N}{N-2}}_{L^{\frac{2N}{N-2}}}-\left(\sum_{j=1}^J\left\|U_{L,n}^j(0)\right\|^{\frac{2N}{N-2}}_{L^{\frac{2N}{N-2}}}+\left\|w_n^J(0)\right\|^{\frac{2N}{N-2}}_{L^{\frac{2N}{N-2}}}\right)=0.
\end{gather}
Let us emphasize the fact that \eqref{prop_wnJ} is essential for these Pythagorean expansions to hold.  
\begin{remark}
\label{R:weak}
Profiles can be expressed in terms of weak limits. More precisely, for all $j$,
$$ \vS_L\left(\frac{t_{j,n}}{\lambda_{j,n}}\right) \left(\lambda_{j,n}^{\frac{N}{2}-1} u_{0,n}\left(0,\lambda_{j,n}\cdot+x_{j,n}\right),\lambda_{j,n}^{\frac{N}{2}} u_{1,n}\left(\lambda_{j,n}\cdot+x_{j,n}\right)\right) \xrightharpoonup[n\to\infty]{}\vU_L^j(0)\text{ in }\hdot\times L^2.$$
Indeed this follows from the orthogonality \eqref{orthogonal} of the parameters, and the property:
$$ J\geq j\Longrightarrow \left(\lambda_{j,n}^{\frac{N}{2}-1} w_{n}^J\left(t_{j,n},\lambda_{j,n}\cdot+x_{j,n}\right),\lambda_{j,n}^{\frac{N}{2}} \partial_tw_{n}^J\left(\lambda_{j,n}\cdot+x_{j,n}\right)\right) \xrightharpoonup[n\to\infty]{}0\text{ in }\hdot\times L^2,$$
which is an easy consequence of \eqref{orthogonal} and \eqref{prop_wnJ}.
\end{remark}
It is possible to modify the profiles and parameters of a profile decomposition:
\begin{lemma}
\label{L:modif_profile}
 Let as before $\profiles$ be a profile decomposition of the sequence $\{(u_{0,n},u_{1,n})\}_n$. For all $j\geq 1$, consider sequences $\left\{\tlambda_{j,n},\tx_{j,n},\ttt_{j,n}\right\}_n$  in $(0,+\infty)\times \RR^N\times \RR$ such that for all $j\geq 1$, there exists $(\mu_j,y_j,s_j) \in (0,+\infty)\times \RR^N\times \RR$ such that
 \begin{equation}
  \label{CV_param} \lim_{n\to\infty} \frac{\tlambda_{j,n}}{\lambda_{j,n}}=\mu_j,\quad \lim_{n\to\infty}\frac{\tx_{j,n}-x_{j,n}}{\lambda_{j,n}}=y_j,\quad\lim_{n\to\infty} \frac{\ttt_{j,n}-t_{j,n}}{\lambda_{j,n}}=s_j.
 \end{equation} 
 Let 
 $$\tU^j_L(t,x)=\mu_j^{\frac{N}{2}-1}U^j_L(s_j+\mu_jt,y_j+\mu_jx).$$
 Then $\tprofiles$ is also a profile decomposition for the sequence $\{(u_{0,n},u_{1,n})\}_n$.
\end{lemma}
\begin{proof}[Sketch of proof]
 The fact that the sequences $\left\{\tlambda_{j,n},\tx_{j,n},\ttt_{j,n}\right\}_n$, $j\geq 1$ are orthogonal follows easily from the orthogonality of the sequences $\left\{\lambda_{j,n},x_{j,n},t_{j,n}\right\}_n$, $j\geq 1$ and \eqref{CV_param}.
 
 Recall from \eqref{linear_profile} the definition of $U^j_{L,n}$, and define similarly $\tU^j_{L,n}$. It is sufficient to show:
 \begin{equation}
 \label{conclu_lemma}
 \forall j\geq 1,\quad \lim_{n\to\infty}\left\| \vU^{j}_{L,n}(0)-\vec{\tU}\vphantom{U}^j_{L,n}(0)\right\|_{\hdot\times L^2}=0. 
 \end{equation} 
 Noting that (by conservation of the linear energy) 
 $$\left\| \vU^{j}_{L,n}(0)-\vec{\tU}\vphantom{U}^j_{L,n}(0)\right\|_{\hdot\times L^2}=\left\| \vU^{j}_{L,n}(t_{j,n})-\vec{\tU}\vphantom{U}^j_{L,n}(t_{j,n})\right\|_{\hdot\times L^2},$$
 \eqref{conclu_lemma}
follows from \eqref{CV_param} and direct computations.
 \end{proof}
We next state a uniqueness result, proving that the transformations of Lemma \ref{L:modif_profile} describe (up to permutations of the indices) all the possible profile decompositions of a given sequence.
\begin{lemma}
\label{L:uniq_profile}
 Let $\profiles$ and $\tprofiles$ be two profile decompositions of the same sequence $\{(u_{0,n},u_{1,n})\}_n$. Assume that each of the sets $$\JJJ=\big\{j\geq 1,\; U_{L}^j\neq 0\big\},\quad \KKK=\big\{k\geq 1,\; \tU_{L}^k\neq 0\big\}$$
 is finite or equal to $\NN\setminus\{0\}$. Then, extracting sequences (in $n$) if necessary there exists a unique one-to-one map 
 $$\varphi:\NN\setminus\{0\}\longrightarrow \NN\setminus\{0\}$$
 with the following property. For all $j\geq 1$, letting $k=\varphi(j)$, then $U^j_L=0$ if and only if $\tU^k_L=0$. Furthermore, if $U^j_L\neq 0$, there exists $(\mu_j,y_j,s_j) \in (0,+\infty)\times \RR^N\times \RR$ such that
\begin{equation}
  \label{CV_param'} \lim_{n\to\infty} \frac{\tlambda_{k,n}}{\lambda_{j,n}}=\mu_j,\quad \lim_{n\to\infty}\frac{\tx_{k,n}-x_{j,n}}{\lambda_{j,n}}=y_j,\quad\lim_{n\to\infty} \frac{\ttt_{k,n}-t_{j,n}}{\lambda_{j,n}}=s_j,
 \end{equation} 
 and 
 \begin{equation}
\label{link_profile}
 \tU^k_L(t,x)=\mu_j^{\frac{N}{2}-1}U^j_L(s_j+\mu_jt,y_j+\mu_jx).  
 \end{equation} 
 \end{lemma}
\begin{proof}[Sketch of proof]
 Let $j\geq 1$ such that $U_L^j\neq 0$. We first prove that there exists $k\in \KKK$ such that the sequence
 \begin{equation}
 \label{non_ortho}
 \left\{\left|\log \frac{\tlambda_{k,n}}{\lambda_{j,n}}\right|+\frac{\left|t_{j,n}-\ttt_{k,n}\right|}{\lambda_{j,n}}+\frac{\left|x_{j,n}-\tx_{k,n}\right|}{\lambda_{j,n}}\right\}_n
 \end{equation} 
 does not tend to $+\infty$ as $n\to\infty$. If not, we have
 \begin{equation}
  \label{ortho+}
  \forall k\in \KKK,\quad 
  \lim_{n\to\infty} \left|\log \frac{\tlambda_{k,n}}{\lambda_{j,n}}\right|+\frac{\left|t_{j,n}-\ttt_{k,n}\right|}{\lambda_{j,n}}+\frac{\left|x_{j,n}-\tx_{k,n}\right|}{\lambda_{j,n}}=+\infty.
 \end{equation} 
 Since $\left(\tU^k_L,\left\{\tlambda_{k,n},\tx_{k,n},\ttt_{k,n}\right\}_n\right)_{k\geq 1}$
 is a profile decomposition for the sequence $\left\{(u_{0,n},u_{1,n})\right\}_n$, we get immediately
 (changing if necessary one of the sequences of parameters $(\tlambda_{k,n},\tx_{k,n},\ttt_{k,n})_n$, where $k\notin \KKK$ to preserve the orthogonality of the parameters) that
 $$ \left(-U^j_L,\left\{\lambda_{j,n},x_{j,n},t_{j,n}\right\}_n\right)\cup \left(\tU^k_L,\left\{\tlambda_{k,n},\tx_{k,n},\ttt_{k,n}\right\}_n\right)_{k\geq 1}$$
 is a profile decomposition for the sequence $\left\{(u_{0,n},u_{1,n})-\vU^j_{L,n}(0)\right\}_n$. Furthemore,
 $$\left(U^k_L,\left\{\lambda_{k,n},x_{k,n},t_{k,n}\right\}_n\right)_{\substack{k\geq 1\\k\neq j}}$$
 is a profile decomposition for the same sequence $\left\{(u_{0,n},u_{1,n})-\vU^j_{L,n}(0)\right\}_n$.

 Using the Pythagorean expansions \eqref{Pyt1}, \eqref{Pyt2} for the above profile decompositions and for the profile decomposition  $\left(\tU^k_L,\left\{\tlambda_{k,n},\tx_{k,n},\ttt_{k,n}\right\}_n\right)_{k\geq 1}$ of $\left\{(u_{0,n},u_{1,n})\right\}_n$, one can prove that $U^j_L=0$, a contradiction.
 
 Next, we notice that by orthogonality of the parameters $\left\{\tlambda_{\ell,n},\tx_{\ell,n},\ttt_{\ell,n}\right\}_n$, $\ell\geq 1$, the index $k$ such that \eqref{non_ortho} holds is unique. We define $\varphi(j)=k$. Extracting subsequences in $n$, we can assume that \eqref{CV_param'} holds. 
Define $u_{L,n}(t)=S_{L}(t)(u_{0,n},u_{1,n})$ and
$$ u_{L,n}^j(t,x)=\lambda_{j,n}^{\frac{N}{2}-1}u_L\left(\lambda_{j,n}t+t_{j,n},\lambda_{j,n}x+x_{j,n}\right),\quad \tu_{L,n}^k(t,x)=\tlambda_{k,n}^{\frac{N}{2}-1}u_L\left(\tlambda_{k,n}t+\ttt_{k,n},\tlambda_{k,n}x+\tx_{k,n}\right).$$
 Then:
 $$ u_{L,n}^j(0)\xrightharpoonup[n\to\infty]{} \vU_L^j(0),\quad \tu_{L,n}^k(0)\xrightharpoonup[n\to\infty]{} \vec{\tU}\vphantom{\vU}_L^k(0)$$
 weakly in $\hdot\times L^2$. From this and \eqref{CV_param'}, we deduce \eqref{link_profile}. We have constructed a map:
 $\varphi:\JJJ\to\KKK$,
 such that for all $j\in \JJJ$, if $k=\varphi(j)$, the limits in \eqref{CV_param'} exist.
 The same construction gives a map $\psi$ from $\KKK$ to $\JJJ$ with an analogous property, and one sees easily that $\psi$ is the inverse function of $\varphi$, and thus that $\varphi$ is a bijection between $\JJJ$ and $\KKK$. This proves that $\JJJ=\KKK=\NN\setminus\{0\}$, or $\JJJ$ and $\KKK$ are finite, with the same cardinal. In this second case, one extends arbitrarily $\varphi$ to a bijection from $\NN\setminus\{0\}$ to $\NN\setminus \{0\}$. In both cases, the conclusion of the lemma is satisfied.
 \end{proof}

\subsection{Reordering the profiles}
\label{SS:preorder}
Let $\left\{(u_{0,n},u_{1,n})\right\}_n$ be a bounded sequence in $\hdot\times L^2$ with a profile decomposition $\left(U_L^j,\left\{\lambda_{j,n},x_{j,n},t_{j,n}\right\}_n\right)_{j\geq 1}$. We will see in the next subsection that this profile decomposition gives, for large $n$, an approximation of the solution $u_n$ of \eqref{CP} with initial data $\left( u_{0,n},u_{1,n} \right)$ on a time interval $[0,\tau_n]$ which depends on the profiles $U^j_L$ and the sequences of parameters $\{\lambda_{j,n},t_{j,n}\}_n$. In this subsection, we reorder the profiles $U_L^j$ so that one can choose the length $\tau_n$ of the interval of approximation in term of the first profile $U_L^1$ and the associated sequence $\left\{\lambda_{1,n},t_{1,n}\right\}_n$ of parameters. 

The approximation of $u_n$ will be given in term of nonlinear profiles that are defined as follows.
\begin{defi}
 Let $j\geq 1$. A nonlinear profile $U^j$ associated to the linear profile $U^j_L$ and the sequence of parameters $\{\lambda_{j,n},t_{j,n}\}_n$ is a solution $U^j$ of \eqref{CP} such that for large $n$, $-t_{j,n}/\lambda_{j,n}\in I_{\max}(U^j)$ and 
 $$ \lim_{n\to \infty} \left\|\vU^j_L\left(  \frac{-t_{j,n}}{\lambda_{j,n}}\right)-\vU^j\left(  \frac{-t_{j,n}}{\lambda_{j,n}}\right)\right\|_{\hdot\times L^2}=0.$$
\end{defi}
Extracting subsequences, we can always assume that for all $j\geq 1$, the following limit exists:
\begin{equation}
 \label{Z6}
 \lim_{n\to\infty}\frac{-t_{j,n}}{\lambda_{j,n}}=\sigma_j\in [-\infty,+\infty]. 
 \end{equation}
Using the local Cauchy theory of \eqref{CP} if $\sigma_j\in \RR$, and the existence of wave operators for \eqref{CP} if $\sigma_j\in \{-\infty,+\infty\}$, we obtain that for all $j$, there exists a unique nonlinear profile $U^j$ associated to $U^j_L$ and $\{\lambda_{j,n},t_{j,n}\}_n$. If $\sigma_j\in \RR$, then $\sigma_j\in (T_-(U^j),T_+(U^j))$. 
If $\sigma_j=-\infty$, then $T_-(U^j)=-\infty$ and $U^j$ scatters backward in time. Finally, if $\sigma_j=+\infty$, then $T_+(U^j)=+\infty$ and $U^j$ scatters forward in time. 

Denoting by
\begin{equation}
 \label{def_Ujn}
 U_{n}^j(t,x)=\frac{1}{\lambda_{j,n}^{\frac{N}{2}-1}}U^j\left(\frac{t-t_{j,n}}{\lambda_{j,n}},\frac{x-x_{j,n}}{\lambda_{j,n}}\right),
\end{equation} 
we see that the maximal positive time of existence of $U_n^j$ is exactly $\lambda_{j,n}T_+(U^j)+t_{j,n}$ (or $+\infty$ if $T_+(U^j)=+\infty$).

\begin{notation}
\label{N:preorder}
If $j$ and $k$ are indices, we write 
$$ \left(U^{j}_L,\{t_{j,n},\lambda_{j,n}\}_n\right)\preceq \left(U^{k}_L,\{t_{k,n},\lambda_{k,n}\}_n\right)$$
if one of the following holds
\begin{enumerate}
 \item the nonlinear profile $U^k$ scatters forward in time or 
\item the nonlinear profile $U^j$ does not scatter forward in time and
\begin{equation}
 \label{Z7}
 \forall T\in \RR,\quad T<T_+(U^j)\Longrightarrow \lim_{n\to\infty} \frac{\lambda_{j,n}T+t_{j,n}-t_{k,n}}{\lambda_{k,n}}<T_+(U^k).
\end{equation} 
\end{enumerate}
If there is no ambiguity in the choice of the profile decomposition, we will simply write
$$(j)\preceq (k).$$
We write $(j)\backsimeq(k)$ if $(j)\preceq (k)$ and $(k)\preceq (j)$, and $(j)\prec (k)$ if $(k)\preceq (j)$ does not hold. As usual, we will extract subsequences so that the limit appearing in \eqref{Z7} exists for all $j$, $k$ and $T<T_+(U^j)$ (see Claim \ref{C:subsequence} in the appendix).
\end{notation}
Note that if $U^{j_0}$ scatters forward in time, then $(j)\preceq (j_0)$ for all $j\geq 1$. Note also that 
\begin{equation}
 \label{imply_scattering}
 \Big((j)\preceq (k)\text{ and } U^j\text{ scatters forward in time}\Big)\Longrightarrow U^k \text{ scatters forward in time}.
\end{equation} 

If $U^{k}$ does not scatter forward in time and $U^j$ scatters forward in time, then $(j)\preceq (k)$ does not hold. 

The relation $(j)\preceq (k)$ is equivalent to the fact that if for a sequence of positive times $\{\tau_n\}_n$, the sequence $\left\{\|U^j_n\|_{S(0,\tau_n)}\right\}_n$ is bounded, then the sequence $\left\{\|U^k_n\|_{S(0,\tau_n)}\right\}_n$ is also bounded. 
\begin{claim}
\label{C:Z2}
 One can extract subsequences in $n$ so that the binary relation $\preceq$ is a total preorder on the set of indices. In other words
 \begin{enumerate}
  \item \label{I:Z2_1}$\forall j\geq 1$, $(j)\backsimeq (j)$. 
  \item \label{I:Z2_2}$\forall j,k,\ell\geq 1$, $\Big((j)\preceq (k)$ and $(k)\preceq (\ell)\Big)$ $\Longrightarrow$ $(j)\preceq (\ell)$.
 \item \label{I:Z2_3}$\forall j,k\geq 1$, $(j)\preceq (k)$ or $(k)\preceq (j)$. 
\item \label{I:Z2_4}$(j)\prec (k)$ is equivalent to $\Big[(j)\preceq (k)$ and not$\,\big((j)\backsimeq (k)\big)\Big]$. 
 \end{enumerate}
\end{claim}
(Claim \ref{C:Z2} is proved in Appendix \ref{A:preorder}).
\begin{defi}
\label{D:W0}
 We say that the profile decomposition is \emph{well-ordered} if 
$$\forall j\geq 1,\quad (j)\preceq (j+1).$$
\end{defi}

Note as a consequence of the Pythagorean expansions \eqref{Pyt1}, \eqref{Pyt2} and the boundedness of the sequence $\left\{(u_{0,n},u_{1,n})\right\}_n$ in $\hdot\times L^2$, $\vU^k_L(0)$ is small in $\hdot \times L^2$ for large $k$, and thus, by the small data global well-posedness theory for \eqref{CP}, $\vU^k$ scatters in both time directions. We deduce: 
$$ \exists k_0,\;\forall k\geq k_0,\;\forall j\geq 1,\quad (j)\preceq (k).$$
Combining with Claim \ref{C:Z2}, we get:
\begin{claim}
 \label{C:Z3}
For any sequence $\left\{(u_{0,n},u_{1,n})\right\}_n$, bounded in $\hdot\times L^2$, with a profile decomposition $\left(U^{j}_L,\left\{\lambda_{j,n},x_{j,n},t_{j,n}\right\}_n\right)_j$, there exists a subsequence
and a one-to-one map $\varphi$ of $\NN\setminus\{0\}$ such that $\left(U^{\varphi(j)}_L,\left\{\lambda_{\varphi(j),n},x_{\varphi(j),n},t_{\varphi(j),n}\right\}_n\right)_j$ is a \emph{well-ordered} profile decomposition of $\left\{(u_{0,n},u_{1,n})\right\}_n$.
\end{claim}
We conclude this section by stating that the preorder  relation ``$\preceq$'' is invariant by the transformations described in Lemma \ref{L:modif_profile}:
\begin{lemma}
\label{L:intrinsec}
 Let $\profiles$, $\tprofiles$ be as in Lemma \ref{L:modif_profile}. Then, denoting as usual by $U^j$, $\tU^j$ the corresponding nonlinear profiles, we have
 \begin{equation}
 \label{link_NL} 
 \tU^j(t,x)=\mu_j^{\frac{N}{2}-1}U^j(s_j+\mu_jt,y_j+\mu_jx).
 \end{equation} 
 Furthermore,
 $$\forall j\geq 1,\quad \left(U^j_L,\{\lambda_{j,n},t_{j,n}\}_n\right)_j\backsimeq \left(\tU^j_L,\{\tlambda_{j,n},\ttt_{j,n}\}_n\right)_j.$$
 In particular $\profiles$ is well-ordered if and only if $\tprofiles$ is well-ordered.
\end{lemma}
\begin{proof}[Sketch of proof]
The equality \eqref{link_NL} follows easily from
$$\tU^j_L(t,x)=\mu_j^{\frac N2-1}U^j_L(s_j+\mu_jt,\mu_jx+y_j),$$
and the limits \eqref{CV_param}.
By \eqref{link_NL},
 $T_+\big(\tU^j\big)=\frac{1}{\mu_j}\left(T_+\big(U^j\big)-s_j\right)$, with the convention that $T_+\big(\tU^j\big)=+\infty$ if $T_+\big(U^j\big)=+\infty$. The conclusion of the lemma is then easy to check, using \eqref{CV_param} and the definition of ``$\preceq$''.
\end{proof}

\subsection{Nonlinear approximation}
\label{SS:approx}
We next state the announced nonlinear approximation result. As above $\left\{(u_{0,n},u_{1,n})\right\}_n$ is a bounded sequence of $\hdot\times L^2$, which admits a profile decomposition $\left(U^j_L,\left\{\lambda_{j,n},x_{j,n},t_{j,n}\right\}_n\right)_{j\geq 1}$. We denote by $U^j$ the corresponding nonlinear profiles (see the preceding subsection). We start with the case where all nonlinear profiles scatter forward in time. 
\begin{prop}
\label{P:Z3'}
 Assume that for all $j$, $U^j$ scatters forward in time. Then, for large $n$,  $u_n$ scatters forward in time and,
 letting 
 $$r_n^J(t)=u_n(t)-\sum_{j=1}^JU_n^j(t,x)-w_n^J(t),\quad t\geq 0,$$
 we have
 $$\lim_{J\to\infty}\limsup_{n\to\infty} \left(\left\|r_n^J\right\|_{S(0,+\infty)}+\sup_{t\geq 0} \left\|r_n^J(t)\right\|_{\hdot\times L^2}\right)=0.$$
\end{prop}
Proposition \ref{P:Z3'} is standard: see for example the main result of \cite{BaGe99}.

In the general case, the approximation of Proposition \ref{P:Z3'} does not hold for all positive times. In this case Proposition \ref{P:Z4} gives an approximation on a time interval that might depend on $n$:
\begin{prop}
\label{P:Z4}
 Assume that the profiles $U^j$ are reordered as in Claim \ref{C:Z3}. Let $T<T_+(U^1)$. Let 
 $$\tau_n=\lambda_{1,n}T+t_{1,n},$$
 and assume that $\tau_n>0$ for large $n$. Then for large $n$, $[0,\tau_n]\subset I_{\max}(u_n)$ and, for all $j$ and large $n$, $[0,\tau_n]\subset I_{\max}(U_n^j)$. Furthermore, letting 
 $$r_n^J(t)=u_n(t)-\sum_{j=1}^JU_n^j(t,x)-w_n^J(t),\quad t\in [0,\tau_n],$$
 we have
 $$\lim_{J\to\infty}\limsup_{n\to\infty} \left(\left\|r_n^J\right\|_{S(0,\tau_n)}+\sup_{t\in [0,\tau_n]} \left\|r_n^J(t)\right\|_{\hdot\times L^2}\right)=0.$$
\end{prop}
\begin{remark}
 The statement of Proposition \ref{P:Z4} is not empty only when there exists $T<T_+(U^1)$ such that $\lambda_{1,n}T+t_{1,n}>0$ for large $n$. This always holds if $\sigma_1 \in \RR\cup\{-\infty\}$ (where $\sigma_1$ is defined by \eqref{Z6}).  If $ \sigma_1=+\infty$, the relation $(1)\preceq (j)$ for all $j$ implies that all nonlinear profiles $U^j$ scatter forward in time, and we are in the setting of Proposition \ref{P:Z3'}. 
\end{remark}
\begin{remark}
 We have stated Propositions \ref{P:Z3'} and \ref{P:Z4} for positive time. Of course the analogs of these propositions for negative time also hold. However, note that the definition of the order relation on the profile must be adapted. 
\end{remark}
\begin{proof}[Sketch of proof of Proposition \ref{P:Z4}]
 Let $j\geq 2$. By definition of $U^j$, $-t_{j,n}/\lambda_{j,n}\in I_{\max}(U^j)$ for large $n$. Since $(1)\preceq (j)$, we know that $U^j$ scatters forward in time, or 
 $$ \lim_{n\to\infty} \frac{\tau_n-t_{j,n}}{\lambda_{j,n}}<T_+(U^j).$$
 In both cases, $[0,\tau_n]\subset I_{\max}(U_n^j)$ for large $n$ and 
\begin{equation}
\label{S_norm_Uj} 
\limsup_{n\to\infty} \|U_n^j\|_{S(0,\tau_n)}=\limsup_{n\to\infty} \|U^j\|_{S\left(\frac{-t_{j,n}}{\lambda_{j,n}},\frac{\tau_n-t_{j,n}}{\lambda_{j,n}}\right)} <\infty. 
\end{equation} 
 Proposition \ref{P:Z4} then follows by a standard approximation result which is a consequence of the long-time perturbation theory of equation \eqref{CP} (see e.g. \cite[Proposition 2.8]{DuKeMe11a}).
\end{proof}
\begin{remark}
 The proof of Proposition \ref{P:Z4} uses the following Pythagorean expansions of the norms of the profiles in $S$, which follows from the orthogonality \eqref{orthogonal}   of the sequences of parameters (see \cite{BaGe99}):
 \begin{equation}
  \label{Pyt_S}
  \forall J\geq 1,\quad 
  \lim_{n\to\infty} \left(\Big\|\sum_{j=1}^J U_n^j\Big\|_{S(0,\tau_n)}^{\frac{2(N+1)}{N-2}}-\sum_{j=1}^J\left\|U^j_n\right\|_{S(0,\tau_n)}^{\frac{2(N+1)}{N-2}}\right)=0.
 \end{equation} 
\end{remark}
\begin{remark}
\label{R:profiles}
 In the setting of Proposition \ref{P:Z4}, let $\{\tau'_n\}_n$ be a sequence of times such that 
 $$\forall n,\quad \tau'_n\in [0,\tau_n].$$ 
 By Proposition \ref{P:Z4}, $\vu_n(\tau'_n)$ is well defined for large $n$, and $\{\vu_n(\tau'_n)\}_n$ is a bounded sequence of $\hdot\times L^2$. Extracting subsequences, we can get from Proposition \ref{P:Z4} a profile decomposition of this sequence. 

More precisely, note that 
 $$ \vU^j_n(\tau'_n)=\left(\frac{1}{\lambda_{j,n}^{\frac{N}{2}-1}}U^j\left( \frac{\tau'_n-t_{j,n}}{\lambda_{j,n}},\frac{x-x_{j,n}}{\lambda_{j,n}} \right),\frac{1}{\lambda_{j,n}^{\frac{N}{2}}}\partial_t U^j\left( \frac{\tau'_n-t_{j,n}}{\lambda_{j,n}},\frac{x-x_{j,n}}{\lambda_{j,n}} \right)\right)$$
Let $s'_{j,n}=\frac{\tau'_n-t_{j,n}}{\lambda_{j,n}}$.
 Extracting subsequences, we can assume that $s'_{j,n}$ has a limit $\theta_j\in \RR\cup \{\pm\infty\}$ as $n\to+\infty$. Observe that there exists a unique solution $V^j_L$ of the linear wave equation \eqref{LCP} such that 
$$\lim_{n\to\infty} \left\|\vV_L^j\left(s_{j,n}'\right)-\vU^j\left(s_{j,n}'\right)\right\|_{\hdot\times L^2}=0.$$
Indeed, if $\theta_j\in \RR$, it follows from \eqref{S_norm_Uj} that $\theta_j\in I_{\max}(U^j)$, and $V_{L}^j$ is the solution of the linear wave equation with initial data $\vU^j(\theta_j)$ at $t=\theta_j$. On the other hand, if $\theta_j=+\infty$ (respectively $-\infty$), then by \eqref{S_norm_Uj}, $U^j$ scatters forward in time (respectively backward in time), and the existence of $V_{L}^j$ follows.

Letting $t'_{j,n}=-\tau'_n+t_{j,n}$,
it is easy to check that the sequences of parameters $\{\lambda_{j,n},x_{j,n},t_{j,n}'\}_n$, $j\geq 1$, are orthogonal in the sense that \eqref{orthogonal} holds. In view of Proposition \ref{P:Z4}, we deduce that $\left(V^j_{L},\{\lambda_{j,n},x_{j,n},t_{j,n}'\}_n\right)_{j\geq 1}$ is a profile decomposition for the sequence $\vu_n(\tau'_n)$. Note that the nonlinear profiles for this decomposition are by definition exactly the nonlinear profiles $U^j$ of the profile decomposition $\profiles$ of $\{\vu_n(0)\}_n$.
\end{remark}
\subsection{Double profile decomposition}
\label{SS:DPD}
We conclude this section by showing the following technical lemma, which will be very useful in the proof of Theorem \ref{T:profiles}.
\begin{lemma}
\label{L:DPD}
 Let $M>0$. Let $\{m_p\}_p$ be a sequence of natural numbers. For any $p\in \NN$, we consider a sequence $\left\{\left(u_{0,n}^p,u_{1,n}^p\right)\right\}_n$ in $\hdot \times L^2$ such that
\begin{equation}
\label{DPD1}
 \limsup_{n\to \infty}\left\|\left(u_{0,n}^p,u_{1,n}^p\right)\right\|_{\hdot\times L^2}\leq M.
\end{equation}
Assume that for all $p$,  $\left\{\left(u_{0,n}^p,u_{1,n}^p\right)\right\}_n$ has a profile decomposition $\left(U_L^{p,j},\left\{\lambda_{p,j,n},x_{p,j,n},t_{p,j,n}\right\}_n\right)_{j\geq 1}$ and that there exists $(\eta_j)_{j\geq 1}$ such that 
\begin{equation}
 \label{DPD2}
\sum_{j\geq 1}\eta_j<\infty\quad \text{and}\quad \forall p,j,\quad \left\|U^{p,j}_L\right\|^{\SN}_{S(\RR)}\leq \eta_j.
\end{equation}  
Assume also that for all $j$, the sequence $\left\{ \vU_L^{p,j}(0)\right\}_{p\geq 1}$ has a profile decomposition 
$$\left(V_L^{j,k},\left\{ \mu_{j,k,p},y_{j,k,p},s_{j,k,p}\right\}_p\right)_{k\geq 1}.$$
 Then there exists an increasing sequence of integers $\{n_p\}_p$ satisfying
\begin{equation}
\label{prop_np} 
\forall p, \quad n_p\geq m_p\quad \text{and} \quad \lim_{p\to\infty} n_p=+\infty
\end{equation} 
and such that the sequence $\left\{\left(u_{0,n_p}^p,u_{1,n_p}^p\right)\right\}_p$ has the profile decomposition 
$$ \left(V_L^{j,k},\left\{ \nu_{j,k,p,n_p},z_{j,k,p,n_p},\tau_{j,k,p,n_p}\right\}_p\right\}_{j,k\geq 1}$$
with
\begin{equation}
 \label{def_param}
\nu_{j,k,p,n}=\lambda_{p,j,n}\mu_{j,k,p},\quad z_{j,k,p,n}=x_{p,j,n}+\lambda_{p,j,n}y_{j,k,p},\quad \tau_{j,k,p,n}=t_{p,j,n}+\lambda_{p,j,n}s_{j,k,p}.
\end{equation} 
\end{lemma}

\begin{proof}
We will use the following claim, proved in appendix \ref{A:profile} (recall \eqref{orthogonal} for the definition of orthogonality and \eqref{linear_profile} for the notation $U^j_{L,n}$).
\begin{claim}
\label{C:profile}
Let $(U^j_L)_{j\geq 1}$ be a family of solutions of the linear wave equation, and $\left\{\lambda_{j,n},x_{j,n},t_{j,n}\right\}_n$, $j\geq 1$, a family of \emph{orthogonal} sequences of parameters. Assume
\begin{equation}
 \label{uniform_UjL}
\sum_{j\geq 1}\left\|U^{j}_L\right\|^{\SN}_{S(\RR)}<\infty.
\end{equation} 
Let $\left\{\left(u_{0,n},u_{1,n}\right)\right\}_{n}$ be a bounded sequence in $\hdot\times L^2$, and $u_{L,n}(t)=S_{L}(t)(u_{0,n},u_{1,n})$. Assume that there exists a sequence $\{J_k\}_k$ of integers such that $\lim_k J_k=+\infty$ and
$$\lim_{k\to\infty} \limsup_{n\to\infty}\left\|u_{L,n}-\sum_{j=1}^{J_k}U^j_{L,n}\right\|_{S(\RR)}=0.$$
Then 
\begin{equation}
\label{profile_UjL}
 \lim_{J\to\infty} \limsup_{n\to\infty}\left\|u_{L,n}-\sum_{j=1}^{J}U^j_{L,n}\right\|_{S(\RR)}=0,
\end{equation} 
i.e. $\left(U^j_L,\left\{\lambda_{j,n},x_{j,n},t_{j,n}\right\}_n\right)_j$ is a profile decomposition for the sequence $\left\{(u_{0,n},u_{1,n})\right\}_n$.
\end{claim} 
 Let $u_{L,n}^p(t,x)=S_L(t)\left(u_{0,n}^p,u_{1,n}^p\right)(x)$. 

\EMPH{Step 1. Formal expansion of $u_{L,n}^p$}

By the definition of a profile decomposition,
\begin{equation}
 \label{DPD3} u_{L,n}^p(t,x)=\sum_{j=1}^J \frac{1}{\lambda_{p,j,n}^{\frac N2-1}} U_L^{p,j}\left(\frac{t-t_{p,j,n}}{\lambda_{p,j,n}},\frac{x-x_{p,j,n}}{\lambda_{p,j,n}}\right)+w_{p,n}^J,
\end{equation} 
where
\begin{equation}
 \label{DPD4}
\lim_{J\to\infty} \limsup_{n\to\infty} \left\|w_{p,n}^J\right\|_{S(\RR)}=0.
\end{equation} 
For $j\geq 1$, we have
\begin{equation}
 \label{DPD5}
U_{L}^{p,j}(s,y)=\sum_{k=1}^K \frac{1}{\mu_{j,k,p}^{\frac{N}{2}-1}}V_L^{j,k}\left(\frac{s-s_{j,k,p}}{\mu_{j,k,p}},\frac{y-y_{j,k,p}}{\mu_{j,k,p}}\right)+W_{j,p}^K,
\end{equation} 
where
\begin{equation}
 \label{DPD6}
\lim_{K\to\infty}\limsup_{p\to\infty}\left\| W_{j,p}^K\right\|_{S(\RR)}=0.
\end{equation} 
Combining \eqref{DPD3} and \eqref{DPD5}, we obtain
\begin{equation}
 \label{DPD7}
u_{L,n}^p(t,x)=\sum_{j=1}^J \left(\sum_{k=1}^K\frac{1}{\nu_{j,k,p,n}^{\frac{N}{2}-1}}V_L^{j,k}\left(\frac{s-\tau_{j,k,p,n}}{\nu_{j,k,p,n}},\frac{y-z_{j,k,p,n}}{\nu_{j,k,p,n}}\right)\right)+A^{J,K}_{p,n}(t,x)+w_{p,n}^J(t,x),
\end{equation} 
where $\nu_{j,k,p,n},z_{j,k,p,n}$, and $\tau_{j,k,p,n}$ are defined in \eqref{def_param} and
\begin{equation}
\label{defApn}
A_{p,n}^{J,K}(t,x)=\sum_{j=1}^J\frac{1}{\lambda_{p,j,n}^{\frac{N}{2}-1}}W_{j,p}^K\left(\frac{t-t_{p,j,n}}{\lambda_{p,j,n}},\frac{x-x_{p,j,n}}{\lambda_{p,j,n}}\right).
\end{equation} 
\EMPH{Step 2} We prove
\begin{equation}
 \label{DPD8}
\forall \eps>0,\; \exists J_{\eps}>0,\; \forall J\geq J_{\eps},\;\forall p,\quad \limsup_{n\to\infty}\left\|w_{p,n}^J\right\|_{S(\RR)}\leq \eps.
\end{equation} 
Indeed, let $\eps>0$ and $J_{\eps}$ such that
\begin{equation}
 \label{DPD9}
\sum_{j\geq J_{\eps}} \eta_j\leq \left(\frac{\eps}{2}\right)^{\SN}.
\end{equation} 
Let $J\geq J_{\eps}$ and $p\geq 1$. By \eqref{DPD4}, there exists $\tJ=\tJ(p)\geq J_{\eps}$ such that
\begin{equation}
 \label{DPD10}
\forall J'\geq \tJ,\quad \limsup_{n\to\infty} \left\|w_{p,n}^{J'}\right\|_{S(\RR)}\leq \frac{\eps}{2}.
\end{equation} 
If $J\geq \tJ$, then by \eqref{DPD10}, $\limsup_{n\to\infty}\left\|w_{p,n}^J\right\|_{S(\RR)}<\eps$. 

Assume $J_{\eps}\leq J<\tJ$. Then 
$$w_{p,n}^J(t,x)=w_{p,n}^{\tJ}(t,x)+\underbrace{\sum_{j=J+1}^{\tJ}\frac{1}{\lambda_{p,j,n}^{\frac N2-1}} U_L^{p,j}\left(\frac{t-t_{p,j,n}}{\lambda_{p,j,n}},\frac{x-x_{p,j,n}}{\lambda_{p,j,n}}\right)}_{(I)_{p,n}}.$$
Using the orthogonality of the sequences of parameters $\left\{\lambda_{p,j,n},x_{p,j,n},t_{p,j,n}\right\}_n$ (at fixed $p$) and \eqref{DPD9}, we obtain
\begin{equation}
 \label{DPD11}
\limsup_{n\to\infty} \left\| (I)_{p,n}\right\|_{S(\RR)}^{\SN}=\sum_{j=J+1}^{\tJ} \left\|U^{p,j}_{L}\right\|_{S(\RR)}^{\SN}\leq\left(\frac{\eps}{2}\right)^{\SN}.
\end{equation}
Combining with \eqref{DPD10}, we deduce  $\limsup_{n\to\infty}\left\|w_{p,n}^J\right\|_{S(\RR)}\leq \eps$.

\EMPH{Step 3. Choice of $n_p$}
Let for $J\geq 1$,
\begin{equation}
 \label{defEpsJ}
\eps_{J}=\max_p\left(\limsup_{n\to\infty} \left\|w_{p,n}^J\right\|_{S(\RR)}\right).
\end{equation} 
By Step 2, $\eps_J<\infty$ for large $J$ and $\lim_J \eps_{J}=0$. In this step we prove that there exists an increasing sequence of integer $\{n_p\}_{p\geq 1}$ such that \eqref{prop_np} holds and, for all $p$, 
\begin{gather}
 \label{DPD12}
\forall J\in \{1,\ldots,p\}, \quad \left\|w_{p,n_p}^J\right\|_{S(\RR)}\leq 2\eps_J\\
\label{DPD13}
\forall (j,k,j',k'),\quad \left(1\leq j,k,j',k'\leq p\text{ and }j\neq j'\right)\Longrightarrow\qquad \qquad \qquad\qquad\\
\notag
\qquad \qquad\qquad \qquad\left|\log\left(\frac{\nu_{j,k,p,n_p}}{\nu_{j',k',p,n_p}}\right)\right|+\frac{\left|\tau_{j,k,p,n_p}-\tau_{j',k',p,n_p}\right|}{\nu_{j,k,p,n_p}}+\frac{\left|z_{j,k,p,n_p}-z_{j',k',p,n_p}\right|}{\nu_{j,k,p,n_p}}\geq p.
\end{gather}
For this, it is sufficient to show, fixing $p$,
\begin{equation}
\label{DPD14}
\forall J\in\{1,\ldots,p\},\quad \limsup_{n\to\infty} \left\|w_{p,n}^J\right\|_{S(\RR)}<2\eps_J,
\end{equation} 
and 
\begin{multline}
 \label{DPD15}
\forall (j,k,j',k'),\quad \left(1\leq j,k,j',k'\leq p\text{ and }j\neq j'\right)\Longrightarrow\\
\lim_{n\to\infty} \left(\left|\log\left(\frac{\nu_{j,k,p,n}}{\nu_{j',k',p,n}}\right)\right|+\frac{\left|\tau_{j,k,p,n}-\tau_{j',k',p,n}\right|}{\nu_{j,k,p,n}}+\frac{\left|z_{j,k,p,n}-z_{j',k',p,n}\right|}{\nu_{j,k,p,n}}\right)=+\infty.
\end{multline}
The property \eqref{DPD14} is given immediately by the definition of $\eps_J$. 

We prove \eqref{DPD15} by contradiction. Assume \eqref{DPD15} does not hold, i.e. that there exists $(j,j',k,k')\in \{1,\ldots,p\}^{4}$ such that $j\neq j'$ and the term in parenthesis in \eqref{DPD15} is, after extraction of subsequences in $n$, bounded by some constant $C_0$. Then, denoting by $C>0$ a large constant that may change from line to line and depends on $j,k,j',k',p$ but not on $n$,
$$C_0\geq \left|\log\left(\frac{\nu_{j,k,p,n}}{\nu_{j',k',p,n}}\right)\right|=\left|\log\frac{\lambda_{p,j,n}}{\lambda_{p,j',n}}+\log\frac{\mu_{j,k,p}}{\mu_{j',k',p}}\right|\geq \left|\log\left(\frac{\lambda_{p,j,n}}{\lambda_{p,j',n}}\right)\right|-C,$$
which proves that the sequence $\left\{\left|\log\left(\frac{\lambda_{p,j,n}}{\lambda_{p,j',n}}\right)\right|\right\}_n$ is bounded.

Similarly, using the boundedness of the sequence $\left\{\left|\log\left(\frac{\lambda_{p,j,n}}{\lambda_{p,j',n}}\right)\right|\right\}_n$, we obtain
\begin{multline*}
 C_0\geq \frac{\left|\tau_{j,k,p,n}-\tau_{j',k',p,n}\right|}{\nu_{j,k,p,n}}=\frac{\left|t_{p,j,n}+s_{j,k,p}\lambda_{p,j,n}-t_{p,j',n}-s_{j',k',p}\lambda_{p,j',n}\right|}{\lambda_{p,j,n}\mu_{j,k,p}}\\
\geq \frac{1}{C} \left|\frac{t_{p,j,n}-t_{p,j',n}}{\lambda_{p,j,n}}\right|-C.
\end{multline*}
Thus the sequence $\left\{\left|\frac{t_{p,j,n}-t_{p,j',n}}{\lambda_{p,j,n}}\right|\right\}_n$ is also bounded. Finally, 
\begin{multline*}
\frac{\left|z_{j,k,p,n}-z_{j',k',p,n}\right|}{\nu_{j,k,p,n}}\geq \frac{1}{\mu_{j,k,p}}\frac{\left|x_{p,j,n}-x_{p,j',n}\right|}{\lambda_{p,j,n}}-\frac{\left|\lambda_{p,j,n}y_{j,k,p}-\lambda_{p,j',n}y_{j',k',p}\right|}{\lambda_{p,j,n}\mu_{j,k,p}}\\
\geq \frac{1}{C}\frac{\left|x_{p,j,n}-x_{p,j',n}\right|}{\lambda_{p,j,n}}-C
\end{multline*}
which proves that $\left\{\frac{\left|x_{p,j,n}-x_{p,j',n}\right|}{\lambda_{p,j,n}}\right\}_n$ is also bounded, contradicting the orthogonality of the sequences of parameters $\left(\left\{\lambda_{p,j,n},x_{p,j,n},t_{p,j,n}\right\}_n\right)$, $j\geq 1$ and concluding the proof of \eqref{DPD15}. 

\EMPH{Step 4: end of the proof} In this step we check that $\left(V_L^{j,k},\left\{\nu_{j,k,p,n_p},z_{j,k,p,n_p},\tau_{j,k,p,n_p}\right\}_p\right)_{j,k\geq 1}$ satisfies the assumptions of Claim \ref{C:profile}.

\EMPH{Orthogonality of the sequences of parameters} Let $(j,k),(j',k')$ be two pairs of indices such that $(j,k)\neq (j',k')$. If $j\neq j'$, then by \eqref{DPD13}, 
$$\lim_{p\to\infty}\left|\log\left(\frac{\nu_{j,k,p,n_p}}{\nu_{j',k',p,n_p}}\right)\right|+\frac{\left|\tau_{j,k,p,n_p}-\tau_{j',k',p,n_p}\right|}{\nu_{j,k,p,n_p}}+\frac{\left|z_{j,k,p,n_p}-z_{j',k',p,n_p}\right|}{\nu_{j,k,p,n_p}}=+\infty.$$
Next assume $j=j'$ (and thus $k\neq k'$). Then by the definition \eqref{def_param} of $\nu_{j,k,p,n}, z_{j,k,p,n}$, $\tau_{j,k,p,n}$
\begin{multline*}
 \left|\log\left(\frac{\nu_{j,k,p,n_p}}{\nu_{j',k',p,n_p}}\right)\right|+\frac{\left|\tau_{j,k,p,n_p}-\tau_{j',k',p,n_p}\right|}{\nu_{j,k,p,n_p}}+\frac{\left|z_{j,k,p,n_p}-z_{j',k',p,n_p}\right|}{\nu_{j,k,p,n_p}}\\
=\left|\log\left(\frac{\mu_{j,k,p}}{\mu_{j,k',p}}\right)\right|+\frac{\left|y_{j,k,p}-y_{j,k',p}\right|}{\mu_{j,k,p}}+\frac{\left|s_{j,k,p}-s_{j,k',p}\right|}{\mu_{j,k,p}}\underset{p\to\infty}{\longrightarrow}+\infty,
\end{multline*}
using the orthogonality of the sequences of parameters $\left(\left\{\mu_{j,k,p},y_{j,k,p},s_{j,k,p}\right\}_p\right)_{k\geq 1}$ (when $j$ is fixed).

\EMPH{Uniform summability of the profiles in $S(\RR)$}

Fix $j\geq 1$. By the orthogonality of the parameters $\left(\left\{\mu_{j,k,p},y_{j,k,p},s_{j,k,p}\right\}_p\right)_{k\geq 1}$, one can prove the Pythagorean-like expansion:
$$\sum_{k=1}^{\infty} \left\|V_L^{j,k}\right\|_{S(\RR)}^{\SN}=\left\|U_L^{p,j}\right\|_{S(\RR)}^{\SN}\leq \eta_j.$$
Since $\sum_j \eta_j<\infty$ we deduce
$$\sum_{j=1}^{\infty}\sum_{k=1}^{\infty} \left\|V_L^{j,k}\right\|_{S(\RR)}^{\SN}<\infty.$$

\EMPH{Convergence to $0$ of the Strichartz norm}
Using the formal expansion \eqref{DPD7}, we prove that there exists a sequence $K_J\to\infty$ such that
\begin{equation}
 \label{DPD16}
\lim_{J\to\infty} \limsup_{p\to\infty}\left\|u_{L,n_p}^p-\sum_{j=1}^J\sum_{k=1}^{K_J}\frac{1}{\nu_{j,k,p,n_p}^{\frac{N}{2}-1}}V_L^{j,k}\left(\frac{\cdot-\tau_{j,k,p,n_p}}{\nu_{j,k,p,n_p}},\frac{\cdot-z_{j,k,p,n_p}}{\nu_{j,k,p,n_p}}\right)\right\|_{S(\RR)}=0.
\end{equation} 
Indeed, recall from \eqref{defApn} the definition of $A_{p,n}^{J,K}$. 
For all $J\geq 1$ choose $K_J$ such that 
$$\limsup_{p\to\infty}\sum_{j=1}^J \left\|W^{K_J}_{j,p}\right\|_{S(\RR)}\leq \frac{1}{J}.$$
Then by the triangle inequality, 
$$ \limsup_{p\to\infty}\left\|A_{p,n_p}^{J,K_J}\right\|_{S(\RR)}\leq \frac{1}{J}$$ 
and combining with \eqref{DPD7} and \eqref{DPD12}, we get (recalling the definition \eqref{defEpsJ} of $\eps_J$):
$$\lim_{p\to\infty} \left\|u_{L,n_p}^p-\sum_{j=1}^J\sum_{k=1}^{K_J}\frac{1}{\nu_{j,k,p,n_p}^{\frac{N}{2}-1}}V_L^{j,k}\left(\frac{\cdot-\tau_{j,k,p,n_p}}{\nu_{j,k,p,n_p}},\frac{\cdot-z_{j,k,p,n_p}}{\nu_{j,k,p,n_p}}\right)\right\|_{S(\RR)}\leq 2\eps_J+\frac{1}{J},$$
hence \eqref{DPD16}. The assumptions of Claim \ref{C:profile} are satisfied, which concludes the proof of Lemma \ref{L:DPD}.
\end{proof}
\begin{remark}
 \label{R:DPD}
In the setting of Lemma \ref{L:DPD}, fix $j\geq 1$. Then, by \eqref{DPD5},
\begin{multline*}
\frac{1}{\lambda_{p,n_p}^{\frac{N}{2}-1}}U_L^{p,j}\left(\frac{t-t_{j,n_p}}{\lambda_{j,n_p}},\frac{x-x_{j,n_p}}{\lambda_{j,n_p}}\right)
\\=
\sum_{k=1}^K\frac{1}{\nu_{j,k,p,n_p}^{\frac{N}{2}-1}}V_L^{j,k}\left(\frac{s-\tau_{j,k,p,n_p}}{\nu_{j,k,p,n_p}},\frac{y-z_{j,k,p,n_p}}{\nu_{j,k,p,n_p}}\right)+ \frac{1}{\lambda_{j,n_p}^{\frac N2 -1} }W_{j,p}^K\left(\frac{t-t_{j,n_p}}{\lambda_{j,n_p}},\frac{x-x_{j,n_p}}{\lambda_{j,n_p}}\right),
\end{multline*}

and \eqref{DPD6} implies that 
$\left( V_L^{j,k},\left\{\nu_{j,k,p,n_p},z_{j,k,p,n_p},\tau_{j,k,p,n_p}\right\}_p \right)_{k\geq 1}$
is a profile decomposition for the sequence $\left\{\frac{1}{\lambda_{p,n_p}^{\frac{N}{2}-1}}U_L^{p,j}\left(\frac{-t_{j,n_p}}{\lambda_{j,n_p}},\frac{x-x_{j,n_p}}{\lambda_{j,n_p}}\right)\right\}_p.$
\end{remark}

\section{Weak and strong local convergence to solitary waves}
\label{S:minimality}
In this section, we prove Theorem \ref{T:profiles}.  
\subsection{Preliminaries}
Let $u$ be as in Theorem \ref{T:profiles}. We denote by $\SSS_0$ the set of sequences $\{t_n\}_n$ in the maximal interval of existence of $u$ such that 
\begin{equation}
\label{lim_T+}
\lim_{n\to\infty}t_n=T_+(u) 
\end{equation} 
and
$\left\{\vu(t_n)\right\}_n$ admits a \emph{well-ordered} (see Definition \ref{D:W0}) profile decomposition $$\profiles$$ 
with nonlinear profiles $U^j$. If $\{t_n\}_n \in \SSS_0$, we denote by $J_0=J_0\left(\{t_n\}_n\right)$ the number of nonlinear profiles $U^j$ that do not scatter forward in time. Note that by Lemma \ref{L:uniq_profile} and \eqref{link_NL} in Lemma \ref{L:intrinsec}, this is independent of the choice of the profile decomposition.  

By Proposition \ref{P:Z3'}, since $u$ does not scatter forward in time, $J_0\geq 1$. By the definition of the order relation $\preceq$, $U^1,\ldots, U^{J_0}$ do not scatter forward in time and for all $j\geq J_0+1$, $U^j$ scatters forward in time. 

According to the small data theory for \eqref{CP}, there exists $\delta_0>0$ such that
$$ \exists n,\quad \left\|U^j\left(\frac{-t_{j,n}}{\lambda_{j,n}}\right)\right\|^2_{\hdot\times L^2}\leq \delta_0\Longrightarrow U^j\text{ scatters in both time directions.}$$
By assumption \eqref{bound_u} of Theorem \ref{T:profiles}, the following supremum is finite:
\begin{equation}
 \label{def_M}
B=\sup_{t\in [0,T_+(u))}\|(u(t),\partial_tu(t))\|_{\hdot\times L^2}^2.
\end{equation} 
Using the Pythagorean expansions \eqref{Pyt1} and \eqref{Pyt2}, we obtain $\delta_0J_0\leq B$, i.e $J_0\left(\{t_n\}_n\right)\leq \delta_0/B$ for any sequence $\{t_n\}_n\in \SSS_0$.
Hence
\begin{equation}
 \label{DEM1}
J_M=\max\Big\{J_0\left(\left\{t_n\right\}_n\right),\;\left\{t_n\right\}_n\in \SSS_0\Big\}
\end{equation} 
is a finite integer, $\geq 1$. We define
\begin{equation}
 \label{DEM2}
\SSS_1=\Big\{\{t_n\}_n \in \SSS_0,\text{ s.t. }J_0\left(\{t_n\}_n\right)=J_M\Big\}.
\end{equation} 
If $\{t_n\}_n\in \SSS_1$, we define:
\begin{equation}
 \label{DEM3}
\EEE\left(\{t_n\}_n\right)=\sum_{j=1}^{J_M} E(U^j).
\end{equation}
By Lemma \ref{L:uniq_profile} and \eqref{link_NL}, $\EEE\left(\{t_n\}_n\right)$ is well defined, i.e. independent of the choice of the profile decomposition of $\{\vu(t_n)\}_n$.
Note that by Sobolev inequality,
\begin{equation*}
\forall n,\quad 
 \EEE\left(\{t_k\}_k\right)\geq -\sum_{j=1}^{J_M} \frac{N-2}{2N}\left\|U^j\left(-t_{j,n}/\lambda_{j,n}\right)\right\|_{L^{\frac{2N}{N-2}}}^{\frac{2N}{N-2}}\\
\geq -C\sum_{j=1}^{J_M} \left\|U^j\left(-t_{j,n}/\lambda_{j,n}\right)\right\|_{\hdot}^{\frac{2N}{N-2}}.
\end{equation*}
Thus, using \eqref{def_M} and the Pythagorean expansion \eqref{Pyt3},
\begin{equation}
 \label{DEM4} \forall \{t_n\}_n\in \SSS_1,\quad \EEE\left(\{t_n\}_n\right)\geq -CJ_MB^{\frac{N}{N-2}}.
\end{equation} 
We let 
\begin{equation}
 \label{DEM5}
\EEE_{m}=\inf \Big\{ \EEE\left(\{t_n\}_n\right),\; \{t_n\}_n\in \SSS_1\Big\} \in \RR.
\end{equation} 
\subsection{Minimization of $\EEE$}
\begin{lemma}
 \label{L:DEM1}
For $p\geq 1$, let $\left\{\tau_{n}^p\right\}_n\in \SSS_1$ such that 
\begin{equation}
 \label{DEM5'}
\lim_{p\to\infty} \EEE\left(\left\{\tau_{n}^p\right\}_n\right)=\EEE_m.
\end{equation} 
Consider, for all $p\geq 1$, a well-ordered profile decomposition $\left(U_L^{p,j},\left\{\lambda_{p,j,n},x_{p,j,n},t_{p,j,n}\right\}_n\right)_{j\geq 1}$ of $\left\{u\left(\tau_n^p\right)\right\}_n$, and denote by $U^{p,j}$, $j\geq 1$, the corresponding nonlinear profiles. Then there exists an increasing sequence $\{n_p\}_p$ of indices such that (after extraction in $p$)
\begin{enumerate}
 \item \label{I:S1} $\ds \left\{\tau_{n_p}^p\right\}_p \in \SSS_1$ and $\EEE\left(\left\{\tau_{n_p}^p\right\}_p\right)=\EEE_m$.
\item \label{I:profile} $\left\{\vu\left(\tau_{n_p}^p\right)\right\}_p$ admits a profile decomposition $\left(\tV_L^j,\left\{\tnu_{j,p},\tz_{j,p},\ttau_{j,p}\right\}_p\right)_{j\geq 1}$ such that, denoting by $\tV^j$ the corresponding nonlinear profiles, $\tV^j$ scatters forward in time if and only if $j\geq 1+J_M$ and 
\begin{multline}
 \label{DEM6}
\forall j\in \{1,\ldots,J_M\},\\
\lim_{p\to\infty} \left\|\frac{1}{\lambda_{p,j,n_p}^{\frac{N}{2}}}\nabla_{t,x}U^{p,j}\left(-\frac{t_{p,j,n_p}}{\lambda_{p,j,n_p}},\frac{\cdot-x_{p,j,n_p}}{\lambda_{p,j,n_p}}\right)-\frac{1}{\tnu_{j,p}^{\frac{N}{2}}}\nabla_{t,x}\tV^{j}\left(-\frac{\ttau_{j,p}}{\tnu_{j,p}},\frac{\cdot-\tz_{j,p}}{\tnu_{j,p}}\right)\right\|_{(L^2)^{N+1}}=0
\end{multline} 
\end{enumerate}
\end{lemma}
\begin{remark}
 \label{R:DEM1}
In \eqref{DEM6}, by definition of $\tV^j$, we can of course replace $\tV^j$ by $\tV_L^j$.
\end{remark}
As an immediate corollary of Lemma \ref{L:DEM1}, we get:
\begin{corol}
 \label{C:DEM1}
 The set
 $$\SSS_2=\Big\{\{t_n\}_n\in \SSS_1,\text{ s.t. } \EEE\left(\{t_n\}_n\right)=\EEE_m\Big\}$$
 is not empty.
\end{corol}

\begin{proof}
 \EMPH{Step 1. Double profile decomposition}
In this step we place ourselves in the assumptions of Lemma \ref{L:DPD}. We first reorder the profiles $U^{p,j}_L$ so that the energy of the nonlinear profiles is a decreasing function of $j\geq 1+J_M$ when $p$ is fixed, i.e. 
\begin{equation}
 \label{DEM8}
\forall p, \;\forall j\geq 1+J_M,\quad E\left[U^{p,j}\right]\geq E\left[U^{p,j+1}\right].
\end{equation} 
Since all profiles $U^{p,j}$, $j\geq 1+J_M$ scatter forward in time, this does not affect the fact that the profile decompositions $\left(U_L^{p,j},\left\{\lambda_{p,j,n},x_{p,j,n},t_{p,j,n}\right\}_n\right)_{j\geq 1}$ are well-ordered.

By a standard diagonal extraction argument we can assume (extracting subsequences in $p$) that for all $j\geq 1$, $\left\{\vU_L^{p,j}(0)\right\}_p$ admits a profile decomposition $\left(V_L^{j,k},\left\{ \mu_{j,k,p},y_{j,k,p},s_{j,k,p}\right\}_p\right)_{k\geq 1}$.  To check that we are exactly in the setting of Lemma \ref{L:DPD}, we must find a sequence $(\eta_j)_{j\geq 1}$ of positive numbers such that \eqref{DPD2} holds.  

We first note that by \eqref{def_M} and the Pythagorean expansions \eqref{Pyt1}, \eqref{Pyt2}, 
$$\limsup_{n\to\infty} \left\|\vU_L^{p,j}\left(-t_{p,j,n}/\lambda_{p,j,n}\right)\right\|_{\hdot\times L^2}^2\leq B.$$
Thus, by Strichartz inequality,
\begin{equation}
 \label{DEM10} \forall p,j, \quad \left\|U_L^{p,j}\right\|_{S(\RR)}^{\SN}\leq CB^{\frac{N+1}{N-2}}.
\end{equation} 
We next find a bound of $\left\|U_L^{p,j}\right\|_{S(\RR)}$ for large $j$. Fix $j$ and $p$ and assume $j\geq 1+J_M$. Since $U^{p,j}$ scatters forward in time, the energy $E\left[U^{p,j}\right]$ is nonnegative. By the Pythagorean expansions \eqref{Pyt1}, \eqref{Pyt2}, \eqref{Pyt3} for the profile decomposition of the sequence $\left\{\vu\left(\tau_n^p\right)\right\}_{n}$:
\begin{equation*}
\sum_{j=1}^{J_M}E\left[U^{p,j}\right]+\sum_{j\geq 1+J_M}E\left[U^{p,j}\right]\leq E[u],
\end{equation*} 
and thus
\begin{equation}
 \label{DEM11}
\sum_{j\geq 1+J_M} E\left[U^{p,j}\right] \leq E[u]-\EEE_m.
\end{equation} 
Using \eqref{DEM11} and the monotonicity \eqref{DEM8} of the energy sequence we obtain:
\begin{equation}
 \label{DEM12}
\forall p,\; \forall j\geq 1+J_M,\quad E\left[U^{p,j}\right]\leq \frac{E[u]-\EEE_m}{j-J_M}.
\end{equation} 
Let $\tJ\geq 1+J_M$ such that $E\left[U^{p,j}\right]\leq \frac{E[u]-\EEE_m}{\tJ-J_M}\leq \delta_1$, where $\delta_1>0$ is a fixed small parameter. If $j\geq \tJ$, then by  \eqref{DEM12} $E\left[U^{p,j}\right]\leq \delta_1$. Since $U^{p,j}$ scatters forward in time, we deduce 
\begin{equation}
 \label{DEM13}
\lim_{t\to\infty}\left\|\vU^{p,j}(t)\right\|^2_{\hdot\times L^2}=2E\left[U^{p,j}\right]\leq 2\delta_1.
\end{equation}  
By the small data theory, choosing $\delta_1$ small enough, we deduce that for all $j\geq \tJ$, $\vU^{p,j}$ is global and
\begin{equation*}
\forall t\in \RR,\quad \left\|\vU^{p,j}(t)\right\|_{\hdot\times L^2}^2\leq 3E\left[U^{p,j}\right].
\end{equation*} 
Thus for large $j$, 
$\left\|\vU^{p,j}_L(0)\right\|_{\hdot\times L^2}^2\leq 4E\left[U^{p,j}\right]$, and we deduce by Strichartz inequality
\begin{equation}
 \label{DEM15}
\forall j\geq \tJ,\quad \left\|U_L^{p,j}\right\|_{S(\RR)}^{\SN}\leq C E\left[U^{p,j}\right]^{\frac{N+1}{N-2}} \leq C \left(\frac{E[u]-\EEE_m}{j-J_M}\right)^{\frac{N+1}{N-2}}.
\end{equation} 
Letting $\eta_j=CB^{\frac{N+1}{N-2}}$ if $1\leq j\leq \tJ-1$ and $\eta_j=C\left(\frac{E[u]-\EEE_m}{j-J_M}\right)^{\frac{N+1}{N-2}}$ if $j\geq \tJ$, we get by \eqref{DEM10} and \eqref{DEM15} that \eqref{DPD2} is satisfied. We are thus exactly in the setting of Lemma \ref{L:DPD}.

\EMPH{Step 2. Application of Lemma \ref{L:DPD}} For any $p\geq 1$, we choose $m_p\in \NN$ such that
\begin{equation}
 \label{DEM16}
\forall j\in \{1,\ldots, J_M\},\; \forall n\geq m_p,\quad 
\left\|\vU_L^{p,j}\left(-t_{j,n}/\lambda_{j,n}\right)-\vU^{p,j}\left(-t_{j,n}/\lambda_{j,n}\right)\right\|_{\hdot\times L^2}\leq \frac{1}{p}
\end{equation} 
and 
\begin{equation}
 \label{DEM17}
\forall n\geq m_p,\quad \tau_{n}^p\geq T_+(u)-\frac 1p. 
\end{equation} 
By Lemma \ref{L:DPD}, with $\left(u_{0,n}^p,u_{1,n}^p\right)=\vu(\tau_n^p)$ there exists an increasing sequence $\{n_p\}_p$ such that 
\begin{equation}
 \label{DEM18}
\forall p,\quad n_p\geq m_p,
\end{equation} 
and $\left\{\vu\left(\tau_{n_p}^p\right)\right\}_p$ has a profile decomposition 
$$ \left(V_L^{j,k},\left\{ \nu_{j,k,p,n_p},z_{j,k,p,n_p},\tau_{j,k,p,n_p}\right\}_p\right\}_{j,k\geq 1}$$
as in the conclusion of Lemma \ref{L:DPD}. We will denote by $V^{j,k}$ the nonlinear profiles corresponding to this decomposition. By \eqref{DEM17} and \eqref{DEM18}, $\lim_{p\to\infty}\tau_{n_p}^p=T_+(u)$. By Claim \ref{C:Z3}, there exists a subsequence of $\left\{\tau_{n_p}^p\right\}_p$, still denoted by $\left\{\tau_{n_p}^p\right\}_p$, which is in $\SSS_0$. 

\EMPH{Step 3. Analysis of the profiles}
In this step we prove that for all $j\in\{1,\ldots,J_M\}$, there is exactly one $k$ (say $k=1$) such that $V^{j,k}$ does not scatter forward in time, and that for all $j\geq 1+J_M$, for all $k\geq 1$, the profile $V^{j,k}$ scatters forward in time. 
For this we will use in a crucial manner the fact that for all $p$, $J_0\left(\left\{\tau_{n}^p\right\}_n\right)=J_M$.
 
We first fix $j\in \{1,\ldots,J_M\}$. By Remark \ref{R:DPD}, 
\begin{equation}
 \label{DEM20star}
\left( V_L^{j,k},\left\{\nu_{j,k,p,n_p},z_{j,k,p,n_p},\tau_{j,k,p,n_p}\right\}_p \right)_{k\geq 1} 
\end{equation} 
is a profile decomposition for the sequence 
$$\left\{\left(\frac{1}{\lambda_{p,n_p}^{\frac{N}{2}-1}}U_L^{p,j}\left(\frac{-t_{j,n_p}}{\lambda_{j,n_p}},\frac{\cdot-x_{j,n_p}}{\lambda_{j,n_p}}\right), \frac{1}{\lambda_{p,n_p}^{\frac{N}{2}}}\partial_tU_L^{p,j}\left(\frac{-t_{j,n_p}}{\lambda_{j,n_p}},\frac{\cdot-x_{j,n_p}}{\lambda_{j,n_p}}\right)\right)\right\}_p.$$
By \eqref{DEM16}, 
\begin{equation}
 \label{DEM20'}
 \lim_{p\to+\infty} \left\|\vU_L^{p,j}\left(-t_{j,n_p}/\lambda_{j,n_p}\right)-\vU^{p,j}\left(-t_{j,n_p}/\lambda_{j,n_p}\right)\right\|_{\hdot\times L^2}=0,
\end{equation} 
and we deduce that \eqref{DEM20star} is a profile decomposition for the sequence 
$$\left\{\left(\frac{1}{\lambda_{p,n_p}^{\frac{N}{2}-1}}U^{p,j}\left(\frac{-t_{j,n_p}}{\lambda_{j,n_p}},\frac{\cdot-x_{j,n_p}}{\lambda_{j,n_p}}\right), \frac{1}{\lambda_{p,n_p}^{\frac{N}{2}}}\partial_tU^{p,j}\left(\frac{-t_{j,n_p}}{\lambda_{j,n_p}},\frac{\cdot-x_{j,n_p}}{\lambda_{j,n_p}}\right)\right)\right\}_p.$$
Note that the nonlinear profiles corresponding to this profile decomposition are also by definition the nonlinear profiles $\left(V^{j,k}\right)_{k\geq 1}$ defined in Step 2. 

Since $1\leq j\leq J_M$, for all $p$, $U^{p,j}$ does not scatter forward in time. As a consequence, at least one of the solutions $V^{j,k}$, $k\geq 1$, say $V^{j,1}$, does not scatter forward in time. We have identified $J_M$ nonlinear profiles $V^{1,1}$, \ldots,$V^{J_M,1}$ that do not scatter forward in time. Since $\left\{\tau_{n_p}^p\right\}_p\in \SSS_0$, we deduce by the definition of $J_M$ that all other profiles $V^{j,k}$ must scatter forward in time, i.e.
\begin{equation}
 \label{DEM21}
 V^{j,k}\text{ scatters forward in time}\iff (k\geq 2) \text{ or } (j\geq 1+J_M).
\end{equation} 
We deduce in particular $J_0\left(\left\{\tau_{n_p}^p\right\}_p\right)=J_M$ and $\left\{\tau_{n_p}^p\right\}_p\in \SSS_1$.

\EMPH{Step 4} Denote (as in the proof of Lemma \ref{L:DPD}),
\begin{equation}
 \label{DEM22}
 W_p^{j,K}(t,x)=U_L^{p,j}(t,x)-\sum_{k=1}^K \frac{1}{\mu_{j,k,p}^{\frac N2 -1}} V_L^{j,k}\left(\frac{t-s_{j,k,p}}{\mu_{j,k,p}},\frac{x-x_{j,k,p}}{\mu_{j,k,p}}\right).
\end{equation} 
In this step we prove using the assumption \eqref{DEM5'}:
\begin{gather}
 \label{DEM23} 
 \EEE\left(\{\tau_{n_p}^p\}_p\right)=\sum_{j=1}^{J_M}E\left[V^{j,1}\right]=\EEE_{m}\\
 \label{DEM24}
 \forall j\in \left\{1,\ldots,J_M\right\},\; \forall k\geq 2,\quad V_L^{j,k}=0\\
 \label{DEM25}
 \lim_{p\to\infty} \left\|W_p^{j,2}(0)\right\|_{\hdot\times L^2}=0.
\end{gather} Note that by \eqref{DEM24}, $W_p^{j,K}$ is independent of $K\geq 2$, so that \eqref{DEM25} is equivalent to
$$ \forall K\geq 2,\quad \lim_{p\to\infty} \left\|W_p^{j,K}(0)\right\|_{\hdot\times L^2}=0.
 $$
 Since by Step 3, \eqref{DEM20star} is a profile decomposition for the sequence $\left\{\frac{1}{\lambda_{p,n_p}^{\frac{N}{2}-1}}\vU^{p,j}\left(\frac{-t_{j,n_p}}{\lambda_{j,n_p}},\frac{\cdot-x_{j,n_p}}{\lambda_{j,n_p}}\right)\right\}_p$, with nonlinear profiles $\left(V^{j,k}\right)_{k\geq 1}$, the Pythagorean expansions \eqref{Pyt1}, \eqref{Pyt2} and \eqref{Pyt3} yield
 \begin{equation}
  \label{DEM26}
  \forall K\geq 2,\; \forall j=1,\ldots, J_M,\quad \lim_{p\to\infty}\left(E\left[U^{p,j}\right]-\sum_{k=1}^K E\left[V^{j,k}\right]-E\left(\vW_p^{j,K}(0)\right)\right)=0.
 \end{equation} 
 Hence
 \begin{equation}
  \label{DEM27}
  \forall K\geq 2,\; \lim_{p\to\infty}\sum_{j=1}^{J_M}E\left[U^{p,j}\right]-\sum_{j=1}^{J_M}\left(\sum_{k=1}^K E\left[V^{j,k}\right]+E\left(\vW_p^{j,K}(0)\right)\right)=0.
 \end{equation} 
 By Step 3, $\left\{\tau_{n_p}^p\right\}_p\in \SSS_1$ and, by the definition of $\EEE_m$,  $\EEE\left(\left\{\tau_{n_p}^p\right\}_p\right)=\sum_{j=1}^{J_M} E\left[V^{j,1}\right]\geq \EEE_m$. By \eqref{DEM5'}, $\lim_{p} \sum_{j=1}^{J_M}E\left[U^{p,j}\right]=\EEE_m$. Combining with \eqref{DEM27}, we obtain 
 \begin{equation}
  \label{DEM28}
  \limsup_{p\to\infty} \sum_{j=1}^{J_M} \left(\sum_{k=2}^K E\left[V^{j,k}\right]+E\left(\vW_p^{j,K}(0)\right)\right)\leq 0.
 \end{equation} 
 Since $V^{j,k}$ scatters for $k\geq 2$, and the solution with initial data $W_p^{j,K}(0)$  scatters for large $K$ and $p$ we deduce, taking $K$ large,
 \begin{gather}
 \label{DEM29}
  \forall j=1,\ldots, J_M,\; \forall k= 2,\ldots, K,\quad E\left[V^{j,k}\right]=0\\
 \label{DEM30}
 \forall j=1,\ldots, J_M,\quad \lim_{p\to\infty}E\left(\vW_p^{j,K}(0)\right)=0.
 \end{gather}
By \eqref{DEM29}, 
$$\lim_{t\to+\infty} \left\|\vV^{j,k}(t)\right\|^2_{\hdot\times L^2}=2 E\left[V^{j,k}\right]=0$$
and thus $V^{j,k}=0$ by the small data theory for \eqref{CP}. Similarly, by \eqref{DEM30}, and the small data theory, $\lim_{p\to\infty}\left\|\vW_p^{J,K}(0)\right\|_{\hdot\times L^2}=0$. Going back to \eqref{DEM27}, we get $\sum_{j=1}^{J_M}E\left[V^{j,1}\right]=\EEE_m$. Step 4 is complete.

\EMPH{Step 5. End of the proof} We next check that the conclusion of Lemma \ref{L:DEM1} holds. 

Point \eqref{I:S1} follows from Step 3 and \eqref{DEM23}. For $j=1,\ldots, J_M$, we define
$$\tnu_{j,p}=\nu_{j,1,p,n_p},\quad z_{j,p}=z_{j,1,p,n_p},\quad \ttau_{j,p}=\tau_{j,1,p,n_p},$$
where $\nu_{j,k,p,n}$, $z_{j,k,p,n}$ and $\tau_{j,k,p,n}$ are defined by \eqref{def_param}, and $\tV^j_L=V_L^{j,1}$.

Choosing an arbitrary order for the profiles $V^{j,k}$, where $j\geq 1+J_{M}$, $k\geq 1$, we get that $\left\{\vu\left(\tau_{n_p}^p\right)\right\}_p$ as a profile decomposition as in point \eqref{I:profile}. The property \eqref{DEM6} is given by \eqref{DEM24}, \eqref{DEM25} and the definitions of $\nu_{j,1,p,n}$, $z_{j,1,p,n}$ and $\tau_{j,1,p,n}$. The proof is complete.
\end{proof}
\subsection{New minimization procedure and adjustment of the parameters}
If $\{t_n\}_n\in \SSS_2$ (defined in Corollary \ref{C:DEM1}) and $\profiles$ is a well-ordered profile decomposition of $\{\vu(t_n)\}_n$, we denote by 
\begin{equation}
 \label{DEM31}
 J_1\left(\{t_n\}_n\right)=\min\Big\{J\geq 1,\; (J)\prec (J+1)\Big\},
\end{equation} 
where the notation ``$\prec$'' stands for the \emph{strict} preorder relation defined in Subsection \ref{SS:preorder}. Since $U^{J_M}$ does not scatter forward in time and $U^{J_M+1}$ scatters forward in time, we have $(J_M)\prec (1+J_{M})$, and thus $J_1(\{t_n\}_n)$ is an integer such that $1\leq J_{1}\left(\{t_n\}_n\right)\leq J_M$. Note that by Lemmas \ref{L:uniq_profile} and \ref{L:intrinsec},  $J_1(\{t_n\}_n)$ is independent of the choice of the profile decomposition of $\{\vu(t_n)\}_n$.  We let
\begin{equation}
 \label{DEM32}
 J_m=\min\Big\{J_1\left(\{t_n\}_n\right),\; \{t_n\}_n\in \SSS_2\Big\},
\end{equation} 
and
\begin{equation}
\label{DEM33}
 \SSS_3=\Big\{\{t_n\}_n\in \SSS_2\text{ s.t. } J_1\left( \{t_n\}_n \right)=J_m\Big\},
\end{equation} 
which is nonempty since $\SSS_2$ (defined in Corollary \ref{C:DEM1}) is nonempty. In this step, we prove
\begin{lemma}
 \label{L:DEM2}
 Let $u$ be as in Theorem 1. There exists $\{t_n\}_n\in \SSS_3$, and a well-ordered profile decomposition $\profiles$ of $\left\{\vu(t_n)\right\}_n$ such that
 \begin{equation}
 \label{D34}
  \forall j=1,\ldots, J_{m}, \quad T_+(U^j)=T_+(U^1),\text{ and }\forall n,\; t_{j,n}=0,\; \lambda_{j,n}=\lambda_{1,n}.
 \end{equation} 
\end{lemma}
\begin{proof}
 We let $\{t_n\}_n\in \SSS_3$, and $\profiles$ be a well-ordered profile decomposition of $\{\vu(t_n)\}_n$. As usual, we denote by $U^j$ the associated nonlinear profiles. Since $J_1(\{t_n\}_n)=J_m$, $J_0(\{t_n\}_n)=J_M$, we have
 \begin{gather}
  \label{DEM36}
  (1)\backsimeq \ldots \backsimeq (J_{m}) \prec (1+J_m)\preceq (2+J_m)\preceq \ldots\\
\label{DEM37}
U^j\text{ scatters forward in time }\iff j>J_M.
  \end{gather} 
  Extracting subsequences and time-translating the profiles (see Lemma \ref{L:modif_profile}), we may assume
  \begin{equation}
   \label{DEM38}
   \forall j\in \{1,\ldots, J_M\},\quad \Big( \lim_{n\to \infty}-t_{j,n}/\lambda_{j,n}=-\infty\text{ or } \forall n,\; t_{j,n}=0\Big).
  \end{equation} 
  As a consequence of Lemma \ref{L:intrinsec}, this does not affect the fact that the profile decomposition is well-ordered.
  
We will use the following claim, proved in Appendix \ref{A:preorder}.
\begin{claim}
\label{C:fact_preorder}
 Let $\profiles$ be a profile decomposition, and $j\neq k$ two indices such that $(j)\backsimeq(k)$. Then we cannot have 
 \begin{equation}
  \label{DEM40}
  \forall n,\quad t_{j,n}=0\text{ and }\lim_{n\to\infty} -t_{k,n}/\lambda_{k,n}=-\infty.
 \end{equation} 
 Moreover, if for all $n$, $t_{j,n}=t_{k,n}=0$, then (extracting subsequences if necessary) there exists $c\in (0,+\infty)$ such that
 \begin{gather}
  \label{DEM42}
  \lim_{n\to\infty} \lambda_{j,n}/\lambda_{k,n}=c\\
  \label{DEM43}
  T_+(U^k)=T_+(U^j)=+\infty\text{ or } \Big(T_+(U^j)<\infty\text{ and }T_{+}(U^k)=cT_+(U^j)\Big).
 \end{gather}
\end{claim}
We distinguish two cases.

  \EMPH{Case 1} We assume
\begin{equation}
 \label{DEM39}
 \forall n,\quad t_{1,n}=0.
\end{equation} 
By Claim \ref{C:fact_preorder},
\begin{equation}
 \label{DEM50}
 \forall n,\quad t_{1,n}=t_{2,n}=\ldots=t_{J_m,n}=0.
\end{equation} 
For $j=1,\ldots, J_m$ we let, after extraction, 
\begin{equation}
 \label{DEM51}
 c_j=\lim_{n\to \infty} \lambda_{j,n}/\lambda_{1,n},
\end{equation} 
and note that by the second part of Claim \ref{C:fact_preorder}, $c_j\in(0,\infty)$, and one of the following holds:
$$ T_+(U^1)=T_{+}(U^2)=\ldots=T_+(U^{J_m})=+\infty,$$
or, for $j=1,\ldots, J_{m}$, $T_+(U^j)$ is finite and $T_+(U^j)=\frac{1}{c_j}T_{+}(U^1)$.
Rescaling the profiles $U^2_L$,\ldots,$U^{J_m}_L$, we can assume 
$c_1=c_2=\ldots=c_{J_m}=1$, and $\lambda_{1,n}=\lambda_{2,n}=\ldots=\lambda_{J_m,n}$ for all $n$, and thus by Claim \ref{C:fact_preorder}, $T_+(U^1)=T_+(U^2)=\ldots=T_+(U^{J_m})$. The proof is complete in this case.

\EMPH{Case 2} We assume
\begin{equation}
 \label{DEM52}
 \lim_{n\to\infty} \frac{-t_{1,n}}{\lambda_{1,n}}=-\infty.
\end{equation} 
Assume (time translating $U^1$ if necessary) that $T_+(U^1)$ is positive. We apply Proposition \ref{P:Z4} with $T=0$. Thus $\tau_n=t_{1,n}$. Let $t_n'=t_n+t_{1,n}$. By Proposition \ref{P:Z4} and Remark \ref{R:profiles}, $t_n'\in I_{\max}(u)$ for large $n$ and $\left\{u(t_n')\right\}_n$ has a profile decomposition
$\left(V_{L}^j,\left\{\lambda_{j,n},x_{j,n},t_{j,n}'\right\}_n\right)_{j\geq 1}$, where $t_{j,n}'=t_{j,n}-t_{1,n}$, and the nonlinear profiles associated to this profile decomposition are exactly the nonlinear profiles $U^j$ associated to the profile decomposition of $\left\{u(t_n)\right\}_n$. Note also that since $t_{1,n}>0$ for large $n$, $\lim_n t_n'=T_+(u)$. Since $t_{j,n}'-t_{k,n}'=t_{j,n}-t_{k,n}$, it is easy to check that $\{t_n'\}_n\in \SSS_3$. Since $t_{1,n}'=0$ for all $n$, we are reduced to Case 1 above. The proof of Lemma \ref{L:DEM2} is complete.
 
\end{proof}
\subsection{A forward in time compactness property for the first profiles}
We next prove:
\begin{lemma}
 \label{L:DEM4}
 Let $u$ be as in Theorem \ref{T:profiles}, and $\{t_n\}_n\in \SSS_3$ the sequence given by Lemma \ref{L:DEM2}. Then, for all $j\in \{1,\ldots,J_{m}\}$, there exists $\mu_j(t)>0$, $y_j(t)\in \RR^N$, defined for $t\in [0,T_+(U^1))=[0,T_+(U^j))$ and such that
$$K_{+}^j=\left\{\left(\mu_j^{\frac N2-1}(t)U^j(t,\mu_{j}(t)\cdot+y_j(t)),\mu_j^{\frac N2}(t)\partial_tU^j(t,\mu_{j}(t)\cdot+y_j(t))\right),\; t\in \left[0,T_+(U^j)\right)\right\}$$
 has compact closure in $\hdot\times L^2$. 
 \end{lemma}
\begin{proof}
 Without loss of generality, we check this property for $j=1$. By a standard lifting lemma, it is sufficient to prove that for all sequence of times $\{s_p\}_p$ such that $0\leq s_p<T_+(U^1)$ and $\lim_p s_p=T_+(U^1)$, there exists a subsequence of $\{s_p\}_p$ and sequences $\{\nu_p\}_p$, $\{z_p\}_p$ such that
\begin{equation}
 \label{COM40}
\lim_{p\to\infty} \left\|\left(\nu_p^{\frac N2-1}U^1(s_p,\nu_p\cdot+z_p),\nu_p^{\frac N2}\partial_t U^1(s_p,\nu_p\cdot+z_p)\right)-(V_0,V_1)\right\|_{\hdot\times L^2}=0.
\end{equation} 
Let $\{s_p\}_p$ be such a sequence, and $\tau_n^p=t_n+s_p\lambda_{1,n}$. Fixing $p$, we see by Proposition \ref{P:Z4} and Remark \ref{R:profiles} that $\tau_n^p<T_+(u)$ for large $n$, and that the sequence $\left\{\vu(\tau_n^p)\right\}_n$ has a profile decomposition $\left(U_L^{p,j},\left\{\lambda_{j,n},x_{j,n},t_{p,j,n}\right\}_n\right)_j$, where $t_{p,j,n}=t_{j,n}-s_p\lambda_{1,n}$ and the nonlinear profiles $U^{p,j}$ for this profile decomposition satisfy
\begin{equation}
 \label{COM41}
\forall p,j\quad U^{p,j}=U^j,
\end{equation} 
where the $U^j$s are the nonlinear profiles of the profile decomposition $\profiles$ of $\{\vu(t_n)\}_n$ given by Lemma \ref{L:DEM2}. 

As a consequence, the profile decomposition $\left(U_L^{p,j},\left\{\lambda_{j,n},x_{j,n},t_{p,j,n}\right\}_n\right)_j$ satisfies
$$(1)\backsimeq \ldots \backsimeq (J_m)\prec (1+J_m)\preceq \ldots.$$
Moreover, $J_1\left(\{\tau_n^p\}_n\right)=J_1\left(\{t_n\}_n\right)=J_m$,
$J_0\left(\{\tau_n^p\}_n\right)=J_0\left(\{t_n\}_n\right)=J_M$, and $\EEE\left(\left\{\tau_{n}^p\right\}_n\right)=\EEE_m$. By Lemma \ref{L:DEM1}, there exists an increasing sequence $n_p$ of indices, a sequence of parameters $\{\tnu_p,\tz_p,\ttau_p\}_p$, and a solution $\tV_L$ of the linear wave equation such that (see \eqref{DEM6})
\begin{equation}
 \label{COM42}
\lim_{p\to\infty}\left\|\frac{1}{\lambda_{1,n_p}^{\frac N2}}\nabla_{t,x}U^1\left(\frac{-t_{p,1,n_p}}{\lambda_{1,n_p}},\frac{\cdot-x_{1,n_p}}{\lambda_{1,n_p}}\right)-\frac{1}{\tnu_p^{\frac N2}}\nabla_{t,x}\tV_L\left(\frac{-\ttau_p}{\tnu_p},\frac{\cdot-\tz_p}{\tnu_p}\right)\right\|_{(L^2)^{N+1}}=0
\end{equation} 
Noting that $\frac{t_{p,1,n_p}}{\lambda_{1,n_p}}=\frac{t_{1,n_p}-s_p\lambda_{1,n_p}}{\lambda_{1,n_p}}=-s_p$, we rewrite \eqref{COM42} as:
\begin{equation}
 \label{COM43}
\lim_{p\to\infty}\left\|\frac{1}{\lambda_{1,n_p}^{\frac N2}}\nabla_{t,x}U^1\left(s_p,\frac{\cdot-x_{1,n_p}}{\lambda_{1,n_p}}\right)-\frac{1}{\tnu_p^{\frac N2}}\nabla_{t,x}\tV_L\left(\frac{-\ttau_p}{\tnu_p},\frac{\cdot-\tz_p}{\tnu_p}\right)\right\|_{(L^2)^{N+1}}=0.
\end{equation} 
To conclude the proof of the lemma, we must show that the sequence $\left\{-\ttau_p/\tnu_p\right\}_p$ is bounded. If, after extraction, $\lim_p -\ttau_p/\tnu_p=+\infty$, then $U^1$ scatters forward in time, a contradiction. If $\lim_p -\ttau_p/\tnu_p=-\infty$, then $U^1$ scatters backward in time and for large $p$,
$$ \left\|U^1\right\|_{S(-\infty,s_p)}\leq 2\left\|\tV_L\right\|_{S(-\infty,-\ttau_p/\tnu_p)}\underset{p\to\infty}{\longrightarrow}0,$$
a contradiction. The proof is complete. 
\end{proof}
\subsection{Adjusting the sequence of times}
The end of the proof of Theorem \ref{T:profiles} consists in changing the sequence $\{t_n\}_n$ given by Lemma \ref{L:DEM4} to improve the properties of the profiles $U^1,\ldots,U^{J_m}$ and strengthen the convergence of $\vu(t_n)$ to this profiles. For this we start by proving a general technical lemma.
\subsubsection{Extraction of profiles along a new sequence of times}
\begin{lemma}
\label{L:F1}
Let $\{m_p'\}$ be a sequence of positive numbers.
Let $u$ be as in Theorem \ref{T:profiles}. Let $\{t_n\}_n\in \SSS_3$ and assume that $\left\{\vu(t_n)\right\}_n$ has a well-ordered profile decomposition $\profiles$ with nonlinear profiles $U^j$. Assume 
$$\forall j\geq 1+J_m, \quad \lim_n -t_{j,n}/\lambda_{j,n}=\pm \infty \text{ or }\forall n,\; t_{j,n}=0,$$
and
\begin{equation}
 \label{F1}
\forall j=1,\ldots, J_m,\quad \lambda_{j,n}=\lambda_{1,n},\quad t_{j,n}=0,\quad T_+(U^j)=T_+(U^1). 
\end{equation} 
Let $\{s_p\}_p$ be an increasing sequence in $[0,T_+(U^1))$ such that 
$$\lim_{p\to\infty} s_p=T_+(U^1)$$
and for all $j\in\{1,\ldots, J_m\}$, there exists $\left(V_0^j,V_1^j\right)\in \hdot\times L^2$, and sequences $\{y_{j,p}\}_p$, $\{\mu_{j,p}\}_p$ such that
\begin{equation}
 \label{F2}
\lim_{p\to\infty}\left\|\vU^j(s_p)-\left(\frac{1}{\mu_{j,p}^{\frac N2-1}}V_0^j\left(\frac{\cdot-y_{j,p}}{\mu_{j,p}}\right),\frac{1}{\mu_{j,p}^{\frac N2}}V_1^j\left(\frac{\cdot-y_{j,p}}{\mu_{j,p}}\right)\right)\right\|_{\hdot\times L^2}=0.
\end{equation} 
Assume that the following limits exist for all $j\geq 1$:
\begin{equation}
 \label{F3}
\theta_{p}^j=\lim_{n\to\infty} \frac{-t_{j,n}+s_p\lambda_{1,n}}{\lambda_{j,n}}\in \RR\cup\{\pm\infty\},\quad \theta^j=\lim_{p\to\infty}\theta_p^j\in \RR\cup\{\pm\infty\}.
\end{equation}
Then there exists an increasing sequence of indices $\{n_p\}_p$ such that $n_p\geq m_p'$ for all $p$ and
\begin{gather}
 \label{F4}
\forall p,\quad t_{n_p}+s_p\lambda_{1,n_p}\in \left[0,T_+(u)\right)\\
\label{F5}
\forall j,\quad \lim_{p\to\infty} \frac{-t_{j,n_p}+s_p\lambda_{1,n_p}}{\lambda_{j,n_p}}=\theta^j,
\end{gather}
and $\left\{\vu(t_{n_p}+\lambda_{1,n_p}s_p)\right\}_p$ has a profile decomposition $\left(\tU^j_L,\left\{\tlambda_{j,p},\tx_{j,p},\ttt_{j,p}\right\}_p\right)_{j\geq 1}$ with nonlinear profiles $\tU^j$ where:
\begin{itemize}
 \item if $1\leq j\leq J_m$:
$$ \ttt_{j,p}=0,\quad \tlambda_{j,p}=\lambda_{1,n_p}\mu_{j,p},\quad \tx_{j,p}=x_{j,n_p}+\lambda_{1,n_p}y_{j,p},\quad \vec{\tU}\vphantom{\tU}^j_L(0)=\left(V_0^j,V_1^j\right).$$
\item if $j>J_m$, 
$$\ttt_{j,p}=t_{j,n_p}-\lambda_{1,n_p}s_p, \quad \tlambda_{j,p}=\lambda_{j,n_p},\quad \tx_{j,p}=x_{j,n_p},\quad \tU^j=U^j.$$
\end{itemize}
Furthermore the profile decomposition $\left(\tU^j_L,\left\{\tlambda_{j,p},\tx_{j,p},\ttt_{j,p}\right\}_p\right)_{j\geq 1}$ is well-ordered. More precisely, it satisfies:
\begin{equation}
\label{t_order}
(1)\backsimeq \ldots \backsimeq (J_m)\prec  (1+J_m) \preceq (2+J_m)\preceq \ldots 
\end{equation} 
Finally $\left\{t_{n_p}+\lambda_{1,n_p}s_p\right\}_p\in \SSS_3$.
\end{lemma}
\begin{remark}
 Lemma \ref{L:F1} allows to modify the times  $t_n$, adding times of slightly greater order than $\lambda_{1,n}$. This time scale is chosen large enough to affect the first profiles $U^j$, $j\leq J_{m}$, but is too small to affect the other profiles $U^j$, $j\geq 1+J_m$, as can be seen in the definitions of $\tU^j$.
\end{remark}

\begin{remark}
 By a diagonal extraction argument, we can always assume  \eqref{F3}.
\end{remark}
\begin{remark}
The assumption $J_1\left(\{t_{n}\}_n\right)=J_m$ is crucial to prove 
 the equivalence $(1)\backsimeq \ldots \backsimeq (J_m)$ for the new profile decomposition. Note that this equivalence implies that if $1\leq j,k\leq J_m$, there exists a constant $C>0$ such that $C^{-1}\leq \mu_{j,p}/\mu_{k,p}\leq C$ for all $p$. 
\end{remark}

\begin{proof}
\EMPH{Step 1} Let $j\in \{1+J_m,\ldots,J_M\}$. We prove:
\begin{gather}
 \label{F7}
\left(\forall n,\; t_{j,n}=0\right)\Longrightarrow \theta^j\in \left(T_-(U^j),T_+(U^j)\right)\\
\label{F8}
\lim_{n\to\infty} -t_{j,n}/\lambda_{j,n}=-\infty\Longrightarrow \forall p,\; \theta_p^j =-\infty\text{ and }\theta^j=-\infty.
\end{gather}
Note that the case $\lim_{n} -t_{j,n}/\lambda_{j,n}=+\infty$ is excluded since $U^j$ does not scatter forward in time.

First assume $t_{j,n}=0$ for all $n$. Since $(1)\prec (j)$, it follows easily from the definition of the relation ``$\prec$'' that one of the following holds (after extraction of subsequences in $n$):
\begin{equation*}
\lim_{n\to\infty}\frac{\lambda_{1,n}}{\lambda_{j,n}}=0\text{ or } \left(\lim_{n\to\infty} \frac{\lambda_{1,n}}{\lambda_{j,n}}=c\in (0,+\infty),\;T_+(U^1)<\infty\text{ and }cT_+(U^1)<T_+(U^j)\right).
\end{equation*} 
In the first case, we obtain $\theta_p^j=0$ for all $p$, and thus $\theta^j=0\in (T_-(U^j),T_+(U^j))$. In the second case, we have $\theta_p^j=cs_p$, and thus $\theta^j =cT_{+}(U^1)\in (0,T_+(U^j))$. Hence \eqref{F7}.

We next assume $\lim_{n} -t_{j,n}/\lambda_{j,n}=-\infty$. Since $(1)\prec (j)$, there exists $\theta<T_+(U^j)$ such that
\begin{equation*}
\lim_{n\to\infty} \frac{\lambda_{j,n}\theta+t_{j,n}}{\lambda_{1,n}}\geq T_+(U^1),
\end{equation*} 
(with the convention that this limit must be $+\infty$ if $T_+(U^1)=+\infty$). Since $\lim_{n} \lambda_{j,n}/t_{j,n}=0$, we deduce 
\begin{equation}
 \label{F9}
\lim_{n\to\infty} \frac{t_{j,n}}{\lambda_{1,n}}\geq T_+(U^1).
\end{equation} 
We fix $p$ and let $T$ such that $s_p<T<T_+(U^1)$. Then for large $n$
\begin{equation*}
 \frac{s_p\lambda_{1,n}-t_{j,n}}{\lambda_{j,n}}=\frac{s_p}{T}\frac{\lambda_{1,n}}{\lambda_{j,n}}\left(T-t_{j,n}/\lambda_{1,n}\right)+\left(\frac{s_p}{T}-1\right)\frac{t_{j,n}}{\lambda_{j,n}}\leq \left(\frac{s_p}{T}-1\right)\frac{t_{j,n}}{\lambda_{j,n}},
\end{equation*} 
where the last inequality follows from  \eqref{F9}. Using that $\lim_n -t_{j,n}/\lambda_{j,n}=-\infty$, we get 
as announced
$$\theta^j_p=\lim_{n\to\infty}\frac{s_p\lambda_{1,n}-t_{j,n}}{\lambda_{j,n}}=-\infty,$$
and thus $\theta^j=-\infty$. 

\EMPH{Step 2} By Proposition \ref{P:Z4} and Remark \ref{R:profiles}, $t_n+\lambda_{1,n}s_p<T_+(u)$ and $\frac{\lambda_{1,n}s_p-t_{j,n}}{\lambda_{j,n}}<T_+(U^j)$ for large $n$, and $\left\{\vu(t_n+\lambda_{1,n}s_p)\right\}_n$ has a profile decomposition $\left(U^{p,j}_L,\left\{\lambda_{j,n},x_{j,n},t_{p,j,n}\right\}_n\right)_j$, where $t_{p,j,n}=t_{j,n}-\lambda_{1,n}s_p$, and $U_L^{p,j}$ is the only solution of the wave equation such that
\begin{equation}
 \label{F10}
\lim_{n\to\infty}
\left\|\vU_L^{p,j}\left(\frac{\lambda_{1,n}s_p-t_{j,n}}{\lambda_{j,n}}\right)-\vU^{j}\left(\frac{\lambda_{1,n}s_p-t_{j,n}}{\lambda_{j,n}}\right)\right\|_{\hdot\times L^2}=0.
\end{equation} 
In this step, we prove:
\begin{enumerate}
 \item \label{prof_1} If $1\leq j\leq J_m$, then
\begin{equation}
 \label{F11}
\lim_{p\to\infty}\left\|\vU_L^{p,j}(s_p)-\left(\frac{1}{\mu_{j,p}^{\frac{N}{2}-1}}V_0^j\left(\frac{\cdot-y_{j,p}}{\mu_{j,p}}\right),\frac{1}{\mu_{j,p}^{\frac{N}{2}}}V_1^j\left(\frac{\cdot-y_{j,p}}{\mu_{j,p}}\right)\right)\right\|_{\hdot\times L^2}=0.
\end{equation} 
In other words, $\left\{\vU_L^{p,j}(0)\right\}_p$ has a profile decomposition with only one nonzero profile, $S_L(t)(V_0,V_1)$ and parameters $\{\mu_{j,p},y_{j,p},s_p\}_p$.
\item \label{prof_2} If $j\geq 1+J_m$ and $\theta_j\in\RR$, then $\vU_L^{p,j}(0)$ converges, as $p\to\infty$ to $\vS_L(-\theta^j)\vU^j(\theta^j)$. We define $\tU^j_L(t)=S_L(t-\theta^j)\vU^j(\theta_j)$.
\item \label{prof_3}If $j\geq 1+J_m$ and $\theta_j=+\infty$ (respectively $\theta_j=-\infty$), then $U^j$ scatters forward (respectively backward) in time, and $\vU_L^{p,j}(0)$ converges, as $p\to\infty$ to  $\tU_L^j(0)$, where $\tU_L^j$ is the unique solution of the linear wave equation such that
$$ \lim_{t\to+\infty } \left\|\tU^j_L(t)-U^j(t)\right\|_{\hdot\times L^2}=0$$
$\Big($respectively 
$$ \lim_{t\to-\infty } \left\|\tU^j_L(t)-U^j(t)\right\|_{\hdot\times L^2}=0\Big).$$
\end{enumerate}
We first assume $1\leq j\leq J_m$ and prove \eqref{prof_1}. In this case, \eqref{F10} means $\vU_L^{p,j}(s_p)=\vU^j(s_p)$ and \eqref{F11} follows from \eqref{F2}.

We next prove \eqref{prof_2}. Assume $j\geq 1+J_m$ and $\theta^j\in \RR$. Then $\theta^j\in \left(T_-(U^j),T_+(U^j)\right)$: it follows from Step 1 if $j\leq J_M$ or from the fact that $U^j$ scatters if $j\geq J_M+1$. By \eqref{F10}, $\tU_L^j(0)=\vS_L(-\theta^j_p)\vU^j(\theta^j_p)$ for all $p$. Passing to the limit, we get \eqref{prof_2}.

It remains to prove \eqref{prof_3}. First assume $j\geq J_{m}+1$ and $\theta^j=+\infty$. By Step 1, we must have $j\geq 1+J_M$. Since by the assumptions of Lemma \ref{L:F1}, $\left\{\theta^j_p\right\}_p$ is a nondecreasing sequence, one of the following holds:
\begin{gather}
 \label{F12}
\forall p,\quad \theta_p^j\in \RR\text{ and }\lim_{p\to\infty}\theta_p^j=+\infty\\
\label{F13} \exists p_0,\;\forall p\geq p_0,\quad \theta_p^j=+\infty.
\end{gather}
First assume \eqref{F12}. By \eqref{F10}, we obtain again  $\vU_L^{p,j}(0)=\vS_L(-\theta^j_p)\vU^j(\theta^j_p)$ for all $p$. Letting $p\to\infty$, we deduce the desired conclusion. Next assume \eqref{F13}. Then by \eqref{F10}, $U_L^{p,j}=\tU^j_L$, where $\lim_{t\to+\infty} \left\|\vU_L^j(t)-\vec{\tU}\vphantom{\tU}^j_L(t)\right\|_{\hdot\times L^2}=0$, which implies the announced result. 

It remains to treat the case when $j\geq 1+J_m$ and $\theta^j=-\infty$. In this case $\theta^j_p=-\infty$ for all $p$, and the proof is the same as when \eqref{F13} holds. 

\EMPH{Step 3}
In this step we choose the sequence $\{m_p\}_p$ of integers appearing in the assumptions of Lemma \ref{L:DPD}. 
We distinguish two cases.

\EMPH{Case 1} $T_+(U^1)=+\infty$. In this case, if $k\in \{1+J_m,\ldots,J_M\}$ (and thus $(1)\prec (k)$) there exists $T_k<T_+(U^k)$ such that $\lim_n \frac{\lambda_{k,n}T_k+t_{k,n}}{\lambda_{1,n}}=+\infty$. For all $p$, we let $m_p\geq m_p'$ such that
\begin{equation}
 \label{F14}
\forall n\geq m_p,\; \forall j\in \{1,\ldots,J_m\},\; \forall k\in \{1+J_m,\ldots,J_M\},\quad
\frac{1}{\mu_{j,p}}\left(\frac{\lambda_{k,n}T_{k}+t_{k,n}}{\lambda_{1,n}}\right)-\frac{s_p}{\mu_{j,p}}\geq p
\end{equation} 
and
\begin{equation}
 \label{F14'}
\forall j\in \{1,\ldots,p\},\; \forall n\geq m_p,\quad
\begin{cases}
 \left| \frac{-t_{j,n}+s_p\lambda_{1,n}}{\lambda_{j,n}}-\theta_p^j\right| \leq \frac{1}{p}&\text{ if } \theta_p^j\in \RR\\
\frac{-t_{j,n}+s_p\lambda_{1,n}}{\lambda_{j,n}}\geq  p&\text{ if }\theta_p^j=+\infty\\
\frac{-t_{j,n}+s_p\lambda_{1,n}}{\lambda_{j,n}}\leq  -p&\text{ if }\theta_p^j=-\infty.
\end{cases}
\end{equation} 

\EMPH{Case 2} $T_+(U^1)\in \RR$. In this case, for all $k\in \{1+J_m,\ldots, J_M\}$, there exists $T_k<T_+(U^k)$ such that $\lim_n \frac{\lambda_{k,n}T_k+t_{k,n}}{\lambda_{1,n}}\geq T_+(U^1)$. We choose $m_p\geq m_p'$ such that
\begin{multline}
 \label{F15} 
\forall n\geq m_p,\; \forall j\in \{1,\ldots,J_m\},\; \forall k\in \{1+J_m,\ldots,J_M\},\quad
\frac{1}{\mu_{j,p}}\left(\frac{\lambda_{k,n}T_{k}+t_{k,n}}{\lambda_{1,n}}\right)-\frac{s_p}{\mu_{j,p}}
\\ \geq \frac{T_+(U^1)-s_p}{\mu_{j,p}}-\frac{1}{p}
\end{multline}
and \eqref{F14'} hold. 

\EMPH{Step 4} We next use Lemma \ref{L:DPD}. Note that the profile decomposition of each sequence $\left\{U_L^{p,j}(0)\right\}_p$ has only one profile, given by Step 2. Note also that for large $j$, the small data theory ensures $\|U_L^{p,j}\|_{S(\RR)}\leq 2\|U^j\|_{S(\RR)}$, and (by the Pythagorean expansion \eqref{Pyt_S}) $\sum_j \|U^j\|_{S(\RR)}^{\frac{2(N+1)}{N-2}}<\infty$. This proves that the technical assumption \eqref{DPD2} of the Lemma is satisfied. 

As a consequence of Lemma \ref{L:DPD}, we obtain an increasing sequence of indices $\{n_p\}_p$, with $n_p\geq m_p$ and such that $\left\{\vu(t_{n_p}+\lambda_{1,n_p}s_p)\right\}_p$ has a profile decomposition $\left(\tU^j_L,\left\{\tlambda_{j,p},\tx_{j,p},\ttt_{j,p}\right\}_p\right)_{j\geq 1}$ with nonlinear profiles $\tU^j$, where, by Step 2,
\begin{itemize}
 \item if $1\leq j\leq J_m$, $\ttt_{j,p}=t_{j,n_p}-s_p\lambda_{1,n_p}+s_p\lambda_{1,n_p}=0$, $\tx_{j,p}=x_{j,n_p}+\lambda_{j,n_p}y_{j,p}$, $\tlambda_{j,p}=\lambda_{1,n_p}\mu_{j,p}$, $\tU_L^j(0)=\left(V_0^j,V_1^j\right)$.
\item if $j\geq 1+J_m$, $\tlambda_{j,p}=\lambda_{j,n_p}$, $\ttt_{j,p}=t_{j,n_p}-\lambda_{1,n_p}s_p$ and $\tx_{j,p}=x_{j,n_p}$, and $\tU_L^j$ is defined in Step 2, \eqref{prof_2} and \eqref{prof_3}. In this case, by \eqref{F14'}, we see that \eqref{F5} holds and
$$\lim_{p\to\infty}  \frac{-\ttt_{j,p}}{\tlambda_{j,p}}=\lim_{p\to\infty}\theta_{p}^j=\theta^j.$$
Combining with the definition of $\tU^j_L$ and $\tU^j$, we obtain immediately $\tU^j=U^j$.
\end{itemize}

\EMPH{Step 5} We prove
\begin{equation}
 \label{F16}
\forall j\in \{1,\ldots, J_m\},\; \forall k\in \{1+J_m,\ldots, J_M\},\quad \left(\tU^j_L,\left\{\tlambda_{j,p},\ttt_{j,p}\right\}_p\right)\prec \left(\tU^k_L,\left\{\tlambda_{k,p},\ttt_{k,p}\right\}_p\right).
\end{equation} 
We let $j\in \{1,\ldots, J_m\}$, $k\in \{1+J_m,\ldots, J_M\}$, and 
recall that the nonlinear profile $\tU^j$ is the solution of \eqref{CP} with data $(V^j_0,V^j_1)$ at $t=0$. Since $U^j$ does not scatter forward in time, we deduce by \eqref{F11} that $\tU^j$ does not scatter forward in time. 

We recall from Step 3 the definition of $T_k$, and we note that $T_k<T_+(U^k)=T_+(\tU^k)$ since $U^k=\tU^k$. Furthermore
\begin{equation}
 \label{F17}
\frac{\tlambda_{k,p}T_k+\ttt_{k,p}}{\tlambda_{j,p}}=\frac{\lambda_{k,n_p}T_k+t_{k,n_p}-\lambda_{1,n_p}s_p}{\mu_{j,p}\lambda_{1,n_p}}=\frac{1}{\mu_{j,p}}\left(\frac{\lambda_{k,n_p}T_k+t_{k,n_p}}{\lambda_{1,n_p}}-s_p\right).
\end{equation} 
\EMPH{First case} $T_+(U^1)=+\infty$. By \eqref{F14} and \eqref{F17},
$$\frac{\tlambda_{k,p}T_k+\ttt_{k,p}}{\tlambda_{j,p}}\geq p\underset{p\rightarrow \infty}{\longrightarrow}+\infty,$$
which proves \eqref{F16} in this case.

\EMPH{Second case} $T_+(U^1)\in \RR$. By \eqref{F15} and \eqref{F17},
$$\frac{\tlambda_{k,p}T_k+\ttt_{k,p}}{\tlambda_{j,p}}\geq \frac{T_+(U^1)-s_p}{\mu_{j,p}}-\frac 1p.$$
We claim 
\begin{equation}
\label{F18}
 \liminf_{p\to\infty} \frac{T_+(U^1)-s_p}{\mu_{j,p}}\geq T_+(\tU^j),
\end{equation} 
which will prove \eqref{F16} in this case also. Let $T<T_+(\tU^j)$. Then  by \eqref{F11} and a standard continuity property of the flow of \eqref{CP}, $s_p+\mu_{j,p}T<T_+(U^1)$ for large $p$, i.e. $\frac{T_+(U^1)-s_p}{\mu_{j,p}}>T$ for large $p$. Hence $\liminf_{p\to\infty} \frac{T_+(U^1)-s_p}{\mu_{j,p}}\geq T$, which yields \eqref{F18} since $T<T_+(U^j)$ is arbitrary.

\EMPH{Step 6. End of the proof}
To conclude the proof of Lemma \ref{L:F1}, we must check that the profile decomposition $\left(\tU^j_L,\left\{\tlambda_{j,p},\tx_{j,p},\ttt_{j,p}\right\}_p\right)_{j\geq 1}$, satisfies \eqref{t_order}.

Since $\tU^j=U^j$ if $j\geq 1+J_{m}$, we know that $\tU^j$ scatters forward in time if and only if $j\geq 1+J_M$. Thus if $k\geq 1+J_{M}$ we have $(j)\preceq (k)$.  If $1+J_m\leq j\leq k\leq J_M$, then $\ttt_{j,p}-\ttt_{k,p}=t_{j,n_p}-t_{k,n_p}$, $\tU^j=U^j$ and $\tU^k=U^k$ by Step 4, and again $(j)\preceq (k)$. By Step 5, if $j\leq J_m$ and $k\geq J_{m}+1$, $(j)\prec (k)$. We deduce $\left\{t_{n_p}+s_p\lambda_{1,n_p}\right\}_p\in \SSS_2$ and $J_{1}\left(\left\{t_{n_p}+s_p\lambda_{1,n_p}\right\}_p\right)\leq J_m$. By the definition of $J_m$, we must have $J_{1}\left(\left\{t_{n_p}+s_p\lambda_{1,n_p}\right\}_p\right)= J_m$, which proves
$$ (1)\backsimeq \ldots \backsimeq (J_m),$$
concluding the proof of \eqref{t_order}.
\end{proof}
\subsubsection{Compactness property for the first profiles}
\begin{lemma}
 \label{L:DEM4'}
 Let $u$ be as in Theorem \ref{T:profiles}. There exists $\{t_n\}_n\in \SSS_3$, such that $\{\vu(t_n)\}_n$ has a well-ordered profile decomposition $\profiles$, such that \eqref{D34} holds, and, for all $j\in \{1,\ldots,J_{m}\}$, $U^j$ has the compactness property.
\end{lemma}
\begin{proof}
Let $\{t_n\}_n\in \SSS_3$ and the profile decomposition $\profiles$ of $\{\vu(t_n)\}_n$ be given by Lemma \ref{L:DEM4}. We let $\{s_p\}_p$ be a sequence converging to $T_+(U^1)$. Using the compactness of $\overline{K}_+^j$, $j=1,\ldots, J_m$, we can extract subsequence a subsequence of $\{s_p\}_p$ (still denoted by $\{s_p\}_p$) so that for all $j\in \{1,\ldots,J_m\}$, there exists $(V_0^j,V_1^j)\in \hdot\times L^2$ with
\begin{equation*}
\lim_{p\to\infty} \left\|\nabla_{t,x}U^j(s_p)-\frac{1}{\mu_j^{\frac N2}(s_p)} \left(\nabla V_0^j,V_1^j\right)\left( \frac{\cdot-y_j(s_p)}{\mu_j(s_p)} \right)\right\|_{(L^2)^{N+1})}=0.
\end{equation*} 
Let $V^j$ be the solution of \eqref{CP} with initial data $(V_0^j,V_1^j)$. By Claim \ref{C:compactness} in Appendix \ref{A:compactness}, $V^j$ has the compactness property. By Lemma \ref{L:F1}, extracting subsequences in $p$, we can find a sequence $\{n_p\}_p$ of indices such that $\left\{t_{n_p}+s_p\lambda_{1,n_p}\right\}_p\in \SSS_3$, and $\left\{\vu(t_{n_p}+s_p\lambda_{1,n_p})\right\}_p$ has a profile decomposition $\tprofilesp$ such that 
\begin{gather*}
j=1,\ldots, J_m\Longrightarrow \tU^j=V^j ,\quad \forall p,\; \ttt_{j,p}=0\\
(1)\backsimeq\ldots \backsimeq (J_m)\prec (1+J_m)\preceq (2+J_m)\preceq\ldots.
\end{gather*} 
As in the proof of Lemma \ref{L:DEM2}, we can rescale the profiles $\tU^j$, $j=1,\ldots, J_m$ so that \eqref{D34} holds, which concludes the proof of Lemma \ref{L:DEM4'}.
\end{proof}
\subsubsection{Weak convergence to the stationary solutions}
\begin{lemma}
 \label{L:DEM5}
 Let $u$ be as in Theorem \ref{T:profiles}. There exists $\{t_n\}_n\in \SSS_3$, such that $\{\vu(t_n)\}_n$ has a well-ordered profile decomposition $\profiles$, such that \eqref{D34} holds, and, for all $j\in \{1,\ldots,J_{m}\}$, there exists $\vell_j\in B^N$, $Q^j\in \Sigma$ such that $\vU^j_L(0)=\vQ^j_{\vell_j}(0)$. In particular, $T_+(U^j)=+\infty$, $j=1,\ldots, J_m$.
\end{lemma}
\begin{proof}
We start by proving that there exists a sequence $\{t_n\}_n\in \SSS_3$ such that $\{\vu(t_n)\}_n$ has a profile decomposition $\profiles$, such that \eqref{D34} holds, and  there exists $\vell_1\in B^N$, $Q^1\in \Sigma$ such that $\vU^1_L(0)=\vQ^1_{\vell_1}(0)$.

Let $\{t_n\}_n\in \SSS_3$ and the profile decomposition $\profiles$ of $\{\vu(t_n)\}_n$ be given by Lemma \ref{L:DEM4'}. 

Since $U^1$ has the compactness property, there exists, by Proposition \ref{P:compact}, a sequence $\{s_p\}_p$ in $[0,T_+(U^1))$, $Q^1\in \Sigma$ and $\vell_1\in B^N$ such that 
$$\lim_{p\to\infty} s_p=T_+(U^1).$$
and
\begin{equation*}
\lim_{p\to\infty} \left\|\nabla_{t,x}U^1(s_p)-\frac{1}{\mu_1^{\frac N2}(s_p)} \nabla_{t,x}Q^1_{\vell_1}\left(0, \frac{\cdot-y_1(s_p)}{\mu_1(s_p)} \right)\right\|_{(L^2)^{N+1}}=0,
\end{equation*} 
where $y_1(t)$, $\mu_1(t)$ are the parameters appearing in the definition of the compactness property. Extracting again subsequences in $p$, we can assume that for all $j=2,\ldots, J_m$, there exists $(V^j_0,V^j_1)$ such that 
\begin{equation*}
\lim_{p\to\infty} \left\|\nabla_{t,x}U^j(s_p)-\frac{1}{\mu_j^{\frac N2}(s_p)} \left(\nabla V_0^j,V_1^j\right)\left( \frac{\cdot-y_j(s_p)}{\mu_j(s_p)} \right)\right\|_{(L^2)^{N+1}}=0.
\end{equation*} 
Using Lemma \ref{L:F1} as in the proof of Lemma \ref{L:DEM4'}, we get exactly the property announced in the beginning of this proof. Iterating this process, we deduce the conclusion of Lemma \ref{L:DEM5}. Note that when iterating the process to show the case $j=2,\ldots, m$, the property $\vU^k_L(0)=\vQ_{\vell_k}^k(0)$ for $k\in\{1,\ldots j-1\}$ is not lost (see \eqref{travelling_wave}) .
\end{proof}
\subsubsection{Strong convergence in Strichartz norm}
\begin{lemma}
 \label{L:DEM6}
Let $u$ be as in Theorem \ref{T:profiles}. There exists $\{t_n\}_n\in \SSS_3$ satisfying the conclusion of Lemma \ref{L:DEM5} and such that for all $T>0$, $t_n+\lambda_{1,n}T<T_+(u)$ for large $n$ and
\begin{equation}
\label{CV_S}
\lim_{n\to\infty}\int_0^T\int_{\RR^N} \left|\lambda_{1,n}^{\frac{N}{2}-1}u\left(t_n+\lambda_{1,n}t,\lambda_{1,n}x\right)-\sum_{j=1}^{J_m}Q_{\vell_j}^{j}\left(t,x-\frac{x_{j,n}}{\lambda_{1,n}}\right)\right|^{\frac{2(N+1)}{N-2}}\,dx\,dt=0.
\end{equation} 
\end{lemma}
\begin{proof}
 \EMPH{Step 1}
We prove that there exists a sequence $\{\ttt_n\}\in \SSS_3$, such that $\left\{\vu(\ttt_n)\right\}_n$ has a profile decomposition $\tprofiles$ satisfying the conclusion of Lemma \ref{L:DEM5} and such that for all $j\geq J_{M}+1$, one of the following holds:
\begin{gather}
\label{COM46}
\lim_{n\to\infty} \frac{-\ttt_{j,n}}{\tlambda_{j,n}}=\lim_{n\to\infty} \frac{-\ttt_{j,n}}{\tlambda_{1,n}}=-\infty\\
\label{COM47}
\lim_{n\to\infty} \frac{-\ttt_{j,n}}{\tlambda_{j,n}}\in \RR\text{ and }\lim_{n\to\infty} \frac{\tlambda_{1,n}}{\tlambda_{j,n}}=0\quad \text{or}\\
\label{COM48}
\lim_{n\to\infty} \frac{-\ttt_{j,n}}{\tlambda_{j,n}}=+\infty.
\end{gather}  
For this, we let $\{t_n\}_n\in \SSS_3$ and $\profiles$ the well-ordered profile decomposition of $\{\vu(t_n)\}_n$  given by Lemma \ref{L:DEM5}. We write $\{j\in \NN,\; j\geq 1+J_M\}$ as the disjoint union $G_1\cup G_2 \cup G_3 \cup B_1\cup B_2$, where 
\begin{align*}
 j\in G_1 &\iff \lim_{n\to\infty} \frac{-t_{j,n}}{\lambda_{j,n}}=\lim_{n\to\infty} \frac{-t_{j,n}}{\lambda_{1,n}}=-\infty \\
j\in G_2&\iff \lim_{n\to\infty} \frac{-t_{j,n}}{\lambda_{j,n}}\in \RR\text{ and }\lim_{n\to\infty} \frac{\lambda_{1,n}}{\lambda_{j,n}}=0 
\\
j\in G_3&\iff \lim_{n\to\infty} \frac{-t_{j,n}}{\lambda_{j,n}}=+\infty\\
j\in B_1&\iff \lim_{n\to\infty}\frac{-t_{j,n}}{\lambda_{j,n}}=-\infty \text{ and }\lim_{n\to \infty}\frac{t_{j,n}}{\lambda_{1,n}}\in [0,+\infty)\\
j\in B_2&\iff \lim_{n\to\infty}\frac{-t_{j,n}}{\lambda_{j,n}}\in \RR \text{ and }\lim_{n\to \infty}\frac{\lambda_{1,n}}{\lambda_{j,n}}\in (0,+\infty].
\end{align*}
Extracting subsequences, we can always assume that the limits appearing in the preceding definitions exist. If $\lim_n -t_{j,n}/\lambda_{j,n}=-\infty$, then $t_{j,n}$ must be positive for large $n$, which shows that $\lim_{n} t_{j,n}/\lambda_{1,n}\in [0,+\infty]$, and thus $j\in G_1\cup B_1$. This proves (after extraction) that the equality
$$\{j\in \NN,\; j\geq 1+J_M\}=G_1\cup G_2 \cup G_3 \cup B_1\cup B_2$$
holds.

If $p\geq 1$, we let $m_p'$ be an integer such that
\begin{equation}
 \label{COM49}
\left(j\in G_1,\; j\leq p\text{ and } n\geq m_p'\right)\Longrightarrow \frac{t_{j,n}}{\lambda_{1,n}}\geq p^2.
\end{equation} 
We will use Lemma \ref{L:F1} with $s_p=p$. In particular,
\begin{equation}
\label{thetapj}
\theta_p^j=\lim_{n\to\infty} \frac{-t_{j,n}+p\lambda_{1,n}}{\lambda_{j,n}} 
\end{equation} 
Let $j=1,\ldots, J_{m}$ and $\ell_j=|\vell_j|$. By \eqref{defQl},
$$U^j(t,x)=Q_{\vell_j}^j(t,x)=Q^j_{\vell_j}(0,x-t\vell_j).$$
Thus \eqref{F2} holds with $s_p=p$, $(V_0^j,V_1^j)=\vQ^j_{\vell}(0)$, $\mu_{j,p}=1$, $y_{j,p}=p\vell_j$. Furthermore using \eqref{thetapj}, the following facts are easy to check:
\begin{align*}
j\in G_1&\Longrightarrow \forall p,\; \theta_p^j=\theta^j=-\infty,& j\in G_2 &\Longrightarrow \forall p,\; \theta_p^j=\theta^j=\lim_{n\to\infty} \frac{-t_{j,n}}{\lambda_{j,n}}\in \RR\\
j\in G_3&\Longrightarrow \forall p,\; \theta_p^j=\theta^j=+\infty, & j\in B_1&\Longrightarrow \exists p_j\text{ s.t. }\forall p\geq p_j,\; \theta_p^j=\theta^j=+\infty\\
j\in B_2&\Longrightarrow \theta^j=+\infty.&&
\end{align*} 
By Lemma \ref{L:F1}, there exists an increasing sequence $\{n_p\}_p$ of indices such that $n_p\geq m_p$ and 
$\left\{\vu(t_{n_p}+p\lambda_{1,n_p})\right\}_p$ has a profile decomposition $\left(\tU_L^j,\left\{\tlambda_{j,p},\tx_{j,p},\ttt_{j,p}\right\}_p\right)$, with the property that 
$$ \forall j\geq 1+J_{m},\quad \lim_{p\to\infty}\frac{-\ttt_{j,p}}{\tlambda_{j,p}}=\theta^j,\quad \tlambda_{j,p}=\lambda_{j,n_p},\quad \tx_{j,p}=x_{j,n_p},\quad \ttt_{j,p}=t_{j,n_p}-\lambda_{1,n_p}p.$$
Denote by $\tG_1,\tG_2,\tG_3,\tB_1,\tB_2$ the analogs of $G_1,G_2,G_2,B_1,B_2$ for this profile decomposition. We will prove that $\tB_1$ and $\tB_2$ are empty, i.e. that for all $j\geq 1+J_M$, \eqref{COM46}, \eqref{COM47} or \eqref{COM48} holds. We let $j\geq 1+J_M$ and distinguish three cases.

If $j\in G_1$. Then $\lim_{p\to\infty}\frac{-\ttt_{j,p}}{\tlambda_{j,p}}=\theta^j=-\infty$, and, by \eqref{COM49},
$$ \frac{\ttt_{j,p}}{\tlambda_{1,p}}=\frac{t_{j,n_p}-p\lambda_{1,n_p}}{\lambda_{1,n_p}}=\frac{t_{j,n_p}}{\lambda_{1,n_p}}-p\geq p^2-p\underset{p\to \infty}{\longrightarrow}+\infty.$$
Thus $j\in \tG_1$. 

If $j\in G_2$. Then $\lim_{p} \frac{-\ttt_{j,p}}{\tlambda_{1,p}}=\theta_j \in \RR$ and $\lim_{p}\frac{\tlambda_{1,p}}{\tlambda_{j,p}}=\lim_{p}\frac{\lambda_{1,n_p}}{\lambda_{j,n_p}}=0$. Thus $j\in \tG_2$. 

If $j\in G_3\cup B_1\cup B_2$, then $\lim_{p} \frac{-\ttt_{j,p}}{\tlambda_{1,p}}=\theta^j=+\infty$ and thus $j\in \tG_3$. 

Hence $\tB_1=\tB_2=\emptyset$ which proves the announced result.

\EMPH{Step 2} We prove that the sequence $\{\ttt_n\}_n$ constructed in the preceding step satisfies the conclusion of the lemma. To lighten notations, we will drop the tildas, denoting this sequence $\{t_n\}_n$ and the profile decomposition $\profiles$. We note that by Lemma \ref{L:F1}, 
$$ \forall j\in \{1,\ldots,J_{1,m}\},\; \forall n,\quad  t_{j,n}=0\text{ and }(1)\backsimeq (j).$$
We can thus rescale the profiles $U^j_L$, $j=1,\ldots, J_m$ as in the proof of Lemma \ref{L:DEM2}, so that \eqref{D34} holds. We let $T>0$. By Proposition \ref{P:Z4},
\begin{multline*}
 \lambda_{1,n}^{\frac{N}{2}-1}u\left(t_n+\lambda_{1,n}t,\lambda_{1,n}x\right)=\sum_{j=1}^{J_m} Q_{\vell_j}^j \left(t,x-\frac{x_{j,n}}{\lambda_{1,n}}\right)+\sum_{j=1+J_m}^J\lambda_{1,n}^{\frac{N}{2}-1} U_n^j(\lambda_{1,n}t,\lambda_{1,n}x)+\tw_n^J,
\end{multline*}
where 
$$ \lim_{J\to\infty}\lim_{n\to\infty} \left\|\tw_n^J\right\|_{S(0,T)}=0.$$
We are thus reduced to prove that for all $j\geq 1+J_m$, 
\begin{equation}
 \label{COM50}
\lim_{n\to\infty}\int_0^T \int_{\RR^N}\left|\lambda_{1,n}^{\frac{N}{2}-1}U_n^j\left(\lambda_{1,n}t,\lambda_{1,n}x\right)\right|^{\frac{2(N+1)}{N-2}}\,dx\,dt=0.
\end{equation} 
Note that the integral term in \eqref{COM50} equals
\begin{equation}
 \label{COM51}
\int_{-\frac{t_{j,n}}{\lambda_{j,n}}}^{\frac{\lambda_{1,n}T-t_{j,n}}{\lambda_{j,n}}} \int_{\RR^N} \left|U^j(t,x)\right|^{\frac{2(N+1)}{N-2}} \,dx\,dt.
\end{equation} 
We distinguish two cases.

\EMPH{Case 1} $1+J_m\leq j\leq J_M$. Since $(1)\prec (j)$ and $T_+(U^1)=+\infty$, there exists $\theta<T_+(U^j)$ such that
\begin{equation}
 \label{COM52}
\lim_{n\to\infty} \frac{\lambda_{j,n}\theta+t_{j,n}}{\lambda_{1,n}}=+\infty.
\end{equation} 
If $\lim_n\lambda_{j,n}/\lambda_{1,n}=+\infty$, then the size of the time interval in \eqref{COM51} goes to $0$ as $n$ goes to infinity, and we obtain that \eqref{COM50} holds. If not, then (extracting in $n$ if necessary), we get that $\lambda_{j,n}/\lambda_{1,n}$ is bounded, and \eqref{COM52} implies 
\begin{equation*}
 \forall \theta'\in\RR,\quad 
\lim_{n\to\infty} \frac{\lambda_{j,n}\theta'+t_{j,n}}{\lambda_{1,n}}=+\infty.
\end{equation*} 
Fixing $\theta'\in (T_-(U^j),T_+(U^j))$, we deduce that for large $n$
$$ \frac{\lambda_{1,n}T-t_{j,n}}{\lambda_{j,n}}\leq \frac{\lambda_{1,n}\left(\frac{\lambda_{j,n}\theta'+t_{j,n}}{\lambda_{1,n}}\right)-t_{j,n}}{\lambda_{j,n}}=\theta'.$$
Thus for large $n$,
$$\int_{-\frac{t_{j,n}}{\lambda_{j,n}}}^{\frac{\lambda_{1,n}T-t_{j,n}}{\lambda_{j,n}}} \int_{\RR^N} \left|U^j(t,x)\right|^{\frac{2(N+1)}{N-2}} \,dx\,dt
\leq \int_{-\frac{t_{j,n}}{\lambda_{j,n}}}^{\theta'} \int_{\RR^N} \left|U^j(t,x)\right|^{\frac{2(N+1)}{N-2}} \,dx\,dt,$$
which proves, since $\theta'$ is arbitrary in $(T_-(U^j),T_+(U^j))$, that \eqref{COM50} holds. 

\EMPH{Case 2} $j\geq 1+J_M$. Then (with the notations of Step 1), $j\in G_1\cup G_2\cup G_3$. Using:
\begin{align*}
j\in G_1&\Longrightarrow \lim_{n\to\infty}\frac{\lambda_{1,n}T-t_{j,n}}{\lambda_{j,n}}=-\infty\\
j\in G_2&\Longrightarrow \lim_{n\to\infty}\frac{\lambda_{1,n}}{\lambda_{j,n}}=0\\
j\in G_3&\Longrightarrow \lim_{n\to\infty}\frac{-t_{j,n}}{\lambda_{j,n}}=+\infty,
\end{align*}
we obtain that the integral \eqref{COM51} goes to $0$ as $n$ goes to infinity, i.e. that \eqref{COM50} holds. The proof is complete.
\end{proof}
\begin{remark}
 In Step 1 of the proof of Lemma \ref{L:DEM6}, we have used in a crucial manner that the profiles $U^j=Q_{\ell_j}^j$, $j=1,\ldots,J_m$ are globally defined for positive time. 
\end{remark}

\subsubsection{Strong local convergence in the energy space}
We conclude the proof of Theorem \ref{T:profiles} by the following Lemma:
\begin{lemma}
 \label{L:DEM7}
Let $u$ be as in Theorem \ref{T:profiles}. There exists $\{t_n\}_n\in \SSS_3$ satisfying the conclusion of Lemma \ref{L:DEM6} and such that furthermore
$$ \forall j\in \{1,\ldots,J_m\},\; \forall R>0,\quad \lim_{n\to\infty} \int_{|x|\leq R} \left|\lambda_{1,n}^{\frac{N}{2}}\nabla_{t,x}u(t_n,\lambda_{1,n}x+x_{j,n})-\nabla_{t,x}Q^j_{\vell_j}(0,x)\right|^2\,dx=0.$$
\end{lemma}
\begin{proof}
 \EMPH{Step 1} Let $\{t_n\}_n$ be given by Lemma \ref{L:DEM6}. 
 Let, 
 \begin{equation}
  \label{DEM53}
  v_n(t,x)=\lambda_{1,n}^{\frac{N}{2}-1}u(t_n+\lambda_{1,n}t,\lambda_{1,n}x)-\sum_{j=1}^{J_{m}} Q_{\vell_j}^j\left(t,x-\frac{x_{j,n}}{\lambda_{1,n}}\right).
 \end{equation} 
 By Proposition \ref{P:Z4}, if $T$ is fixed, then $v_n(t,x)$ is defined for large $n$ and $t\in [0,T]$. 
 
 If $\lim_{n\to\infty}\|\vv_n(0,x)\|_{\hdot\times L^2}=0$ then we are done. Thus we can assume (extracting subsequences if necessary) that there exists $\eps>0$ such that
 \begin{equation}
  \label{DEM53'}
  \forall n,\quad \|\vv_n(0,x)\|_{\hdot\times L^2}\geq \eps.
 \end{equation} 
 By \eqref{CV_S} in Lemma \ref{L:DEM6}, for all $T>0$, 
 \begin{equation}
 \label{DEM54}
 \limsup_{n\to\infty} t_n+\lambda_{1,n}T<T_+(u)\text{ and }
  \lim_{n\to\infty}\int_0^T\int_{\RR^N} \left|v_n(t,x)\right|^{\frac{2(N+1)}{N-2}}\,dx\,dt=0.
 \end{equation} 
 We next notice that for all $n\geq 1$, there exists $T_n>0$ such that $t_n+\lambda_{1,n}T_n<T_+(u)$ and
 \begin{equation}
 \label{DEM55}
 \int_0^{T_n}\int_{\RR^N} \left|v_n(t,x)\right|^{\frac{2(N+1)}{N-2}}\,dx\,dt=\frac{1}{T_n}.
 \end{equation} 
 Indeed $T\mapsto \int_0^{T}\int \left|v_n(t,x)\right|^{\frac{2(N+1)}{N-2}}\,dx\,dt-\frac 1T$ is continuous on $\left(0,(T_+(u)-t_n)/\lambda_{1,n}\right)$, goes to $-\infty$ as $T\to 0$ and has a limit $L_n$ in $(0,+\infty]$ as $T\to (T_+(u)-t_n)/\lambda_{1,n}$. Note that the case $L_n=0$ is excluded by \eqref{DEM53'}. Obviously, by \eqref{DEM54} and \eqref{DEM55},
 \begin{equation}
  \label{DEM56}
  \lim_{n\to\infty} T_n=+\infty,\quad \lim_{n\to\infty} \int_0^{T_n} \int_{\RR^N} \left|v_n(t,x)\right|^{\frac{2(N+1)}{N-2}}\,dx\,dt=0.
 \end{equation} 
 We let $v_{L,n}(t,x)=S_L(t)\left(\vv_{n}(0,x)\right)$. In view of \eqref{DEM56}, by the Cauchy theory for \eqref{CP},
 \begin{equation}
  \label{DEM57}
  \sup_{t\in [0,T_n]} \left\|\vv_n(t,x)-\vv_{L,n}(t,x)\right\|_{\hdot\times L^2}+\int_0^{T_n} \int_{\RR^N} \left| v_{L,n}(t,x)\right|^{\frac{2(N+1)}{N-2}}\,dx\,dt\underset{n\to\infty}{\longrightarrow}0.
 \end{equation} 
 
 \EMPH{Step 2} We prove (after extraction of subsequences in $n$) that there exists a sequence $\{s_p\}_p$ in $(0,+\infty)$ such that
 \begin{equation}
  \label{DEM58}
  \forall p\geq 1,\quad \lim_{n\to\infty} \sum_{j=1}^{J_m} \int_{\left|x-x_{j,n}/\lambda_{j,n}\right|\leq p} \left|\nabla_{t,x} v_n(s_p,x)\right|^2\,dx\leq \frac{1}{p}.
 \end{equation} 
 We fix $p$. In view of \eqref{DEM57}, it is sufficient to prove \eqref{DEM58} with $v_n$ replace by $v_{L,n}$. Recall that for all $R>0$, there exists a constant $C_R$ such that any finite-energy solution $w$ of the linear wave equation satisfies (see e.g. \cite{KePoVe95}), 
 \begin{equation*}
  \int_{-\infty}^{+\infty} \int_{|x|\leq R}|\nabla_{t,x}w(t,x)|^2\,dx\,dt\leq C_{R}\left\|\vw(0)\right\|_{\hdot\times L^2}^2,
 \end{equation*} 
Fixing $\vell\in B^N$, and applying the preceding inequality to the Lorentz tranform of $w$
$$w_{-\vell}(t,x)=w\left(\frac{t+\vell \cdot x}{\sqrt{1-\ell^2}},\Big[\frac{t}{\sqrt{1-\ell^2}}+\frac{1}{\ell^2} \Big(\frac{1}{\sqrt{1-\ell^2}}-1\Big)\vell\cdot x\Big]\vell+x\right)$$
we obtain that there exists a constant $C_{R,\vell}$ such that:
 \begin{equation}
  \label{DEM59}
  \int_{-\infty}^{+\infty} \int_{|x-\vell t|\leq R}|\nabla_{t,x}w(t,x)|^2\,dx\,dt\leq C_{R,{\vell}}\left\|\vw(0)\right\|_{\hdot\times L^2}^2.
 \end{equation} 
 As a consequence,
 \begin{equation}
  \label{DEM60}
\int_0^{+\infty} \sum_{j=1}^{J_m} \int_{\left|x-x_{j,n}/\lambda_{1,n}-\vell_jt\right|\leq p} \left|\nabla_{t,x} v_{L,n}(t,x)\right|^2\,dx\,dt\leq C\|v_{L,n}(0,x)\|^2_{\hdot\times L^2},
 \end{equation}
where $C=\sum_{j=1}^{J_m}C_{p,\vell_j}$.
By \eqref{def_M} and the Pythagorean expansions \eqref{Pyt1}, \eqref{Pyt2}, 
$$\sup_n\left\|\vv_{L,n}(0)\right\|^2_{\hdot\times L^2}\leq B<\infty.$$
Coming back to \eqref{DEM60}, we get 
 \begin{equation*}
\int_0^{+\infty} \sum_{j=1}^{J_m} \int_{\left|x-x_{j,n}/\lambda_{1,n}-\vell_jt\right|\leq p} \left|\nabla_{t,x} v_{L,n}(t,x)\right|^2\,dx\,dt\leq CB.
 \end{equation*}
By Fatou's Lemma,
 \begin{equation*}
\int_0^{+\infty} \liminf_{n\to\infty}\left(\sum_{j=1}^{J_m} \int_{\left|x-x_{j,n}/\lambda_{1,n}-\vell_jt\right|\leq p} \left|\nabla_{t,x} v_{L,n}(t,x)\right|^2\,dx\right)dt\leq C B.
 \end{equation*}
As a consequence, there exists $s_p>0$ such that 
$$\liminf_{n\to\infty}\sum_{j=1}^{J_m} \int_{\left|x-x_{j,n}/\lambda_{1,n}-\vell_js_p\right|\leq p} \left|\nabla_{t,x} v_{L,n}(s_p,x)\right|^2\,dx\leq \frac{1}{p}.$$
Extracting a subsequence in $n$, we can assume
\begin{equation}
  \label{DEM61}
\lim_{n\to\infty}\sum_{j=1}^{J_m} \int_{\left|x-x_{j,n}/\lambda_{1,n}-\vell_js_p\right|\leq p} \left|\nabla_{t,x} v_{L,n}(s_p,x)\right|^2\,dx\leq \frac{1}{p}. 
\end{equation} 
To get that there exist subsequences such that \eqref{DEM61} for all $p$ (and thus such that \eqref{DEM58} holds), one needs to use a classical diagonal extraction argument. We omit the details. 
 
\EMPH{Step 3} By Lemma \ref{L:DEM6} and Step 2, for all $p$, there exists $m_p'\in \NN$ such that
\begin{gather}
 \label{DEM62}
\forall n\geq m_p',\quad  \sum_{j=1}^{J_m} \int_{\left|x-x_{j,n}/\lambda_{1,n}-\vell_js_p\right|\leq p} \left|\nabla_{t,x} v_{L,n}(s_p,x)\right|^2\,dx\leq \frac{2}{p}\quad \text{and }\\
\label{DEM63}
\forall n\geq m_p',\quad  \int_0^{s_p+p}\int_{\RR^N} \left|\lambda_{1,n}^{\frac{N}{2}-1}u\left(t_n+\lambda_{1,n}t,\lambda_{1,n}x\right)-\sum_{j=1}^{J_m}Q_{\vell_j}^{j}\left(t,x-\frac{x_{j,n}}{\lambda_{1,n}}\right)\right|^{\frac{2(N+1)}{N-2}}\,dx\,dt\leq \frac{1}{p}.
\end{gather}
We note also that the nonlinear profiles $U^j$ of the profile decomposition $\profiles$ of $\{\vu(t_n)\}_n$ satisfy $U^j=Q^j_{\vell_j}$ for $j=1,\ldots, J_m$, and thus
$$ U^j(s_p,x)=Q^j_{\vell_j}(s_p,x)=Q^j_{\vell_j}(0,x-s_p\vell_j).$$
Thus all the assumptions of Lemma \ref{L:F1} are satisfied. There exists an increasing sequence of indices $\{n_p\}_p$ such that $n_p\geq m_p'$ and $\left\{\vu(t_{n_p}+s_p\lambda_{1,n_p})\right\}_p$ has a profile decomposition $$\tprofilesp$$
 satisfying \eqref{t_order}, and such that, if $j=1,\ldots, J_m$, 
 $$\tU^j=Q^j_{\vell_j}\text{ and } \forall p,\quad \tx_{j,p}=x_{j,n_p}+\lambda_{1,n_p}s_p\vell_j,\;\ttt_{j,p}=0, \; \tlambda_{j,p}=\lambda_{j,n_p}. $$
Note that $\frac{\tx_{j,p}}{\lambda_{1,n_p}}=\frac{x_{j,n_p}}{\lambda_{1,n_p}}+s_p\ell_j$,
 and also that $\tlambda_{j,p}=\tlambda_{1,p}=\lambda_{1,n_p}$ for $j=1,\ldots, J_m$. Finally, by \eqref{DEM62} and \eqref{DEM63}, the new times sequence $\{t_{n_p}+s_p\lambda_{1,n_p}\}_p$ satisfy all the conclusions of Lemma \ref{L:DEM7}. The proof of Theorem \ref{T:profiles} is complete. 
\end{proof}

\section{Nonlinear Schr\"odinger equations}
\label{S:NLS}
In this section we give an application of our strategy to nonlinear Schr\"odinger equations (NLS). For brevity, we will only consider the energy-critical equation. Let us mention however that results analogous to Theorem \ref{T:profiles} and Theorem \ref{T:NLS} below are also available in subscritical contexts, for example for the focusing mass-supercritical, energy-subcritical NLS in all dimensions, when the initial data is taken in $H^1$. We refer to \cite{HoRo07,DuHoRo08,FaXiCa11,Guevara12P,AkNa13} and references therein for profile decompositions, concentration/compactness arguments and rigidity theorems in this setting. Note that in a subcritical context, this type of arguments are simpler since there is no scaling parameters in the profile decompositions. 

Let again $N\in \{3,4,5\}$ and recall the critical nonlinear Schr\"odinger equation on $\RR^N$
\begin{equation}
\label{NLS'}
\left\{
\begin{gathered}
i\partial_tu+\Delta u=-|u|^{\frac{4}{N-2}}u,\\
u_{\restriction t=0}=u_0\in \hdot(\RR^N).
\end{gathered}
\right.
\end{equation} 
We recall that this equation is locally well-posed in $\hdot(\RR^N)$. The equation is invariant by scaling:
\begin{equation}
 \label{NLS_scaling}
u(t,x)\text{ solution }\Longrightarrow \lambda^{\frac{N}{2}-1}u(\lambda^2t,\lambda x)\text{ solution.}
\end{equation} 
Furthermore, denoting by $(T_-(u),T_+(u))$ the maximal interval of existence of $u$, and, if $I$ is an interval, $S(I)=L^{\frac{2(N+2)}{N-2}}\left(I\times\RR^N\right)$, we have the following scattering/blow-up criterion:
if $\|u\|_{S([0,T_+(u)))}$ is finite, then $T_+(u)=+\infty$
and $u$ scatters forward in time to a solution of the linear Schr\"odinger equation. We refer to \cite{CaWe90}, and to \cite{KeMe06} and references therein for details on the well-posedness theory. 

We say that a solution $u$ of \eqref{NLS'} has the \emph{compactness property} if there exist $\lambda(t)>0$, $x(t)\in \RR^N$, defined for $t\in (T_-(u),T_+(u))$ and such that
$$ K=\left\{ \lambda(t)^{\frac{N}{2}-1}u(t,\lambda(t)\cdot+x(t)),\; t\in (T_-(u), T_+(u))\right\}$$
has compact closure in $\hdot\times L^2$. We note from case 1 of the proof of Proposition 5.3 in \cite{KeMe06} that a solution of \eqref{NLS'} with the compactness property is always global. However, no analog of Proposition \ref{P:compact} is known for equation \eqref{NLS'}. We claim:
\begin{theo}
 \label{T:NLS}
Let $u$ be a solution of \eqref{NLS'}. Assume that $u$ does not scatter forward in time and 
$$ \sup_{t\in [0,T_+(u))} \|u(t)\|_{\hdot}<\infty.$$
Then there exists a sequence of times $\{t_{n}\}_n$ in $[0,T_+(u))$, an integer $J\geq 1$, solutions $U^1,\ldots, U^J$ of \eqref{NLS'} with the compactness property, a sequence $\{\lambda_n\}_n$ in $(0,+\infty)$, $J$ sequences $\{x_{j,n}\}_{n}$, $j=1,\ldots, J$ in $\RR^N$ such that:
\begin{equation*}
 \lim_{n\to\infty} t_n=T_+(u),\quad
1\leq j<k\leq J\Longrightarrow \lim_{n\to\infty}\frac{x_{j,n}-x_{k,n}}{\lambda_{n}}=+\infty,
\end{equation*}
and
\begin{itemize}
\item $\ds \forall j\in \{1,\ldots, J\}$, $\lambda_{n}^{\frac{N}{2}-1}u\left(t_n,\lambda_{n}\cdot+x_{j,n}\right)\xrightharpoonup[n\to \infty]{} U^j(0)$ weakly in $\hdot$.
 \item for all $T>0$, $\lambda_n^2T+t_n<T_+(u)$ for large $n$ and 
\begin{equation}
\label{NLS_CV1}
\lim_{n\to\infty}\int_0^T\int_{\RR^N} \left|\lambda_{n}^{\frac{N}{2}-1}u\left(t_n+\lambda_{n}^2t,\lambda_{n}x\right)-\sum_{j=1}^{J}U^{j}\left(t,x-\frac{x_{j,n}}{\lambda_{n}}\right)\right|^{\frac{2(N+2)}{N-2}}\,dx\,dt=0.
\end{equation} 
\item Furthermore, if $u$ is radial, then $J=1$, $U^1$ is radial, $x_{1,n}=0$ for all $n$, and for all $R>0$, 
\begin{equation}
\label{NLS_CV2}
\lim_{n\to\infty} \int_{|x|\leq R} \left|\lambda_{n}^{\frac{N}{2}}\nabla u\left(t_n,\lambda_{n}x\right)-\nabla U^{1}(0,x)\right|^2\,dx=0.
\end{equation}
\end{itemize}
\end{theo}
We omit the proof, which is very close to the proof of Theorem \ref{T:profiles}, replacing the profile decomposition of H.~Bahouri and P.~G\'erard by the profile decomposition of S.~Keraani \cite{Keraani01} adapted to the energy-critical NLS. Note that since no analog of Proposition \ref{P:compact} is available for NLS, one must skip Lemma \ref{L:DEM5} in this proof. Note also that in view of the scaling \eqref{NLS_scaling} of the equation, the condition \eqref{Z7} in the definition of the pre-order relation of \S \ref{SS:preorder} becomes:
\begin{equation*}
 \forall T\in \RR,\quad T<T_+(U^j)\Longrightarrow \lim_{n\to\infty} \frac{\lambda_{j,n}^2T+t_{j,n}-t_{k,n}}{\lambda_{k,n}^2}<T_+(U^k).
\end{equation*} 
Let us mention that to prove \eqref{NLS_CV1}, we need, as in the proof of Lemma \ref{L:DEM6}, that the solution $U^j$ with the compactness property is global, a fact that is known for \eqref{NLS'} as mentioned above. To prove \eqref{NLS_CV2}, one must replace the inequality \eqref{DEM59} used in the proof of Lemma \ref{L:DEM7} by:
$$ \int_{-\infty}^{+\infty} \int_{|x|\leq R} \left|\left(\nabla e^{it\Delta} u_0\right)(x)\right|^2 \,dx\,dt\leq C_R\|u_0\|_{\hdot}^2,$$
which follows immediately from the local smoothing effect for the linear Schr\"odinger equation (see \cite{CoSa89,Sj87,Ve88}). The analog of \eqref{NLS_CV2} in the nonradial case would require a control of the space translation parameters of the solutions with the compactness property, which is not known for equation \eqref{NLS'}, except in the case $N\geq 5$ (see \cite{KiVi10}). 

\appendix

\section{Total preorder on the profiles' indices}
\label{A:preorder}
In this Appendix we prove Claims \ref{C:Z2} and Claim \ref{C:fact_preorder}. We start by proving that we can extract sequences, so that the limit appearing in \eqref{Z7} always exists.
\begin{claim}
 \label{C:subsequence}
Let $\{\lambda_{j,n},t_{j,n}\}_n$, $j\geq 1$ be sequences in $(0,+\infty)\times \RR$. Then we can extract subsequences (in $n$) such that for all $j,k\geq 1$, for all $T\in \RR$, the limit
\begin{equation}
 \label{the_limit}
\lim_{n\to\infty} \frac{\lambda_{j,n}T+t_{j,n}-t_{k,n}}{\lambda_{k,n}}
\end{equation} 
exists in $\RR\cup \{\pm\infty\}$.
 \end{claim}
 \begin{proof}
 By a standard diagonal extraction argument it is sufficient to fix $j$ and $k$ and to prove that the limit \eqref{the_limit} exists for all $T$. We extract sequences, so that
 $$ \ell_1=\lim_{n\to \infty}\frac{\lambda_{j,n}}{\lambda_{k,n}},\;\ell_2=\lim_{n\to\infty}\frac{t_{j,n}-t_{k,n}}{\lambda_{k,n}}\text{ and }\ell_3=\lim_{n\to \infty}\frac{t_{j,n}-t_{k,n}}{\lambda_{j,n}}$$ 
 exist in $\RR\cup\{\pm \infty\}$. If $\ell_1$ or $\ell_2$ is finite, the limit \eqref{the_limit} exists trivially for all $T$. Assume that $$\ell_1=\ell_2=+\infty$$
 (the case $\ell_1=+\infty$, $\ell_2=-\infty$ will follow, replacing $t_{j,n}$ and $t_{k,n}$ by $-t_{j,n}$ and $-t_{k,n}$). Then, as $n\to\infty$,
 $$ \frac{\lambda_{j,n}T+t_{j,n}-t_{k,n}}{\lambda_{k,n}}\sim\begin{cases}
 \frac{t_{j,n}-t_{k,n}}{\lambda_{k,n}}&\text{ if }T=0\text{ or }\ell_3=\infty\\
 \frac{\lambda_{j,n}}{\lambda_{k,n}}(T+\ell_3)&\text{ if }\ell_3\in [0,\infty)\text{ and }T\neq 0.
 \end{cases}$$
 and the limit \eqref{the_limit} always exists in this case also.
\end{proof}
\begin{proof}[Proof of Claim \ref{C:Z2}]
\EMPH{Proof of \eqref{I:Z2_1}}
If $U^j$ scatters, we obtain immediately that $(j)\backsimeq (j)$. If $U^j$ does not scatter, \eqref{Z7} with $k=j$ means
$$ T<T_+(U^j)\Longrightarrow T<T_+(U^j)$$
which is tautological. Thus again $(j)\backsimeq (j)$.

\EMPH{Proof of \eqref{I:Z2_2}}
Assume $(j)\preceq (k)$ and $(k)\preceq (\ell)$. Let $T<T_+(U^j)$. 

First assume that $U^{\ell}$ scatters forward in time. Then $(j)\preceq (\ell)$ follows immediately from the definition.

Next assume that $U^{\ell}$ does not scatter forward in time. Then by \eqref{imply_scattering} $U^{k}$ and $U^{j}$ do not scatter forward in time. Since $(j)\preceq (k)$, there exists $\tau<T_+(U^k)$ such that for large $n$,
\begin{equation}
 \label{G2}
 \frac{\lambda_{j,n}T+t_{j,n}-t_{k,n}}{\lambda_{k,n}}\leq \tau.
\end{equation} 
Since $(k)\preceq (\ell)$ and $\tau<T_+(U^k)$, there exists $\tau'<T_+(U^{\ell})$ such that for large $n$
\begin{equation}
 \label{G3}
 \frac{\lambda_{k,n}\tau+t_{k,n}-t_{\ell,n}}{\lambda_{\ell,n}}\leq \tau'. 
\end{equation} 
Using successively \eqref{G2} and \eqref{G3}, we get that for large $n$
\begin{equation*}
 \frac{\lambda_{j,n}T+t_{j,n}-t_{\ell,n}}{\lambda_{\ell,n}}=\frac{\left(\frac{\lambda_{j,n}T+t_{j,n}-t_{k,n}}{\lambda_{k,n}}\right)\lambda_{k,n} +t_{k,n}-t_{\ell,n}}{\lambda_{\ell,n}}\leq \frac{\tau\lambda_{k,n}+t_{k,n}-t_{\ell,n}}{\lambda_{\ell,n}}\leq \tau'.
\end{equation*}
This proves $(j)\preceq (\ell)$. 

\EMPH{Proof of \eqref{I:Z2_3}}

We assume that $(j)\preceq (k)$ does not hold, and prove $(k)\preceq (j)$. 

The profile $U^k$ does not scatter for positive time, and there exists $T<T_+(U^j)$ such that
\begin{equation}
 \label{G4}
 \lim_{n\to \infty} \frac{\lambda_{j,n}T+t_{j,n}-t_{k,n}}{\lambda_{k,n}}\geq T_+(U^k).
\end{equation} 
Let $\tau<T_+(U^k)$. By \eqref{G4}, for large $n$, $\tau\leq \frac{\lambda_{j,n}T+t_{j,n}-t_{k,n}}{\lambda_{k,n}}$, i.e
$$ \frac{\lambda_{k,n}\tau+t_{k,n}-t_{j,n}}{\lambda_{j,n}}\leq T<T_+(U^j)$$ 
for large $n$, which proves $(k)\preceq (j)$, concluding the proof of Claim \ref{C:Z2}.

\EMPH{Proof of \eqref{I:Z2_4}}
It is a general properties of preorder relations, that follows easily from \eqref{I:Z2_3}.
\end{proof}

\begin{proof}[Proof of Claim \ref{C:fact_preorder}]
We first assume \eqref{DEM40}, and get a contradiction. Since $(j)\backsimeq (k)$, and $t_{j,n}=0$, we have $T_+(U^j)>0$ and
\begin{gather}
 \label{DEM44}
 T<T_+(U^k)\Longrightarrow \lim_{n\to\infty}\frac{\lambda_{k,n}T+t_{k,n}}{\lambda_{j,n}}<T_+(U^j)\\
 \label{DEM45}
 T<T_+(U^j)\Longrightarrow \lim_{n\to\infty}\frac{\lambda_{j,n}T-t_{k,n}}{\lambda_{k,n}}<T_+(U^k) 
\end{gather}
Fix $T<T_+(U^k)$. By \eqref{DEM40},
$$\lim_{k\to\infty} \frac{T+t_{k,n}/\lambda_{k,n}}{t_{k,n}/\lambda_{k,n}}=1.$$
Combining with \eqref{DEM44} we get
\begin{equation*}
\lim_{n\to\infty}\frac{t_{k,n}}{\lambda_{j,n}}=\lim_{n\to\infty}\frac{\lambda_{k,n}T+t_{k,n}}{\lambda_{j,n}}\times \frac{t_{k,n}/\lambda_{k,n}}{T+t_{k,n}/\lambda_{k,n}} <T_+(U^j).
\end{equation*}
Let $\theta\in \RR$ such that
\begin{equation}
 \label{DEM46}
 \lim_{n\to\infty} \frac{t_{k,n}}{\lambda_{j,n}} <\theta<T_+(U^j).
\end{equation} 
We have
\begin{equation}
 \label{DEM47}
 \lim_{n\to\infty} \frac{\lambda_{k,n}}{\lambda_{j,n}}=\lim_{n\to\infty}\frac{\lambda_{k,n}}{t_{k,n}}\times \frac{t_{k,n}}{\lambda_{j,n}}=0.
\end{equation}
Combining \eqref{DEM46} and \eqref{DEM47}, we get
\begin{equation*}
\lim_{n\to\infty} \frac{\lambda_{j,n}\theta-t_{k,n}}{\lambda_{k,n}}=\lim_{n\to\infty}\frac{\lambda_{j,n}\left(\theta-t_{k,n}/\lambda_{j,n}\right)}{\lambda_{k,n}}=+\infty,
 \end{equation*} 
 which contradicts \eqref{DEM45}, proving as announced that \eqref{DEM40} cannot hold.
 
We next assume that $t_{j,n}=t_{k,n}=0$ for all $n$ and prove that there exists $c\in (0,+\infty)$ such that \eqref{DEM42} and \eqref{DEM43} hold. By our assumptions, $T_+(U^j)>0$, $T_+(U^k)>0$ and
\begin{gather}
 \label{DEM48}
 T<T_+(U^k)\Longrightarrow \lim_{n\to\infty} \frac{\lambda_{k,n} T}{\lambda_{j,n}}<T_+(U^j)\\
\label{DEM49}
 T<T_+(U^j)\Longrightarrow \lim_{n\to\infty} \frac{\lambda_{j,n} T}{\lambda_{k,n}}<T_+(U^k).
\end{gather}
Letting $c=\lim_{n}\lambda_{j,n}/\lambda_{k,n}\in [0,+\infty]$, we see immediately that $c>0$ by \eqref{DEM48} and $c<\infty$ by \eqref{DEM49}. Furthermore, if $T_+(U^k)<\infty$, then \eqref{DEM48} implies $c^{-1}T_+(U^k)\leq T_+(U^j)$ and \eqref{DEM49} implies $cT_+(U^j)\leq T_+(U^k)$, which concludes the proof.
\end{proof}

\section{A sufficient condition for profile decomposition}
\label{A:profile}
In this appendix we prove Claim \ref{C:profile}. Let $\eps>0$. Let $\tJ>0$  such that
\begin{equation}
\label{AP3}
\sum_{j\geq \tJ} \left\|U_L^j\right\|^{\SN}_{S(\RR)}\leq \left(\frac{\eps}{2}\right)^{\SN}. 
\end{equation} 
Let $k$ such that $J_k\geq \tJ$ and
\begin{equation}
 \label{AP4} \limsup_{n\to\infty} \left\|u_{L,n}-\sum_{j=1}^{J_k} U_{L,n}^j\right\|_{S(\RR)}\leq \frac{\eps}{2}.
\end{equation} 
Let $J\geq J_k$, and $w_n^J=u_{L,n}-\sum_{j=1}^J U_{L,n}^j$. Then 
$$w_n^J=u_{L,n}-\sum_{j=1}^{J_k} U_{L,n}^j-\sum_{j=J_k+1}^J U_{L,n}^j$$
where, using the orthogonality of the sequences,
$$\limsup_{n\to\infty}\left\|\sum_{j=J_{k+1}}^J U_{L,n}^j\right\|^{\SN}_{S(\RR)}=\sum_{j=J_k+1}^J \left\|U_L^j\right\|^{\SN}_{S(\RR)}\leq \left(\frac 12 \eps\right)^{\SN}.$$
To obtain the last inequality, we have used \eqref{AP3} and the fact that $J_k$ is greater than $\tJ$. Combining with \eqref{AP4}, we deduce
$$\forall J\geq J_k,\quad \limsup_{n\to\infty}\left\|w_n^J\right\|_{S(\RR)}\leq \eps$$
which concludes the proof of the claim. \qed

\section{Preservation of the compactness property}
\label{A:compactness}
In this appendix, we prove:
\begin{claim}
\label{C:compactness}
 Let $u$ be a solution of \eqref{CP} such that there exists $\lambda(t)>0$, $x(t)\in \RR^N$, defined for $t\in [0,T_+(u))$ such that
$$K_+=\left\{ \left( \lambda^{\frac{N}{2}-1}(t)u\left(t,\lambda(t)\cdot+x(t)\right), \lambda^{\frac N2}(t)\partial_t u\left(t,\lambda(t)\cdot+x(t)\right)\right),\; t\in [0,T_+(u))\right\}$$
has compact closure in $\hdot\times L^2$. Let $\{t_n\}_n$ be a sequence of times in $[0,T_+(u))$ such that $\lim_nt_n=T_+(u)$ and assume that there exists $(v_0,v_1)\in \hdot\times L^2$ such that 
\begin{equation}
 \label{conv_to_v}
\lim_{n\to\infty} \left\|\lambda(t_n)^{\frac{N}{2}}\nabla_{t,x}u\left(t_n,\lambda(t_n)\cdot+x(t_n)\right)-(v_1,\nabla v_0)\right\|_{(L^2)^{N+1}}=0.
\end{equation} 
Let $v$ be the solution of \eqref{CP} with initial data $(v_0,v_1)$ at $t=0$. Then $v$ has the compactness property.
\end{claim}
\begin{proof}
This is classical. We give a proof for the sake of completeness.

If $(u_0,u_1)=(0,0)$, the conclusion is obvious. We assume $(u_0,u_1)\neq (0,0)$.
 Let $s\in (T_-(v),T_+(v))$.  

\EMPH{Step 1} We prove by contradiction that $t_n+\lambda(t_n)s\geq 0$ for large $n$. If not, extracting subsequences, we can assume 
\begin{equation}
 \label{Ap1}
\forall n,\quad t_n+\lambda(t_n)s<0.
\end{equation} 
Let $s_n=-t_n/\lambda(t_n)$. Then by \eqref{Ap1} and since $t_n$ is positive, $s\leq s_n\leq 0$. Extracting subsequences again, we can assume
$$\lim_{n\to\infty} s_n =\theta\in [s,0]\subset(T_-(v),T_+(v)).$$
By \eqref{conv_to_v} and a standard continuity property of the flow of \eqref{CP},
$$ \left\|\lambda(t_n)^{\frac{N}{2}}\nabla_{t,x}u\left(0,\lambda(t_n)\cdot+x(t_n)\right)-\nabla_{t,x} v(\theta)\right\|_{(L^2)^{N+1}}=0.$$
Since $(u_0,u_1) \neq 0$, this proves that there exists a constant $C>0$ such that $C^{-1}\leq \lambda(t_n)\leq C$ for all $n$, contradicting \eqref{Ap1} if $T_+(u)=+\infty$, or the fact that $\lambda(t_n)$ must go to $0$ if $T_+(u)$ is finite.

\EMPH{Step 2} We prove that there exists $\mu(s)>0$, $y(s)\in \RR^N$ such that
\begin{equation}
\label{Ap2} 
\left(\mu^{\frac{N}{2}-1}(s)v\left(s,\mu(s)\cdot+y(s)\right), \mu^{\frac N2}(s)\partial_t v\left(s,\mu(s)\cdot+y(s)\right)\right)\in \overline{K}_+.
\end{equation} 
Indeed, let $\lambda_n=\lambda\left(t_n+\lambda(t_n)s\right)$, $x_n=x\left(t_n+\lambda(t_n)s\right)$ and
$$(u_{0,n},u_{1,n})(x)=\left( \lambda_n^{\frac{N}{2}-1}u\left(t_n+\lambda(t_n)s,\lambda_nx+x_n\right), \lambda_n^{\frac N2}\partial_t u\left(t_n+\lambda(t_n)s,\lambda_nx+x_n\right)\right).$$
By the definition of $K_+$ and Step 1, $(u_{0,n},u_{1,n})\in K_+$ for large $n$. By \eqref{conv_to_v} and a standard continuity property of the flow of \eqref{CP}, 
$$ \lim_{n\to\infty} \left\|\lambda(t_n)^{\frac{N}{2}}\nabla_{t,x}u\left(t_n+\lambda(t_n)s,\lambda(t_n)\cdot+x(t_n)\right)-\nabla_{t,x}v(s)\right\|_{(L^2)^{N+1}}=0.$$
Combining, we get that there exists $\mu_n$, $y_n$ such that 
$$ \lim_{n\to\infty} \left\| \left(\frac{1}{\mu_n^{\frac{N}{2}-1}}u_{0,n}\left(\frac{\cdot-y_n}{\mu_n}\right),\frac{1}{\mu_n^{\frac{N}{2}}}u_{1,n}\left(\frac{\cdot-y_n}{\mu_n}\right)\right)-\vv(s)\right\|_{\hdot\times L^2}.$$
Since $(u_0,u_1)\neq (0,0)$, 
$$\inf_{t\in I_{\max}(u)}\|\vu(t)\|_{\hdot\times L^2}>0$$ 
which shows that there exist a constant $C>0$ such that 
$$ \forall n,\quad C^{-1}\leq \mu_n\leq C,\; |y_n|\leq C.$$
Extracting sequences, we obtain that $\{(\mu_n,y_n)\}_n$ has a limit $(\mu(s),y(s))\in (0,+\infty)\times \RR^N$ such that \eqref{Ap2} holds. The proof is complete.
\end{proof}

\bibliographystyle{acm}
\bibliography{toto}

\end{document}